\documentclass[10pt]{amsart}
\usepackage{amsthm,amssymb,amscd,amsmath,amsfonts,amssymb,amscd,hyperref,mathrsfs,bbm} 
\usepackage{stmaryrd}
\usepackage{yhmath}
\usepackage[shortlabels]{enumitem}
\usepackage[all]{xy}
\usepackage{hyperref} 					
\usepackage{tikz}

\theoremstyle{plain}
\makeatletter
\def\cal@symb#1|#2{\expandafter\def\csname #2#1\endcsname{\mathcal{#1}}}
\def\calsymbols#1#2{\@for\@tmpz:=#2\do{\expandafter\cal@symb\@tmpz|{#1}}}
\def\bb@symb#1|#2{\expandafter\def\csname #2#1\endcsname{\mathbb{#1}}}
\def\bbsymbols#1#2{\@for\@tmpz:=#2\do{\expandafter\bb@symb\@tmpz|{#1}}}
\def\bold@symb#1|#2{\expandafter\def\csname #2#1\endcsname{\mathbf{#1}}}
\def\boldsymbols#1#2{\@for\@tmpz:=#2\do{\expandafter\bold@symb\@tmpz|{#1}}}
\def\scr@symb#1|#2{\expandafter\def\csname #2#1\endcsname{\mathscr{#1}}}
\def\scrsymbols#1#2{\@for\@tmpz:=#2\do{\expandafter\scr@symb\@tmpz|{#1}}}
\def\frak@symb#1|#2{\expandafter\def\csname #2#1\endcsname{\mathfrak{#1}}}
\def\fraksymbols#1#2{\@for\@tmpz:=#2\do{\expandafter\frak@symb\@tmpz|{#1}}}

\def\dmth@p#1|{\expandafter\let\csname#1\endcsname\relax
  \expandafter\DeclareMathOperator\csname#1\endcsname{#1}}
\def\operators#1{\@for\@tmpz:=#1\do{\expandafter\dmth@p\@tmpz|}}
\makeatother
\calsymbols{c}{A,B,C,D,E,F,G,H,I,J,K,L,M,N,O,P,Q,R,S,T,U,V,W,X,Y,Z}
\bbsymbols{b}{A,B,C,D,E,F,G,H,I,J,K,L,M,N,O,P,Q,R,S,T,U,V,W,X,Y,Z}
\boldsymbols{bf}{r,q,A,B,C,D,E,F,G,H,I,J,K,L,M,N,O,P,Q,R,S,T,U,V,W,X,Y,Z}
\scrsymbols{s}{A,B,C,D,E,F,G,H,I,J,K,L,M,N,O,P,Q,R,S,T,U,V,W,X,Y,Z}
\fraksymbols{fr}{a,b,c,d,e,f,g,h,i,j,k,l,m,n,o,p,q,r,s,t,u,v,w,x,y,z,F,M,O,R,X,Z}

\theoremstyle{plain}
\makeatletter
\def\cal@symb#1|#2{\expandafter\def\csname #2#1\endcsname{\mathcal{#1}}}
\def\calsymbols#1#2{\@for\@tmpz:=#2\do{\expandafter\cal@symb\@tmpz|{#1}}}
\def\bb@symb#1|#2{\expandafter\def\csname #2#1\endcsname{\mathbb{#1}}}
\def\bbsymbols#1#2{\@for\@tmpz:=#2\do{\expandafter\bb@symb\@tmpz|{#1}}}
\def\bold@symb#1|#2{\expandafter\def\csname #2#1\endcsname{\mathbf{#1}}}
\def\boldsymbols#1#2{\@for\@tmpz:=#2\do{\expandafter\bold@symb\@tmpz|{#1}}}
\def\scr@symb#1|#2{\expandafter\def\csname #2#1\endcsname{\mathscr{#1}}}
\def\scrsymbols#1#2{\@for\@tmpz:=#2\do{\expandafter\scr@symb\@tmpz|{#1}}}
\def\frak@symb#1|#2{\expandafter\def\csname #2#1\endcsname{\mathfrak{#1}}}
\def\fraksymbols#1#2{\@for\@tmpz:=#2\do{\expandafter\frak@symb\@tmpz|{#1}}}

\def\dmth@p#1|{\expandafter\let\csname#1\endcsname\relax
  \expandafter\DeclareMathOperator\csname#1\endcsname{#1}}
\def\operators#1{\@for\@tmpz:=#1\do{\expandafter\dmth@p\@tmpz|}}
\makeatother
\calsymbols{c}{A,B,C,D,E,F,G,H,I,J,K,L,M,N,O,P,Q,R,S,T,U,V,W,X,Y,Z}
\bbsymbols{b}{A,B,C,D,E,F,G,H,I,J,K,L,M,N,O,P,Q,R,S,T,U,V,W,X,Y,Z}
\boldsymbols{bf}{r,q,A,B,C,D,E,F,G,H,I,J,K,L,M,N,O,P,Q,R,S,T,U,V,W,X,Y,Z}
\scrsymbols{s}{A,B,C,D,E,F,G,H,I,J,K,L,M,N,O,P,Q,R,S,T,U,V,W,X,Y,Z}
\fraksymbols{fr}{a,b,c,d,e,f,g,h,i,j,k,l,m,n,o,p,q,r,s,t,u,v,w,x,y,z,F,M,O,R,X,Z}

\theoremstyle{definition}
\newtheorem{defn}{Definition}[subsection]
\newtheorem{thm}[defn]{Theorem}
\newtheorem{cor}[defn]{Corollary}
\newtheorem{prop}[defn]{Proposition}

\newtheorem{lem}[defn]{Lemma}
\newtheorem{ex}[defn]{Example}
\newtheorem{exs}[defn]{Examples}

\newtheorem{rmk}[defn]{Remark}

\newtheorem{notn}[defn]{Notation}

\newtheorem{MainThm}{Theorem}

\calsymbols{c}{A,B,C,D,E,F,G,H,I,J,K,L,M,N,O,P,Q,R,S,T,U,V,W,X,Y,Z}
\bbsymbols{b}{A,B,C,D,E,F,G,H,I,J,K,L,M,N,O,P,Q,R,S,T,U,V,W,X,Y,Z}
\scrsymbols{s}{A,B,C,D,E,L,M,N,P,R,S,T,U,V,X,Y}
\fraksymbols{f}{g,m,h,t,n,b,p}
\bbsymbols{}{A,R,C,Z,N,Q,F,G,H,X,D}

\DeclareSymbolFont{largesymbols}{OMX}{yhex}{m}{n}
\DeclareMathAccent{\wideparen}{\mathord}{largesymbols}{"F3}
\DeclareMathAlphabet{\mathpzc}{OT1}{pzc}{m}{it}

\operators{tr,Tor,Der,gr,Spec,Proj,MaxSpec,ad,Loc,Stab,Aut,Hom,End,im,coker,Lie,Spf,Sp,Ad,aug,ac,tors,Sym,mod,Homeo,op,qc,Res,alg,GL,dlog,ann,Ch,rk,colim,pre,an,sp,Gr,Supp,rig,Cycl,Gal,Irr,PicCon,Pic,Con,coh,Ab,Aff,PreSh,ind,Ban,pr,res,ab,Nrd,id}

\DeclareMathAlphabet{\mathpzc}{OT1}{pzc}{m}{it}

\newcommand{\h}[1]{\widehat{#1}}

\newcommand{\fr}[1]{\mathfrak{{#1}}}

\newcommand{\ts}[1]{\texorpdfstring{$#1$}{}}
\newcommand{\st}{\mid}

\newcommand{\be}{\begin{enumerate}[{(}a{)}]}
\newcommand{\ee}{\end{enumerate}}
  \let\leq=\leqslant
  \let\geq=\geqslant
\newcommand{\qmb}[1]{\quad\mbox{#1}\quad}

\newcommand{\hsp}{\hspace{0.1cm}}

\newcommand{\sq}[1]{\sqrt{|#1^\times|}}

\def\arrowlim#1#2{\mathop{\underset{\scriptstyle #1}{\underset
    {\raisebox{0ex}[0.25ex][-0.5ex]{$#2$}}{\operatorname{lim}}}}}
\newcommand{\invlim}[1][]{\arrowlim{#1}{\longleftarrow}}        
  \let\leq=\leqslant
  \let\geq=\geqslant

\def\Zp{\Z_p}

\begin{document}
\title[Equivariant line bundles with connection]{Equivariant line bundles with connection on the $p$-adic upper half plane}

\author{Konstantin Ardakov}
\address{Mathematical Institute\\University of Oxford\\Oxford OX2 6GG}
\email{ardakov@maths.ox.ac.uk}
\author{Simon Wadsley}
\address{Homerton College\\ Cambridge CB2 8PH}
\email{S.J.Wadsley@dpmms.cam.ac.uk}
\vspace{-2cm}

\begin{abstract}Let $F$ be a finite extension of $\bQ_p$, let $\Omega_F$ be Drinfeld's upper half-plane over $F$ and let $G^0$ the subgroup of $GL_2(F)$ consisting of elements whose determinant has norm $1$. By working locally on $\Omega_F$, we construct and classify the torsion $G^0$-equivariant line bundles with integrable connection on $\Omega$ in terms of the smooth linear characters of the units of the maximal order of the quaternion algebra over $F$. \end{abstract}

\maketitle
\vspace{-1cm}
\tableofcontents

\section{Introduction}
\subsection{Background}
Let $p$ be a prime number, $\Q_p$ the field of $p$-adic numbers and $F$ a finite extension of $\Q_p$ with ring of integers $\cO_F$ and uniformiser $\pi_F$. Let $\Omega_F$ denote the the Drinfeld upper half-plane: it is a rigid $F$-analytic space whose underlying set is the set of $\Gal(\overline{F}/F)$-orbits in $\mathbb{P}^1(\overline{F})\backslash \mathbb{P}^1(F)$. In \cite{DrinfeldCoverings}, Drinfeld defined a tower of $G:=GL_2(F)$-equivariant \'etale rigid $\breve{F}$-analytic coverings of $\Omega := \Omega_F \times_F \breve{F}$, where $\breve{F}$ is the completion of the maximal unramified extension of $F$. 

We consider a slight modification of Drinfeld's construction, due to Rapoport-Zink \cite{RZ96}. They constructed a tower of coverings \[\cdots \to \cM_n\to \cM_{n-1}\to \cdots\to\cM_1\to \cM_0\] of rigid $\breve{F}$-analytic spaces. Their base space $\cM_0$ can be viewed as being the disjoint union of countably many copies of $\Omega$ (cf \cite[Theorem 3.72]{RZ96}): there is a non-canonical isomorphism $\cM_0\cong \Omega \times \bZ$. The tower comes equipped with an action of $G$, so that the maps $\cM_n\to \cM_{n-1}$ are all $G$-equivariant. The $G$-action on the base space is given by \[ g\cdot (z,n)=(g\cdot z, n+v_{\pi_F}(\det g)).\] Here the $G$-action on $\Omega$ is the usual one by M\"obius transformations. Thus each copy of $\Omega$ is stabilised by the subgroup $G^0$ of $G$ consisting of matrices $g$ in $G$ such that $v_{\pi_F}(\det g)=0$. In this way we may view $\cM_0$ as being isomorphic to $G\times_{G^0} \Omega$. 

Let $D$ be the quaternion division algebra over $F$ with valuation ring $\cO_D$ and let $\Pi$ denote a generator of the unique maximal ideal in $\cO_D$. The maps $\cM_n\to \cM_{n-1}$ in the tower are all finite \'etale and Galois with \[\Gal(\cM_n/\cM_0)\cong\cO_D^\times/(1+\Pi^n\cO_D).\] Moreover, the actions of $G$ and $\Gal(\cM_n/\cM_0)$ on $\cM_n$ commute. 

Scholze--Weinstein have proved in \cite[Theorem 7.3.1]{SW13} that if $\bfC$ denotes a complete and algebraically closed extension of $\breve{F}$ and $\widetilde{\Omega}$ is a finite \'etale $G$-equivariant cover of the base-change $\Omega \times_{\breve{F}} \bfC$, then there is an integer $m\geq 0$ such that $\cM_m \times_{\breve{F}} \bfC\to \Omega \times_{\breve{F}} \bfC$ factors through $\widetilde{\Omega}$.  

One way to better understand these finite \'etale equivariant covers is through the study of equivariant vector bundles with flat connections on the base space: the two theories are essentially equivalent. Here we briefly sketch how to obtain equivariant vector bundles with flat connection from the Drinfeld tower. Suppose that $\rho$ is a smooth geometrically irreducible representation of $\cO_D^\times$ whose character is defined over some finite extension $K$ of $\breve{F}$. There is some $m\geq 0$ such that $\rho$ factors through $\cO_D^\times/(1+\Pi^m\cO_D)$ and, for all $n\geq m$, $\cO_{\cM_{n,K}}$ has a $\rho$-isotypical component $\cV^\rho$ independent of the choice of $n$. Then the pushforward of $\cV^\rho$ down to $\cM_{0,K}$ is a $G$-equivariant vector bundle over $\cM_{0,K}$ of rank $(\dim \rho)^2$, equipped with an integrable connection. 

In this paper we only consider the smooth $K$-linear representations of $\cO_D^\times$ of degree $1$, that is, the torsion characters $\theta : \cO_D^\times \to K^\times$. Then the pushforward of $\cV^\theta$ is a $G$-equivariant line bundle on $\cM_{0,K}$ with an integrable connection. For example, when $\bf{1}$ is the trivial representation, the corresponding line bundle is just the structure sheaf $\cO_{\cM_0}$ with its natural $G$-action and natural integrable connection. The set $\Hom(\cO_D^\times,K^\times)_{\tors}$ of torsion characters $\cO_D^\times \to K^\times$ forms a group under pointwise multiplication of characters; given a continuous action\footnote{see $\S \ref{PicConG}$ for the precise definitions} of a topological group $H$ on a smooth rigid $K$-analytic space $X$, we denote by $\PicCon^H(X)$ the group of isomorphism classes of $H$-equivariant line bundles with an integrable connection on $X$, where the group structure comes from the tensor product of these line bundles. In this way, the Drinfeld tower gives us a group homomorphism
\begin{equation}\label{DrinTowerMap} \Hom(\cO_D^\times,K^\times)_{\tors}\quad \longrightarrow\quad\PicCon^G(\cM_{0,K})_{\tors}.\end{equation}
\subsection{The main result} In this paper we will consider, for each complete valued field extension $K$ of $F$, the subgroup $\PicCon^{G^0}(\Omega)_{\tors}$ of $\PicCon^{G^0}(\Omega)$ consisting of the isomorphism classes of the $G^0$-equivariant line bundles with integrable connection on the rigid $K$-analytic space $\Omega := \Omega_F \times_F K$ that have finite order under $\otimes$. This group is naturally isomorphic to $\PicCon^G(G\times_{G_0} \Omega)_{\tors}$ via an induction map, and therefore also, non-canonically, to $\PicCon^G(\cM_{0,K})_{\tors}$ whenever $K$ contains $\breve{F}$. Our main result is then
\begin{MainThm}\label{ThmA} Suppose that $K$ contains the quadratic unramified extension $L$ of $F$. Then there is an isomorphism of abelian groups \[ \PicCon^{G^0}(\Omega)_{\tors} \quad \stackrel{\cong}{\longrightarrow}\quad \Hom(\cO_D^\times,K^\times)_{\tors}.\]
\end{MainThm}

We note that Theorem \ref{ThmA} cannot be true without the condition that $K$ contains $L$ in the statement, because $\PicCon^{G^0}(\Omega)[p']$ is  cyclic of order $q^2-1$ for any choice of field extension $K$, but $K^\times$ contains no such subgroup unless $K$ contains $L$ and so $\Hom(\cO_D^\times, K^\times)$ cannot either. We also note here in passing that in fact all line bundles on $\Omega$ are known to be trivial --- see, e.g. \cite[Th\'{e}or\`{e}me A]{Junger23}.

Our isomorphism depends on a choice of a point $z \in \Omega_F(L)$ and an $F$-algebra embedding $\iota : L \hookrightarrow D$. However, there is a natural $G$-action on $\PicCon^{G^0}(\Omega)_{\tors}$ and a natural conjugation $D^\times$-action on $\Hom(\cO_D^\times, K^\times)_{\tors}$: the first action factors through $G / G^0 F^\times$ and the second action factors through $D^\times / \cO_D^\times F^\times$, both cyclic groups of order $2$. The isomorphisms in Theorem \ref{ThmA} are compatible with respect to these actions, provided we identify $G / G^0 F^\times$ with $D^\times / \cO_D^\times F^\times$ in the only possible way. Quotienting out by these actions on both sides, we obtain a bijection
\[ \PicCon^{G^0}(\Omega)_{\tors}/G \quad \stackrel{\cong}{\longrightarrow}\quad\Hom(\cO_D^\times,K^\times)_{\tors}/D^\times\] 
which no longer depends on any choices. We expect, but do not check it in this paper, that when $K=\breve{F}$, the isomorphism in Theorem \ref{ThmA} is in fact inverse to (\ref{DrinTowerMap}) for some choice of identification of $\cM_0$ with $G\times_{G^0}\Omega$, depending on our other choices.

\subsection{Motivation} We do not use the result of Scholze--Weinstein \cite{SW13}, nor even the existence of the Drinfeld tower, to prove Theorem \ref{ThmA}. Instead, we give an explicit construction of the elements $\PicCon^{G^0}(\Omega)_{\tors}$ as finite order $G^0$-equivariant line bundles with connection on $\Omega$ and prove directly that there are no others. More precisely: in this paper we give an \emph{elementary construction} \footnote{elementary in the sense of not depending on the duality of Rapoport-Zink spaces, nor on the theory of perfectoid spaces} of finite order $GL_2(\cO_F)$-equivariant line bundles with connection on the $K$-affinoid subdomain $\Omega_0$ of $\Omega$ that is the inverse image of the vertex fixed by $GL_2(\cO_F)$ in the Bruhat-Tits tree under the reduction map, and show that each of these line bundles extends uniquely to a $G^0$-equivariant line bundle with connection on all of $\Omega$. We have provided full arguments for some results that already appear in the literature in order to stress the elementary nature of our work. 

In our companion work \cite{AWgoat2}, we use this elementary construction to better understand the structure of the global sections $\sL^\theta(\Omega)$ of these equivariant line bundles as modules over the locally $F$-analytic distribution algebra of $G^0$. The results of this paper are used to show that if $j : \Omega \to \bP^1_K$ denotes the open inclusion, then $j_\ast \sL^\theta$ is a coadmissible $G^0$-equivariant $\cD$-module on $\bP^1_K$ in the sense of \cite{EqDCap}. We then use this to prove that for each $\theta\in \Hom(\cO_D^\times, K^\times)_{\tors}$, $\sL^\theta(\Omega)$ is \emph{irreducible} as a coadmissible $D(G^0,K)$-module when the underlying connection is non-trivial, and that it has length two  otherwise. The condition that the underlying connection is non-trivial can be shown to be equivalent to $\theta$ not being fixed by the $D^\times$-action on $\Hom(\cO_D^\times,K^\times)_{\tors}$ discussed above.

\subsection{An overview of some related works} Because of the central position of the theory of Drinfeld coverings and Rapoport-Zink spaces in the local Langlands program \cite{HarrisTaylor}, it is difficult to make a comprehensive literature review. In his work \cite{Teit1990}, Teitelbaum found explicit local equations for the first Drinfeld covering of $\Omega$: this yields explicit torsion line bundles with connection on $\Omega_0$ together with a (non-explicit) $GL_2(\cO_F)$-equivariant structure. However neither Teitelbaum, nor Lue Pan in his closely-related work \cite{Pan2017}, consider the flat connections on these line bundles, nor do they make an attempt to classify the appropriate equivariant structures in an elementary manner.

We acknowledge our intellectual debt to the Introduction of \cite{DospLeBras}, where Dospinescu and Le Bras explain the importance of the equivariant vector bundles on $\Omega$ in the context of locally analytic representations of $GL_2(F)$, and in particular in the $p$-adic local Langlands program. We mention here in passing that the constructions in our paper enable us to define the first Drinfeld covering of $\Omega$ directly over $F$ rather than over $\breve{F}$ without appealing to the theory of Weil descent. We also take the opportunity to mention here the monumental works of Colmez, Dospinescu and Nizio\l{} \cite{CDN20a}, \cite{CDN20b} that use Scholze's pro-\'etale methods and build upon \cite{DospLeBras} to show, amongst other things, that the $p$-adic \'etale cohomology of the Drinfeld tower realises the $p$-adic local Langlands correspondence, at least in the case when $F = \bQ_p$.

Finally, we mention that Junger in his recent preprint \cite{Jun23pre} has classified equivariant line bundles on Deligne's formal model $\widehat{\Omega}$ of $\Omega_F$. There are some formal similarities in our methods (for example his Proposition 2.10 plays a similar role to our Proposition \ref{PicSeq}), however we cannot deduce his results from ours, nor vice versa. We expect that our results here can be naturally extended to analogues of $\Omega$ in higher dimensions.

\subsection{Acknowledgements} \label{AckSec} We thank Aurel Page for pointing the paper of Riehm \cite{Riehm} on \href{https://mathoverflow.net/questions/433274/abelianization-of-unit-quaternions-over-a-p-adic-field}{MathOverflow}. We also thank James Taylor for his comments. The second author thanks Brasenose College, Oxford for its hospitality.
\subsection{Conventions and Notation}  \label{ConvNotn}
$F$ will denote a finite extension of $\mathbb{Q}_p$ with ring of integers $\cO_F$, uniformiser $\pi_F$ and residue field $k_F$ of order $q$. $K$ will denote a field containing $F$ that is complete with respect to a non-archimedean norm $|\cdot|$ such that $|p|=1/p$. We will write: 
\begin{itemize}
\item $K^\circ := \{a \in K : |a| \leq 1\}$ for the valuation ring of $K$,
\item $K^{\circ\circ} := \{a \in K : |a| < 1\}$ for the maximal ideal of $K^\circ$,
\item $\overline{K}$ for a fixed algebraic closure of $K$, and
\item $\bfC$ for the completion of $\overline{K}$.
\end{itemize}
Let $X$ be a rigid $K$-analytic variety. When $Y$ is a subset $X$, we will say that $Y$ is an \emph{affinoid subdomain} of $X$ to mean that $Y$ is an admissible open subspace of $X$, itself isomorphic to an affinoid $K$-variety. When $X$ itself happens to be a $K$-affinoid variety, this agrees by \cite[Corollary 8.2.1/4]{BGR} with the standard definition found at \cite[Definition 7.2.2/2]{BGR}. We will write
\begin{itemize}
\item $|\cdot|_X$ to denote the (power-multiplicative) supremum seminorm on $X$,
\item $\cO(X)^\circ := \{f \in \cO(X) : |f|_X \leq 1\}$,
\item $\cO(X)^{\circ\circ} := \{f \in \cO(X) : |f|_X < 1\}$,
\item $\cO(X)^{\times\times} := 1 + \cO(X)^{\times\times}$, the subgroup of \emph{small units} in $\cO(X)^\times$, and
\item $\cO(X)^{\times\times}_r := \{1 + f : |f|_X \leq r\} \leq \cO(X)^{\times\times}$ for each real number $r \in (0,1)$.
\end{itemize}
$\Pic(X)$ will denote the \emph{Picard group} of $X$ consisting of the isomorphism classes of line bundles on $X$ with the group operation given by tensor product. 

When $K'$ is a complete field extension of $K$, we will write $X_{K'} := X\times_K K'$ for the base change of $X$ to $K'$, and \[ X(K')=\{\phi\colon \Sp K'\to X: \phi \mbox{ is a morphism of rigid }K\mbox{-analytic varieties} \}\] 
will denote the set of {\emph{$K'$-valued points of $X$}}. 

Let $A$ be an abelian group and let $d$ be a non-zero integer. We will write
\begin{itemize} 
\item $A[d]=\{a\in A\st da=0\}$ to denote the $d$-torsion subgroup of $A$, 
\item $A[p']:= \bigcup_{(d,p)=1} A[d]$ to denote the prime-to-$p$ torsion subgroup of $A$,
\item $A[p^\infty]:=\bigcup_{n=1}^\infty A[p^n]$ to denote the $p$-power torsion subgroup of $A$. 
\item $A_{\tors}:=\bigcup_{d\geq 1}A[d]$ to denote the full torsion subgroup of $A$.
\end{itemize} 

Let a group $G$ act on a set $X$. We will write
\begin{itemize}
\item $G_x:=\{g\in G:gx=x\}$ for the stabilizer of a point $x\in X$, and
\item $X^G:=\{x\in X: gx=x $ for all $g\in G\}$ for the set of elements fixed by $G$.
\end{itemize}

When we discuss cochains, cocycles and coboundaries we will work with the continuous cochain cohomology of Tate, \cite{TateK2}. That is if $G$ is a topological group and $A$ is a topological abelian group equipped with a continuous action of $G$ then:
\begin{itemize} \item $C^n(G,A):=\{f\colon G^n\to A: f$ is continuous$\}$ is the set of continuous $A$-valued $n$-cochains, 
\item $Z^n(G,A)$ is the set of continuous $A$-valued $n$-cocycles,
\item $B^n(G,A)$ is the set of continuous $A$-valued $n$-coboundaries, 
\item $H^n(G,A)$ is the $n$th continuous cohomology group $Z^n(G,A)/B^n(G,A)$, 
\item $\delta_G\colon C^0(G,A)=A\to C^1(G,A)$ is the map given by $\delta_G(a)(g)=g\cdot a-a$.
\end{itemize}
Note that whenever $\theta\colon A\to B$ is a $G$-equivariant map of topological abelian groups with continuous $G$-action, we have $\delta_G\theta=\theta\delta_G$. 

In the particular case where $G$ acts trivially on $A$, we will usually write $\Hom(G,A)$ instead of $Z^1(G,A)$ or $H^1(G,A)$ to denote the group of continuous homomorphisms from $G$ to $A$.  For any continuous group homomorphism $\varphi \colon G \to H$, we denote by \[\varphi^\ast : \Hom(H, A) \to \Hom(G,A)\] the map given by pre-composition by $\varphi$.

\section{Background from algebra}

\subsection{Measures on profinite sets}\label{Measures}
We begin by adapting some definitions from \cite[Definitions 2.7.10]{FvdPut}.

\begin{defn} Let $Z$ be a profinite set and let $\fr{a}$ be an abelian group. 
	\be \item An \emph{$\fr{a}$-valued measure on Z} is a function $\nu$ from the set of clopen subsets of $Z$ to $\fr{a}$, satisfying $\nu(U) = \nu(V) + \nu(U \backslash V)$ whenever $V \subseteq U$ are clopen subsets of $Z$. 
	\item $M(Z, \fr{a})$ denotes the abelian group of all $\fr{a}$-valued measures on $Z$ under pointwise operations.
	\item $M_0(Z,\fr{a}) := \{\nu \in M(Z, \fr{a}) : \nu(Z)=0\}$ is the subgroup of \emph{$\fr{a}$-valued measures on $Z$ with total value zero}.
	\ee\end{defn}

\begin{exs} \label{deltameas} \hfill \begin{enumerate} \item If $Z$ is any profinite set and $\mathfrak{a}$ is a unital ring then for each $z\in Z$ we can define a measure $\delta_z$ by \[\delta_z(U)=\begin{cases} 1 & \mbox{ if }z\in U \\ 
			0 & \mbox{ otherwise.} \end{cases} \]
		\item If $Z$ happens to be finite we can define a `counting measure' in $M(Z,\bZ)$ via \[ \Sigma_Z(U)=|U| \mbox{ for all  }U\subset Z.\] Indeed $\Sigma_Z=\sum_{z\in Z}\delta_z$. \ee
	\end{exs}
	
	Let $C(Z,\bZ)$ be the set of continuous $\bZ$-valued functions on $Z$, where we give $\bZ$ the discrete topology. Of course every such function is locally constant.
	
	\begin{prop}\label{ExactMeasures} Let $Z$ be a profinite set and let $\fr{a}$ be an abelian group. 
		\be \item There is a natural additive isomorphism $M(Z, \fr{a}) \to \Hom_{\bZ}(C(Z, \bZ), \fr{a})$.
		\item Let $(Z_i)_{i \in I}$ be a filtered inverse system of finite sets. Then every isomorphism $Z \cong \invlim Z_i$ of profinite sets induces an isomorphisms of abelian groups $M(Z, \fr{a}) \cong \invlim M(Z_i, \fr{a})$ and $M_0(Z,\fr{a})\cong \invlim M_0(Z_i,\fr{a})$.
		\item The functor $\fr{a} \mapsto M(Z,\fr{a})$ is exact.
		\ee
	\end{prop}
	\begin{proof} (a) Let $\nu \in M(Z,\fr{a})$ and let $f : Z \to \bZ$ be locally constant. Because $Z$ is profinite and hence compact, we can choose an open partition $\{U_1,\cdots,U_m\}$ of $Z$ such that $f_{|U_i}$ is constant for each $i$. Choose $z_i \in U_i$ for each $i = 1,\cdots, m$ and define $\langle \nu, f \rangle := \sum\limits_{i=1}^m f(z_i) \nu(U_i) \in \fr{a}$. Then $\langle \nu, f\rangle$ does not depend on the choice of open partition or the choice of the $z_i$'s, and $\nu \mapsto (f \mapsto \langle \nu, f\rangle)$ defines an additive map $M(Z,\fr{a}) \to \Hom_{\bZ}(C(Z,\bZ), \fr{a})$. 
		
		Let $1_U$ denote the characteristic function of the clopen subset $U$ of $Z$. Given an additive map $g : C(Z,\bZ) \to \fr{a}$, setting $\nu(U) := g(1_U) \in \fr{a}$ for each clopen $U$ defines an element $\nu \in M(Z,\fr{a})$ such that $\langle \nu, f \rangle = g(f)$ for all $f \in C(Z, \bZ)$ because the characteristic functions $1_U$ generate $C(Z,\bZ)$ as an abelian group. The function $g \mapsto \nu$ is then an inverse to $M(Z,\fr{a}) \to \Hom_{\bZ}(C(Z,\bZ), \fr{a})$.
		
		(b) The isomorphism $Z \cong \invlim Z_i$ induces an isomorphism of abelian groups $C(Z,\bZ) \cong \lim\limits_{\longrightarrow} C(Z_i,\bZ)$. The functor $\Hom_{\bZ}(-,\fr{a})$ converts colimits into limits; now apply part (a) to get $M(Z,\fr{a})\cong \invlim M(Z_i,\fr{a})$. Since taking limits commutes with taking kernels the other part follows immediately. 
		
		(c) By a theorem of N\"obeling --- see \cite[Theorem 5.4]{SchCond} --- $C(Z,\bZ)$ is a free abelian group. Hence $\Hom_{\bZ}(C(Z,\bZ), -)$ is exact; now apply part (a).
	\end{proof}
	
	\begin{defn}\label{Mfunct} For every abelian group $\fr{a}$ the map $Z\mapsto M(Z,\fr{a})$ from profinite sets to abelian groups extends to a functor from the category of profinite sets and continuous functions to the category of abelian groups sending a continuous function $f\colon Z_1\to Z_2$ to the group homomorphism \[f_\ast\colon M(Z_1,\fr{a})\to M(Z_2,\fr{a});\hspace{0.2cm} f_\ast(\nu)(U)=\nu(f^{-1}(U)).\] 
		
		In particular an action of group $G$ on a profinite set $Z$ by homeomorphisms induces an action on $M(Z,\fr{a})$ by automorphisms of abelian groups via $g\cdot\nu=g_\ast(\nu)$.  
	\end{defn}

\begin{rmk} If $\nu\in M_0(Z_1,\fr{a})$ in the setting of Definition \ref{Mfunct} then $f_\ast(\nu)\in M_0(Z_2,\fr{a})$ since $f_\ast(\nu)(Z_2)=\nu(f^{-1}(Z_2))=\nu(Z_1)=0$ so $Z\mapsto M_0(Z,\fr{a})$ also defines a functor.  
\end{rmk}	
	The following result will be useful in what follows.
	
	\begin{lem}\label{M0modD} Let $Z$ be a profinite set and let $d$ be a non-zero integer. Then the sequence $0 \to M_0(Z, \bZ) \stackrel{d}{\longrightarrow} M_0(Z,\bZ) \to M_0(Z, \bZ/d \bZ) \to 0$ is exact.
	\end{lem}
	\begin{proof} Consider the multiplication-by-$d$ map on the short exact sequence of abelian groups $0 \to M_0(Z, \bZ) \to M(Z,\bZ) \to \bZ \to 0$. This gives a $3 \times 3$ diagram of abelian groups, whose rows are exact, whose second column is exact by Proposition \ref{ExactMeasures}(c) and whose third column is exact for trivial reasons. Hence the first column is also exact by the Nine Lemma. \end{proof}
Recall our conventions concerning continuous group cohomology from $\S \ref{ConvNotn}$.
	\begin{lem}\label{HiM0} Let $G$ be a profinite group with a continuous transitive action on a finite set $X$. Then \be
		\item $M(X,\bZ)^G=\mathbb{Z}\cdot\Sigma_X$;
		\item	$M_0(X,\bZ)^G=0$;
		\item $H^1(G,M(X,\bZ))=0$.
		 \ee
	\end{lem}
	
	\begin{proof}
		(a) It is easy to compute that $M(X,\bZ)^G=\bZ\cdot \Sigma_X$ where $\Sigma_X$ is the counting measure from Example \ref{deltameas}(2). 
		
(b) Use $M_0(X,\bZ)^G=\ker \left(M(X,\bZ)^G\to \bZ; \nu\mapsto \nu(X)\right)$ and $n\Sigma_X(X)=n|X|$.
	
		(c) Choose an arbitrary point $x\in X$. Since $X$ is finite, $M(X,\bZ)$ is isomorphic to the (co)induced module $\mathrm{Ind}^{G_{x}}_G \bZ$ in the sense of \cite[Chapter I, \S6]{NSW}, where $\bZ$ denotes the trivial $G_{x}$-module. Then by Shapiro's Lemma, \cite[Proposition 1.6.4]{NSW}, we have \[H^1(G,M(X,\bZ))\cong H^1(G_{x},\bZ).\]
		
	Since the action of $G_{x}$ on $\bZ$ is trivial, $G_{x}$ is profinite and $\bZ$ is a discrete torsion-free group we deduce that $H^1(G_{x},\bZ)=\Hom(G_{x},\bZ)=0$.
	\end{proof}
	
	\begin{prop} \label{M0Gfinite} Let $G$ be a profinite group acting continuously and transitively on two non-empty finite sets $X,Y$, and let $\pi\colon X\to Y$ be a $G$-equivariant function. Let $d\geq 1$ be an integer and let $h:=\gcd(d, |X|)$.  Then
		\be \item $M(X,\bZ/d\bZ)^G$ is cyclic of order $d$, generated by the image of $\Sigma_X$,
		\item $M_0(X,\bZ/d\bZ)^G$ is cyclic of order $h$, generated by the image of $\frac{d}{h}\Sigma_X$, and
		\item $\pi_\ast \left(\Sigma_X\right)=\frac{|X|}{|Y|}\Sigma_Y$.  \ee
	\end{prop}
	
	\begin{proof}
		(a) The short exact sequence \[ 0\to M(X,\bZ) \stackrel{d}{\to} M(X,\bZ)\to M(X,\bZ/d\bZ)\to 0\] induces a long exact sequence of cohomology \[ 0\to M(X,\bZ)^G\stackrel{d}{\to} M(X,\bZ)^G\to M(X,\bZ/d\bZ)^G{\to} H^1(G,M(X,\bZ))\] whose final term is $0$ by Lemma \ref{HiM0}(c). Since $M(X,\bZ)^G=\bZ\cdot \Sigma_X$ by Lemma \ref{HiM0}(a) the result follows.
		
		(b) Note that $M_0(X,\bZ/d\bZ)^G = \{\nu \in M(X,\bZ/d\bZ)^G : \nu(X)=d\bZ\}$. By part (a), every such $\nu$ is the image of $n\Sigma_X$ for some $n$. But $n\Sigma_X(X)=n|X|$ so \[M_0(X,\bZ/d\bZ)^G=\{n\Sigma_X\in M(X,\bZ/d\bZ):d\mbox{ divides }n|X|\}\] giving the result. 
		
	
		(c) Whenever $U\subseteq Y$, we have $\pi_\ast(\Sigma_X)(U)=\Sigma_X(\pi^{-1}(U))=\left|\pi^{-1}(U)\right|$. Since $G$ acts transitively on $X$ and $Y$, all the fibres have order $\frac{|X|}{|Y|}$ and so \[ \left|\pi^{-1}(U) \right|=\frac{|X|}{|Y|}|U|\] as required. 
	\end{proof}

\subsection{Some stabilisers in \ts{G^0} and their linear characters} 

\begin{notn} \label{IwahoriDefn}\hsp 
\be \item $G^0:=\{ g \in GL_2(F) : v_{\pi_F}(\det g)=1\}$.
\item The \emph{Iwahori subgroup} $I$ of $GL_2(\cO_F)$ is
\[ I := \left\{ \begin{pmatrix} a & b \\ c & d\end{pmatrix}\in GL_2(\cO_F) : c\equiv 0 \mod \pi_F\cO_F\right\}.\]
\item We write $w:=\begin{pmatrix} 0 & 1 \\ \pi_F & 0\end{pmatrix}\in GL_2(F)$.
\item We write $A := GL_2(\cO_F)$ and $B := {}^w A = w A w^{-1}$.
\ee
\end{notn}

We recall from \cite[\S\S II.1.2-3]{Trees} that if $\cB\cT$ is the Bruhat--Tits tree associated with $PGL_2(F)$, then $A,B$ and $I$ arise as stabilisers in $G^0$ under the natural $G^0$-action on $\cB\cT$ as follows: there is a vertex $s$ of $\cB\cT$ such that $(s$ $ws)$ is an edge of $\cB\cT$, and
\[A=G^0_s, \quad B=G^0_{ws}, \qmb{and} I=G^0_{(s \mbox{ }ws)}.\] 
In particular, we have $I = A \cap B$.  The following classical result will be crucial to our arguments later on in $\S \ref{MainProofSec}$.

\begin{thm}\label{AmalgTrees} $G^0$ is the amalgamated product of its open subgroups $A$ and $B$ over their intersection $I$: \[ G^0=A\underset{I}{\ast}{} B.\]\end{thm}
\begin{proof} This is \cite[Theorem II.3]{Trees}. \end{proof}

 Recall that $F^\times$ is the direct product of its subgroups $\mu_{p'}(F)$, $\cO_F^{\times\times}$ and $\langle \pi_F\rangle$: \[ F^\times \cong \mu_{p'}(F)\times \cO_F^{\times\times}\times \langle \pi_F\rangle. \] We write $a\mapsto \widehat{a}$ to denote the homomorphism that is the projection $F^\times\to \mu_{p'}(F)$ with kernel $\cO_F^{\times\times}\times \langle \pi_F\rangle$. We will use the same notation $\hat{\cdot}$ to denote the analogous projection $L^\times\to \mu_{p'}(L)$ for other finite extensions $L/\bQ_p$. 
\begin{lem}\hsp 
\label{Hommup'} \hfill \be \item $\Hom(A,K^\times)[p'] =\left\{\widehat{\det}^k: k\in \bZ/(q-1)\bZ\right\}$.
	\item Every element of $\Hom(I,K^\times)[p']$ is of the form \[ \begin{pmatrix} a & b \\ \pi_F c & d\end{pmatrix} \mapsto \widehat{a}^{n_1}\widehat{d}^{n_2} \] for some $n_1,n_2\in \bZ/(q-1)\bZ$.
	\item $\left(\Hom(I,K^\times)[p']\right)^{\langle w\rangle}=\left\{\widehat{\det}^k: k\in \bZ/(q-1)\bZ\right\}$.
\ee
\end{lem} 
\begin{proof} 
	(a) For any $\theta\in \Hom(A, K^\times)[p']$, every pro-$p$ subgroup of $A$ is contained in $\ker \theta$. Since $SL_2(\cO_F)$ is generated by its pro-$p$ subgroups it follows that $\theta$ factors through $\det\colon A\to \cO_F^\times$. Since $\cO_F^{\times\times}$ is also pro-$p$ we see that any element of $\Hom(\cO_F^\times, K^\times)[p']$ factors through $\widehat{-}$. It remains to observe that $\Hom(\mu_{p'}(F), K^\times)$ consists of maps of the form $a\mapsto a^k$ for some $k\in \bZ/(q-1)\bZ$, because $\mu_{p'}(F)$ is cyclic of order $q-1$.
	
	(b) Once again, if $\theta\in \Hom(I, K^\times)[p']$ then any pro-$p$ subgroup of $I$ is contained in $\ker \theta$. The kernel of the map $I\to \mu_{p'}(F) \times \mu_{p'}(F)$ sending $\begin{pmatrix} a & b \\ \pi_F c & d\end{pmatrix}$ to $(\widehat{a},\widehat{c})$ is pro-$p$. Now the result follows as in (a).
 
 	(c) By part (b), every $\chi\in \Hom(I,K^\times)[p']$ is of the form  \[ \begin{pmatrix} a & b \\ \pi_F c & d\end{pmatrix} \mapsto \widehat{a}^{n_1}\widehat{d}^{n_2} \] for some $n_1,n_2\in \bZ/(q-1)\bZ$. Now \[ (w\cdot \chi)\left(\begin{pmatrix} a & b \\ \pi_F c & d\end{pmatrix}\right)= \chi\left(\begin{pmatrix} d & c \\ \pi_F b & a \end{pmatrix}\right) = \widehat{a}^{n_2}\widehat{d}^{n_1}. \] Thus if $\chi=w\cdot\chi$ then $n_1=n_2$ and $\chi=\widehat{\det}^{n_1}$. 
\end{proof}
\begin{rmk} \label{p'exp} It follows from the proof of Lemma \ref{Hommup'} that any $p'$-quotient of $A, B$ or $I$ has exponent dividing $q-1$.\end{rmk}
We will now recall the structure of the stabilisers in $GL_2(F)$ of points in $\bP^1(L)$ for quadratic extensions $L$ of $F$, under the $GL_2(F)$-action by M\"{o}bius transformations.
\begin{lem}\label{stabz} Suppose that $F(z)$ is a quadratic extension of $F$. Let \[N,\tr\colon F(z)\to F\]  be the norm and trace maps respectively.  \be \item The stabiliser $GL_2(F)_{z}$ of $z$ in $GL_2(F)$ is 
\[ GL_2(F)_{z}=\left\{\begin{pmatrix} a & -cN(z) \\ c & a-c\tr(z) \end{pmatrix}\,\middle\vert (a,c) \in F^2\backslash\{(0,0)\} \right\}. \]

	 \item There is a commutative diagram of groups\[ \xymatrix{GL_2(F)_z  \ar[r]^{j_z} \ar[rd]_\det  & F(z)^\times \ar[d]^N \\ & F^\times } \] whose horizontal arrow  \[j_z\colon \begin{pmatrix} a & -cN(z) \\ c & a-c\tr(z) \end{pmatrix}\mapsto a-cz \] is an isomorphism of topological abelian groups. 
	 \item $j_z(G^0_z)=\cO_{F(z)}^\times$ and $j_z(SL_2(F)_{z})=\ker N \cap \cO_{F(z)}^\times$.
	 \item Let $\sigma \in \Gal(F(z)/F)$ be the unique element of order two. Then 
	 \[\sigma\cdot g= \det (g)g^{-1} \qmb{for all} g\in GL_2(F)_z\] 
defines an continuous action of $\Gal(F(z)/F)$ on $GL_2(F)_{z}$ such that 
\begin{enumerate}[{(}i{)}]
\item $j_z$ is $\Gal(F(z)/F)$-equivariant, 
\item $G^0_z$, $SL_2(F)_z$ and the Sylow pro-$p$-subgroup $P_z$ of $SL_2(F)_z$ are all $\Gal(F(z)/F)$-stable. 
\end{enumerate} \ee
 \end{lem}
\begin{proof}
(a)	We can compute that $\frac{az+b}{cz+d}=z$ if and only if $cz^2+(d-a)z-b=0$. Moreover the minimal polynomial of $z$ is $t^2-\tr(z)t+N(z)$. So $\begin{pmatrix} a & b \\ c & d\end{pmatrix}\in GL_2(F)$ fixes $z$ if and only if $(d-a)=-c\tr(z)$ and $b=-cN(z)$ as claimed.
	
(b)	Since \[\det\begin{pmatrix} a & -cN(z) \\ c & a-c\tr(z) \end{pmatrix}=a^2-ac\tr(z)+c^2N(z)=N(a-cz) \]  the given diagram commutes. 
	For any $(a_1,c_1),(a_2,c_2)\in F^2\backslash \{(0,0\}$,  we have \[ \begin{pmatrix} a_1 & -c_1N(z) \\ c_1 & a_1-c_1\tr(z) \end{pmatrix}\begin{pmatrix} a_2 & -c_2N(z) \\ c_2 & a_2-c_2\tr(z) \end{pmatrix}= \begin{pmatrix} a_1a_2-c_1c_2N(z) & * \\  \\ c_1a_2+a_1c_2-c_1c_2\tr(z) &  * \end{pmatrix}\in GL_2(F)_z\]and 
	
	\begin{eqnarray*} (a_1-c_1z)(a_2-c_2z)&  = &(a_1a_2-(a_1c_2+c_1a_2)z +  c_1c_2z^2)\\ & = & (a_1a_2-c_1c_2N(z)) - (c_1a_2+a_1c_2-c_1c_2\tr(z))z\in F(z)^\times. \end{eqnarray*}
 This implies that $j_z$ is a group isomorphism.
	 
	 (c) Since $G^0_z=\ker \left(v_{\pi_F}\circ \det\colon GL_2(F)_z\to \bZ\right)$, part (b) implies that \[ j_z(G^0_z)=\ker\left( v_{\pi_F}\circ N\colon F(z)^\times \to\bZ\right)=\cO_{F(z)}^\times.\] Similarly since $SL_2(F)_z= \ker \left( \det\colon GL_2(F)_z\to F^\times\right)$, we see that $j_z(SL_2(F)_z)=\ker \left(N\colon F(z)^\times \to F^\times\right)$, which is contained in $\cO^\times_{F(z)}$.
	
(d) We have $N(\lambda)=\lambda\sigma(\lambda)$ for any $\lambda\in F(z)^\times$. Now it follows from (b) that \[j_z^{-1}(\sigma(\lambda))=\det(j_z^{-1}(\lambda))j_z^{-1}(\lambda)^{-1} \qmb{for any} \lambda \in F(z)^\times\] and we can define the claimed action by transport of structure. Since $SL_2(F)_z$ is $\sigma$-stable and $P_z$ is a characteristic subgroup of $SL_2(F)_z$, the rest follows.
 \end{proof}

\begin{cor} \label{G0zSL2} Let $L = F(z)$ be a quadratic unramified extension of $F$. Then
\[ G^0_z \cdot SL_2(F) = G^0.\]
\end{cor}
\begin{proof} Let $g \in G^0$ so that $\det(g) \in \cO_F^\times$. Since the residue field $k_F$ is finite, the norm map $N : k_L^\times \to k_F^\times$ on residue fields is surjective by \cite[p. 82, Remark(1)]{SerreLF}. Since $L/F$ is unramified, the norm map $N : \cO_{L}^\times \to \cO_F^\times$ is surjective by \cite[p. 82, Corollary]{SerreLF}, so we can find $x \in \cO_L^\times$ such that $N(x) = \det(g)$. Using Lemma \ref{stabz}(c), we can choose $h \in G^0_z$ such that $j_z(h) = x$. Then Lemma \ref{stabz}(b) implies that $\det(h) =  N(j_z(h)) = N(x)=\det(g)$, so $g = h \cdot h^{-1}g \in G^0_z \cdot SL_2(F)$ as required.
\end{proof}

	 \begin{rmk}\label{jzcan} Note that, with the notation of Lemma \ref{stabz}, $GL_2(F)_z=GL_2(F)_{\sigma\cdot z}$ 
	 	but $j_z\neq j_{\sigma\cdot z}$. However $j_z^\ast$ does induce a canonical bijection \[ \Hom(G^0_z,K^\times)/\Gal(F(z)/F)\to \Hom(\cO_{F(z)}^\times,K^\times)/\Gal(F(z)/F). \]
	 	\end{rmk}
When $F(z)$ a quadratic extension of $F$, we define a continuous homomorphism $\widehat{j_z}\colon G^0_z\to F(z)^\times$ by setting
\[ \widehat{j_z}(g):= \widehat{j_z(g)} \qmb{for all} g \in G^0_z.\]
\begin{lem}\label{HomG0z} Suppose that $F(z)$ is a quadratic extension of $F$. Then  \[\Hom(G^0_z,K(z)^\times)[p']=\Hom(G^0_z,\mu_{p'}(F(z)))= \left\langle \widehat{j_z} \right\rangle\]
is a cyclic group. Its order is $q^2-1$ if $F(z)/F$ is unramified, and $q-1$ otherwise. \end{lem}
\begin{proof}
Since $j_z\colon G^0_z\to \cO_{F(z)}^\times$ is an isomorphism of topological groups by Lemma \ref{stabz}(c), it suffices to show that every element of $\Hom(\cO_{F(z)}^\times, K(z)^\times)[p']$ is of the form $a\mapsto \hat{a}^k$ for some $k\in \bZ$. Since the kernel of $\hat{\cdot}\colon \cO_{F(z)}^\times\to \mu_{p'}(F(z))$ is pro-$p$, any $\theta\in \Hom(\cO_{F(z)}^\times, K(z)^\times)[p']$ factors through $\hat{\cdot}$. Since $\mu_{p'}(F(z))$ is cyclic of order $q^2-1$ if $F(z)/F$ is unramified and cyclic of order $q-1$ if $F(z)$ is ramified the result is now straightforward.
\end{proof}

\subsection{Quaternions}\label{sec:Quat}
We use \cite[\S1.4]{PlatRap94} and \cite{Pierce} as basic references for this material. By \cite[Proposition 12.5b, Theorem 17.10]{Pierce} there is a central division algebra $D$ over $F$ of dimension $4$ which is unique up to isomorphism. The norm on $F$ extends uniquely to a norm $|\cdot |_D$ on the division algebra by \cite[Proposition 17.6]{Pierce}.

\begin{notn} Let $\cO_D := \{ d \in \cO_D : |d|_D \leq 1\}$ be the maximal order in $D$, $\cP_D := \{ d \in \cO_D : |d|_D < 1\}$ its unique maximal ideal and let $k_D := \cO_D/\cP_D$. We will write $\omega_D : \cO_D^\times\to k_D^\times$ to denote the reduction map modulo $\cP_D$ on multiplicative groups.  Let $L$ denote the unramified quadratic field extension of $F$; the reduction map $\omega_L : \cO_L^\times \to k_L^\times$ on multiplicative groups is defined similarly. \end{notn} 

We recall the Noether--Skolem Theorem.

\begin{thm}[Noether--Skolem]Suppose that $B$ is a central simple algebra over $F$ and that $A$ is a simple subalgebra of $B$. If $\chi\colon A\to B$ is an algebra homomorphism then there is $b\in B^\times$ such that $\chi(x)=bxb^{-1}$ for all $x\in A$.\end{thm}
	\begin{proof}
		See \cite[Noether--Skolem Theorem 12.6]{Pierce}. 
	\end{proof}
\begin{cor}\label{CorToSKN} $D^\times$ acts transitively on the set of $F$-algebra homomorphisms $\iota\colon L\to D$ via conjugation: \[ (d \cdot \iota)(x)=d\hsp\iota(x)\hsp d^{-1} \qmb{for all} d \in D^\times, x \in L. \]\end{cor}
Recall the reduced norm map $\Nrd\colon D^\times\to F^\times$ from \cite[Chapter 16]{Pierce}.
\begin{prop}\label{DerSubO_D} The derived subgroup of $\cO_D^\times$ is the intersection of \[\ker \left(\Nrd|_{\cO_D^\times}\colon \cO_D^\times\to \cO_F^\times\right)\mbox{ and }\ker \left(\omega_D : \cO_D^\times\to k_D^\times\right).\] \end{prop}
\begin{proof} Since $\cO_F^\times$ and $k_D^\times$ are both abelian groups, the derived subgroup must be contained in the intersection given in the statement. By a result of Riehm \cite{Riehm}, the two subgroups are in fact equal --- see the remark following \cite[Theorem 1.9]{PlatRap94}. 
\end{proof}
\begin{defn} \label{DefP1L} Let $P^1_L$ denote the following subgroup of $\cO_L^\times$:\[ P^1_L:=\ker N_{L/F} \quad\cap\quad \ker(\omega_L : \cO_L^\times\to k_L^\times).\] 
\end{defn}

Let $\varrho: \cO_D^\times \twoheadrightarrow (\cO_D^\times)^{\ab}$ denote the canonical projection. 
\begin{prop} \label{AbelO_D^x} The $F$-algebra homomorphism $\iota : F \hookrightarrow L$ induces an isomorphism of profinite abelian groups
\[ \overline{\varrho \circ \iota} : \cO_L^\times / P^1_L \stackrel{\cong}{\longrightarrow} (\cO_D^\times)^{\ab}.\]
\end{prop}
\begin{proof} The projection $\varrho$ appears in the following diagram:
\[\xymatrix{ \cO_L^\times \ar[rr]^\iota \ar[d]_{N_{L/F} \times \omega_L} && \cO_D^\times \ar@{>>}[rrrr]^\varrho \ar[d]^{\Nrd \times \omega_D} &&&& (\cO_D^\times)^{\ab} \ar[dllll]^q \\
\cO_F^\times \times k_L^\times \ar[rr]_{\id \times \overline{\iota}} && \cO_F^\times \times k_D^\times
} \]
We have $\Nrd \circ \iota = N_{L/F}$ by \cite[Proposition 16.2(b)]{Pierce}. The inclusion $\iota : L \hookrightarrow D$ induces ring homomorphisms $\iota : \cO_L \hookrightarrow \cO_D$ and $\overline{\iota} : k_L \hookrightarrow k_D$ \footnote{In fact, $\overline{\iota}$ is an isomorphism}, and we have $\omega_D \circ \iota = \overline{\iota} \circ\omega_L$ . We can now see that the square on the left is commutative. Since $\cO_F^\times \times k_D^\times$ is abelian, $\Nrd \times \omega_D$ factors through $\varrho$, giving the diagonal arrow $q$ and making the entire diagram commutative. 

Proposition \ref{DerSubO_D} tells us that $\ker \varrho = \ker(\Nrd \times \omega_D)$, so the map $q$ is injective. Since $\overline{\iota}$ is injective, chasing the diagram shows that
\[ \ker(\varrho \circ \iota) = \ker(q \circ \varrho \circ \iota) = \ker((\Nrd \times \omega_D) \circ \iota) = \ker(N_{L/F} \times \omega_L) = P^1_L.\]
Hence $\varrho \circ \iota : \cO_L^\times \to (\cO_D^\times)^{\ab}$ descends to give an injective group homomorphism
\[ \overline{\varrho \circ \iota} : \cO_L^\times / P^1_L \hookrightarrow (\cO_D^\times)^{\ab}.\]
Since both $\cO_L^\times/P^1_L$ and $(\cO_D^\times)^{\ab}$ are abelian groups that are virtually pro-$p$, to show that $\overline{\varrho \circ \iota}$ is surjective, it suffices to check this on the Sylow pro-$p$ subgroups of both groups, and on the subgroups of elements of order coprime to $p$. 

The Sylow pro-$p$ subgroup $S$ of $(\cO_D^\times)^{\ab}$ appears in the commutative triangle
\[ \xymatrix{  1 + \pi_F \cO_L  \ar[drr]_{N_{L/F} \times 1} \ar[rr]^{\varrho \circ \iota} && S \ar[d]^q \\
&& (1 + \pi_F \cO_F) \times 1 }\]
where the diagonal arrow is surjective by \cite[Chapter V, Proposition 3(a)]{SerreLF}. Since $q$ is injective, we see that $\varrho \circ \iota : 1 + \pi_F \cO_L \to S$ is surjective as well. 

Finally, $\ker \omega_D = 1 + \cP_D$ is a pro-$p$ subgroup of $\cO_D^\times$, so $\varrho : \cO_D^\times \twoheadrightarrow (\cO_D^\times)^{\ab}$ induces a surjective homomorphism $k_D^\times \twoheadrightarrow (\cO_D^\times)^{\ab}[p']$. Since $\overline{\iota} : k_L \to k_D$ is an isomorphism, this implies that $\varrho \circ i : \cO_L^\times[p'] \to (\cO_D^\times)^{\ab}[p']$ is surjective as well.
\end{proof}

\begin{cor}\label{charactersofDandL} Let $\iota : L \hookrightarrow D$ be an $F$-algebra homomorphism.
\be \item There is an isomorphism of abelian groups
\[\overline{\varrho \circ \iota}^\ast : \Hom(\cO_D^\times, K^\times) \stackrel{\cong}{\longrightarrow} \Hom(\cO_L^\times/P^1_L, K^\times).\]
\item This induces a bijection 
\[ \overline{\overline{\varrho \circ \iota}^\ast} : \Hom(\cO_D^\times,K^\times)/D^\times \stackrel{\cong}{\longrightarrow} \Hom(\cO_L^\times/P^1_L,K^\times)/\Gal(L/F).\] 
\item The bijection $\overline{\overline{\varrho \circ \iota}^\ast}$ does not depend on the choice of $\iota$. \ee\end{cor}
\begin{proof} (a) This follows immediately from Proposition \ref{AbelO_D^x}.

(b) Let $\Gal(L/F) = \langle \sigma\rangle$; then, by \cite[Theorem 17.10, Proposition 15.1a]{Pierce}, we can find an element $\Pi \in D$ such that $\Pi^2 = \pi_F$ and 
\[ \Pi \hsp \iota(a) \hsp \Pi^{-1} = \iota(\sigma(a)) \qmb{for all} a \in L.\]
Since $\cO_D^\times$ is normal in $D^\times$, there is a natural conjugation action of $D^\times$ on $(\cO_D^\times)^{\ab}$ which evidently factors through $D^\times / F^\times \cO_D^\times$, a group of order $2$ generated by the image of $\Pi$. Then the above formula shows that this action of $D^\times / F^\times \cO_D^\times$ on $(\cO_D^\times)^{\ab}$ corresponds under the isomorphism $\overline{\varrho \circ \iota}$  to the natural $\Gal(L/F)$-action on $\cO_L^\times / P^1_L$. Hence, the isomorphism in part (a) is equivariant with respect to the $D^\times/F^\times \cO_D^\times$-action on $\Hom(\cO_D^\times,K^\times)$ and the $\Gal(L/F)$-action on $\Hom(\cO_L^\times/P^1_L,K^\times)$, when we identify $\Gal(L/F)$ with $D^\times/F^\times \cO_D^\times$ via $\sigma \mapsto \Pi$.

(c) Let $\iota' : L \hookrightarrow D$ be another $F$-algebra homomorphism. Then by Corollary \ref{CorToSKN} we can find $d \in D^\times$ such that $\iota'(x) = d\hsp \iota(x)\hsp d^{-1}$ for all $x \in L$. 

Let $c_d : \cO_D^\times \to \cO_D^\times$ and $\overline{c_d} : (\cO_D^\times)^{\ab} \to (\cO_D^\times)^{\ab}$ denote the conjugations by $d$, so that $\iota' = c_d \circ \iota$ and $\overline{c_d} \circ \varrho = \varrho \circ c_d$. Then $\overline{c_d} \circ \varrho \circ \iota = \varrho \circ c_d \circ \iota = \varrho \circ \iota'$, so $\overline{c_d} \circ \overline{\varrho \circ \iota} = \overline{\varrho \circ \iota'}$. Hence $\overline{\varrho \circ \iota'}^\ast = \overline{\varrho \circ \iota}^\ast \circ \overline{c_d}^\ast$ and we can now see that $\overline{\overline{\varrho \circ \iota'}^\ast} = \overline{\overline{\varrho \circ \iota}^\ast}$.
 \end{proof}

\subsection{Equivariant sheaves and amalgamated products}\label{EqShAmPr}
We begin by recalling some material from \cite[\S 2.3]{EqDCap}. Let $X$ be a set equipped with a Grothendieck topology in the sense of \cite[Definition 9.1.1/1]{BGR}. Note that we do not assume at the outset that there is a final object in the category of admissible open subsets of $X$, as $X$ is not itself required to be admissible open in the $G$-topology. 

Let $\Homeo(X)$ be the group of continuous bijections from $X$ to itself. We say that a group $G$ \emph{acts on $X$} if there is given a group homomorphism $\rho : G \to \Homeo(X)$. If this action is understood, we write $gU$ to denote the image of an admissible open subset $U$ of $X$ under the action of $g \in G$. For every $g \in G$, there is an auto-equivalence $\rho(g)_\ast$ of the category of sheaves on $X$, with inverse $\rho(g)^\ast = \rho(g^{-1})_\ast$. To simplify the notation, we will simply denote these auto-equivalences by $g_\ast$ and $g^\ast$, respectively. Thus
\[ (g_\ast \cF)(U) = \cF(g^{-1}U)\qmb{and} (g^\ast \cF)(U) = \cF(gU)\]
for all admissible open subsets $U$ of $X$ and all $g\in G$. 

\begin{defn}\label{DefnEquivSheaf} Let $G$ act on $X$, and let $\cF$ be a presheaf of $R$-modules on $X$, where $R$ is a commutative base ring.
\be \item  An \emph{$R$-linear equivariant structure} on $\cF$ is a set $\{g^{\cF}: g \in G\}$, where 
\[ g^{\cF} : \cF \to g^\ast \cF\]
is a morphism of presheaves of $R$-modules for each $g \in G$,  such that 
\begin{equation}\label{Cocycl} (gh)^{\cF} = h^\ast(g^\cF) \circ h^\cF \qmb{for all} g,h \in G,\qmb{and} 1^{\cF} = 1_{\cF}.\end{equation}
\item An \emph{$R$-linear $G$-equivariant presheaf} is a pair $(\cF, \{g^{\cF}\}_{g \in G})$, where $\cF$ is a presheaf of $R$-modules on $X$, and  $\{g^{\cF}\}_{g \in G}$ is an $R$-linear equivariant structure on $\cF$. 
\item A \emph{morphism} of $R$-linear $G$-equivariant presheaves 
\[\varphi : (\cF, \{g^{\cF}\}) \to (\cF', \{g^{\cF'}\})\] 
is a morphism of presheaves of $R$-modules $\varphi : \cF \to \cF'$ such that 
\[ g^\ast(\varphi) \circ g^{\cF} = g^{\cF'} \circ \varphi \qmb{for all} g \in G.\]
\ee\end{defn}
We will frequently use this abuse of notation, and simply write $\varphi(x)$ to mean $\varphi(U)(x)$ if $x$ is a section of $\cF$ over the admissible open subset $U$ of $X$. Note that with this abuse of notation, the cocycle condition $(\ref{Cocycl})$ becomes simply
\begin{equation}\label{EasyCocyc} g^{\cF}(h^{\cF}(x)) = (gh)^{\cF}(x) \qmb{for all} x \in \cF, g,h \in G.\end{equation}
When the base ring $R$ and the $R$-linear equivariant structure on a sheaf $\cF$ of $R$-modules is understood, we will simply say that $\cF$ is a \emph{$G$-equivariant sheaf}, and omit the equivariant structure from the notation. 
\begin{defn}\label{AutFX}  Let $\cF$ be a presheaf of $R$-modules on $X$. 
\be \item An \emph{automorphism of $\cF$ over $X$} is a pair $(\alpha, \beta)$, where $\alpha \in \Homeo(X)$ and $\beta : \cF \to \alpha^\ast \cF$ is an $R$-linear isomorphism of presheaves on $X$. 
\item Define $\Aut(\cF/X)$ to be the set of all automorphisms of $\cF$ over $X$.
\item Given $(\alpha_1,\beta_1), (\alpha_2,\beta_2) \in \Aut(\cF/X)$, define 
\[ (\alpha_1,\beta_1) \square (\alpha_2,\beta_2) := (\alpha_1\alpha_2, \alpha_2^\ast(\beta_1) \circ \beta_2).\]
This is again an element of $\Aut(\cF/X)$. 
\ee
\end{defn}
\begin{lem} Let $\cF$ be a presheaf of $R$-modules on $X$. Then the binary operation $\square$ turns $\Aut(\cF/X)$ into a group.
\end{lem}
\begin{proof} The identity element is $(1_X,1_{\cF})$. Let $(\alpha_i,\beta_i)$, $i = 1,2,3$ be three elements of $\Aut(\cF/X)$. Checking that the operation $\square$ is associative boils down to the formula
\[ \alpha_3^\ast(\alpha_2^\ast(\beta_1)\circ \beta_2) \circ \beta_3 = (\alpha_2\alpha_3)^\ast(\beta_1) \circ \alpha_3^\ast(\beta_2) \circ \beta_3,\]
which is readily verified. The inverse of $(\alpha,\beta) \in \Aut(\cF/X)$ is $(\alpha^{-1}, \alpha_\ast(\beta)^{-1})$.
\end{proof}

By Definition \ref{AutFX}(c), the first projection map $\pr_1 : \Aut(\cF/X) \to \Homeo(X)$ is a group homomorphism.
\begin{defn} Let $G$ be a group acting on $X$ via $\rho : G \to \Homeo(X)$, and let $\cF$ be a presheaf of $R$-modules on $X$. Form the fibre product 
\[\Aut(\cF/X/G) :=  G \quad \underset{\Homeo(X)}{\times}{} \quad \Aut(\cF/X)\]
with respect to the group homomorphisms 
\[\rho : G \to \Homeo(X) \qmb{and} \pr_1 : \Aut(\cF/X) \to \Homeo(X).\]
\end{defn}
By definition, the elements of $\Aut(\cF/X/G)$ have the form $(g, (\rho(g), \beta))$ for some $g \in G$ and some $R$-linear isomorphism $\beta : \cF \stackrel{\cong}{\longrightarrow} g^\ast \cF$; evidently, such an element is completely determined by the pair $(g, \beta)$. In order to simplify the notation, we will abuse notation and write
\[ \Aut(\cF/X/G) = \left\{ (g, \beta) : g \in G, \quad \beta : \cF \stackrel{\cong}{\longrightarrow} g^\ast \cF \mbox{ } R-\mbox{linear} \right\},\]
where the product is given by the formula
\begin{equation}\label{GpLawAutFXG}(g_1, \beta_1) \square (g_2,\beta_2) = (g_1g_2, g_2^\ast(\beta_1) \circ \beta_2).\end{equation}

\begin{defn} Let $G$ be a group acting on $X$ via $\rho : G \to \Homeo(X)$, and let $\cF$ be a presheaf of $R$-modules on $X$. Define
\[ \cS(G,\cF) := \left\{ \sigma \in \Hom(G, \Aut(\cF/X/G)) : \pr_1 \circ \sigma = 1_G\right\}\]
to be the set of sections of the first projection $\pr_1 : \Aut(\cF/X/G) \to G$.
\end{defn}
We make these definitions in order to formulate the following

\begin{lem}\label{RGstrAsSections} Let $G$ be a group acting on $X$ via $\rho : G \to \Homeo(X)$, and let $\cF$ be a presheaf of $R$-modules on $X$. Then the rule
\[ \{ g^{\cF} \}_{g \in G} \quad \mapsto \quad \left[ g \mapsto (g, g^{\cF}) \in \Aut(\cF/X/G) \right]\]
defines a bijection between the set all $R$-linear $G$-equivariant structures on $\cF$ and $\cS(G,\cF)$.
\end{lem}
\begin{proof} Let $\{ g^{\cF} \}_{g \in G}$ be an $R$-linear $G$-equivariant structure on $X$. Define the map $\sigma : G \to \Aut(\cF/X/G)$ by setting $\sigma(g) = (g, g^{\cF})$ for all $g \in G$. Using the cocycle condition (\ref{Cocycl}), we compute that for all $g, h \in G$ we have
\[ \sigma(gh) = (gh, h^\ast(g^{\cF}) \circ h^{\cF}) = (g, g^{\cF}) \square (h, h^{\cF}) = \sigma(g) \square \sigma(h).\]
Since $\sigma(1) = (1, 1^{\cF}) = (1, 1_{\cF})$, we see that $\sigma$ is a group homomorphism such that $\pr_1 \circ \sigma = 1_G$, that is, $\sigma \in \cS(G,\cF)$. 

Conversely, for each $\sigma \in \cS(G,\cF)$, the set $\{ \pr_2(\sigma(g)) \}_{g \in G}$ forms an $R$-linear $G$-equivariant structure on $\cF$ by reversing the above argument.
\end{proof}

Next, we recall the following definitions from \cite{EqDCap}:

\begin{defn} Let $G$ act on $X$, and let $\cA$ be a sheaf of $R$-algebras on $X$. 
\be \item
We say that $\cA$ is a \emph{$G$-equivariant sheaf of $R$-algebras} if there is given an $R$-linear $G$-equivariant structure $\{g^{\cA} : g \in G\}$ such that each $g^{\cA} : \cA \to g^\ast \cA$ is a morphism of sheaves of $R$-algebras.
\item Let $\cA$ be a $G$-equivariant sheaf of $R$-algebras on $X$. A \emph{$G$-$\cA$-module} is an $R$-linear $G$-equivariant sheaf $\cM$ on $X$, such that $\cM$ is a sheaf of left $\cA$-modules and $g^{\cM}(a \cdot m) = g^{\cA}(a) \cdot g^{\cM}(m)$ for all $g \in G$, $a \in \cA, m \in \cM$. 
\ee
\end{defn}
We want to study all possible $G$-$\cA$-module structures on a given $\cA$-module $\cM$.

\begin{defn}\label{AutAMXG} Let $G$ act on $X$, let $\cA$ be a $G$-equivariant sheaf of $R$-algebras on $X$ and let $\cM$ be an $\cA$-module. We define
\[ \Aut_{\cA}(\cM/X/G) := \left\{ \begin{array}{rcl} (g, \beta) &\in& \Aut(\cM/X/G) : \\ \beta(a \cdot m) &=& (g \cdot a) \cdot \beta(m) \qmb{for all} a \in \cA, m \in \cM\end{array}\right\}.\]
\end{defn}

\begin{lem} With the hypotheses of Definition \ref{AutAMXG}, $\Aut_{\cA}(\cM/X/G)$ is a subgroup of $\Aut(\cM/X/G)$. \end{lem}
\begin{proof} It is clear that $(1, 1_{\cF})$ lies in $\Aut_{\cA}(\cM/X/G)$. Let $(g_1,\beta_1)$ and $(g_2,\beta_2)$ be two elements of $\Aut_{\cA}(\cM/X/G)$, so that 
\begin{equation}\label{BetaiAM}\beta_i(a \cdot m) = (g_i \cdot a) \cdot \beta_i(m) \qmb{for all} a \in \cA, m \in \cM, i=1,2.\end{equation} 
Let $(g_3,\beta_3) := (g_1,\beta) \square (g_2,\beta_2)$ so that $g_3 = g_1g_2$ and $\beta_3 = g_2^\ast(\beta_1) \circ \beta_2$. On local sections, $\beta_3$ is simply the composition $\beta_1\beta_2$. For any $a \in \cA$ and $m \in \cM$, we use the fact that $\cA$ is a $G$-equivariant sheaf together with $(\ref{BetaiAM})$ to compute
\[ \begin{array}{lll} \beta_3(a \cdot m) &=& \beta_1(\beta_2(a \cdot m)) = \beta_1( (g_2\cdot a) \cdot \beta_2(m)) = (g_1\cdot (g_2\cdot a)) \cdot \beta_1(\beta_2(m)) \\
 &=& ((g_1g_2) \cdot a) \cdot (\beta_1\beta_2(m)) = (g_3 \cdot a)\cdot \beta_3(m). \end{array}\]
So, $(g_3,\beta_3) \in \Aut_{\cA}(\cM/X/G)$ and $\Aut_{\cA}(\cM/X/G)$ is closed under composition.

To show that $\Aut_{\cA}(\cM/X/G)$ is stable under inversion in $\Aut(\cM/X/G)$, let $(g,\beta) \in \Aut_{\cA}(\cM/X/G)$. Then for $b := g^{-1} \cdot a \in \cA$ and $w := \beta^{-1}(v) \in \cM$ we have $\beta( b \cdot w) = (g \cdot b) \cdot \beta(w) = a \cdot v.$ Applying $\beta^{-1}$ to this equation gives $\beta^{-1}(a\cdot v) = b \cdot w = (g^{-1} \cdot a) \cdot \beta^{-1}(v)$, so $(g,\beta)^{-1} = (g^{-1}, g_\ast(\beta^{-1})) \in \Aut_{\cA}(\cM/X/G)$.
\end{proof}

\begin{defn} With the hypotheses of Definition \ref{AutAMXG}, define
\[ \cS_{\cA}(G,\cM) := \left\{ \sigma \in \Hom(G, \Aut_{\cA}(\cM/X/G)) : \pr_1 \circ \sigma = 1_G\right\}\]
to be the set of sections of the first projection $\pr_1 : \Aut_{\cA}(\cM/X/G) \to G$.
\end{defn}
We can now give the generalisation of Lemma \ref{RGstrAsSections} to the case of $\cA$-modules: the proof is completely straightforward and is therefore omitted.

\begin{prop}\label{GAmodStrAsSections} Assume the hypotheses of Definition \ref{AutAMXG}. Then
\[ \{ g^{\cM} \}_{g \in G} \quad \mapsto \quad \left[ g \mapsto (g, g^{\cM}) \in \Aut_{\cA}(\cM/X/G)) \right]\]
defines a bijection between the set all $G$-$\cA$-module structures on $\cM$  extending the given $\cA$-module structure on $\cM$, and $\cS_{\cA}(G,\cM)$.
\end{prop}
Next, we briefly study the functorialities of $\Aut_{\cA}(\cM/X/G)$ and $\cS_{\cA}(G,\cM)$. 

\begin{lem}\label{AutAMXGFunc} Assume the hypotheses of Definition \ref{AutAMXG}, and let $H$ be a subgroup of $G$. 
\be \item $\Aut_{\cA}(\cM/X/H)$ is a subgroup of $\Aut_{\cA}(\cM/X/G)$.
\item Restriction of functions induces a map $\Res^G_H : \cS_{\cA}(G, \cM) \to \cS_{\cA}(H,\cM)$.
\item For any subgroup $J$ of $H$, we have  $\Res^G_J = \Res^H_J \circ \Res^G_H$. 
\ee
\end{lem}
\begin{proof} (a) An element of $\Aut_{\cA}(\cM/X/H)$ is a pair $(h, \beta)$ where $h\in H$ and $\beta : \cM \to h^\ast \cM$ is an $R$-linear isomorphism of sheaves such that $\beta(a \cdot m) = (h \cdot a) \cdot \beta(m)$ for all $a \in \cA$ and $m \in \cM$. Evidently such a pair is also an element of $\Aut_{\cA}(\cM/X/G)$. 

(b) Given a group homomorphism $\sigma : G \to \Aut_{\cA}(\cM/X/G)$ such that $\pr_1 \circ \sigma = 1_G$, the restriction $\sigma|_H : H \to \Aut_{\cA}(\cM/X/G)$ takes values in $\Aut_{\cA}(\cM/X/H)$. It is still a group homomorphism, and $\pr_1 \circ  \sigma|_H = (\pr_1 \circ \sigma)|_H = (1_G)_H = 1_H$. Hence $\sigma \mapsto \sigma|_H$ defines the required function $\Res^G_H : \cS_{\cA}(G, \cM) \to \cS_{\cA}(H,\cM)$.

(c) This is clear from the definitions.
\end{proof}

We now come to the application of the above formalism. Suppose that the group $G$ is equal to an amalgamated product
\[ G = A \underset{C}{\ast} B\]
of its subgroups $A$ and $B$, along their common subgroup $C$. Using Lemma \ref{AutAMXGFunc}, we see that sending $\sigma \in \cS_{\cA}(G,\cM)$ to the pair $(\Res^G_A(\sigma), \Res^G_B(\sigma))$ defines a function
\begin{equation}\label{SGM} \cS_{\cA}(G,\cM)  \longrightarrow \cS_{\cA}(A,\cM) \underset{\cS_{\cA}(C,\cM)}{\times}{} \cS_{\cA}(B,\cM). \end{equation}

\begin{thm}\label{amalgequiv}Let $G$ be a group acting on $X$, let $\cA$ be a $G$-equivariant sheaf of $R$-algebras on $X$ and let $\cM$ be an $\cA$-module. Suppose further that $G$ is equal to the amalgamated product $G = A \underset{C}{\ast} B$ of its subgroups $A$ and $B$ along their common subgroup $C$. Then the map $(\ref{SGM})$ is a bijection.
\end{thm}
\begin{proof} Using Lemma \ref{AutAMXGFunc}(a), we have the commutative diagram of groups and group homomorphisms
\[ \xymatrix{ & \Aut_{\cA}(\cM/X/A) \ar[dr] & \\ \Aut_{\cA}(\cM/X/C) \ar[ur]\ar[dr] & & \Aut_{\cA}(\cM/X/G). \\ & \Aut_{\cA}(\cM/X/B) \ar[ur] &}\]
Let $\sigma_1, \sigma_2 \in \cS_{\cA}(G,\cM)$ be such that $\Res^G_A(\sigma_1) = \Res^G_A(\sigma_2)$ and $\Res^G_B(\sigma_1) = \Res^G_B(\sigma_2)$. Using the above diagram, we may regard $\sigma_1$ and $\sigma_2$ having the same codomain $\Aut_{\cA}(\cM/X/G)$. Then $(\sigma_1)|_A = (\sigma_2)|_A$ and $(\sigma_1)|_B = (\sigma_2)|_B$. Since $A$ and $B$ generate $G$ as a group, it follows that $\sigma_1 = \sigma_2$.

Suppose now that $(\tau,\psi)$ is an element of the fibre product on the right hand side of $(\ref{SGM})$. Then $\tau : A \to \Aut_{\cA}(\cM/X/G)$ and $\psi : B \to \Aut_{\cA}(\cM/X/G)$ have the same restriction to $C$. By the universal property of amalgamated products --- see \cite[equation $(\ast), \S$ 1.1]{Trees} --- $\tau$ and $\psi$ extend to a unique group homomorphism $\sigma : G \to \Aut_{\cA}(\cM/X/G)$ such that $\sigma|_A = \tau$ and $\sigma|_B = \psi$. Then $(\pr_1 \circ \sigma)|_A = \pr_1 \circ (\sigma|_A) = \pr_1 \circ \tau = 1_A$ and $(\pr_1 \circ \sigma)|_B = \pr_1 \circ (\sigma|_B) = \pr_1 \circ \psi = 1_B$ because $\tau \in \cS_{\cA}(\cM/X/A)$ and $\psi \in \cS_{\cA}(\cM/X/B)$. Since $A$ and $B$ generate $G$ as a group, and the group homomorphism $\pr_1 \circ \sigma : G \to G$ fixes both $A$ and $B$ pointwise, we conclude that $\pr_1 \circ \sigma = 1_G$. So, $\sigma \in \cS_{\cA}(\cM/X/G)$, $\Res^G_A(\sigma) = \tau$ and $\Res^G_B(\sigma) = \psi$.
\end{proof}

To spell out the meaning of Theorem \ref{amalgequiv} together with Proposition \ref{GAmodStrAsSections}: the data of a $G$-$\cA$-module structure on $\cM$ is equivalent to the data of an $A$-$\cA$-module structure and a $B$-$\cA$-module structures whose restrictions to a $C$-$\cA$-module structure agree.

\section{Topics in rigid analytic geometry}

\subsection{Line bundles with flat connection on smooth rigid spaces}\label{ELBsec}

Let $X$ be a non-empty, smooth rigid $K$-analytic space. By a \emph{line bundle with flat connection} we mean a $\cD$-module $\sL$ on $X$ which is invertible as an $\cO$-module.

If $\sL$ and $\sM$ are two line bundles with flat connection on $X$, then so are $\sL \otimes_{\cO} \sM$ and $\sL^{\otimes -1} := \mathpzc{Hom}_{\cO}(\sL, \cO)$: the tangent sheaf $\cT_X$ acts via the Leibniz rule on $\sL \otimes_{\cO} \sM$, and on $\sL^{\otimes -1}$ via the rule
\[(v \cdot f)(\ell) = v \cdot f(\ell) - f(v\cdot \ell)\qmb{for all} v \in \cT_X, f \in \sL^{\otimes -1}, v \in \sL.\]

\begin{defn}\label{PicConDefn}\hfill \be \item $\PicCon(X)$ denotes the abelian group of isomorphism classes of line bundles with flat connection on $X$ under the operation $-\otimes_\cO$-.  \item $\Con(X) := \ker(\PicCon(X) \to \Pic(X))$ denotes the group of isomorphism classes of line bundles with flat connection on $X$ that are trivial after forgetting the connection.  \ee
\end{defn}

We now show that when $X$ is connected, $\sL$ is a simple $\cD_X$-module for any $[\sL] \in \PicCon(X)$. We start with the case where $X$ is $K$-affinoid. 

\begin{lem}\label{LXsimple} Suppose that $X$ is a connected $K$-affinoid variety such that $\cT_X$ is a free $\cO_X$-module. Then for every $[\sL]\in \Con(X)$, $\sL(X)$ is a simple $\cD(X)$-module.
\end{lem}

\begin{proof} Suppose that $n=\dim X$. Following the formalism of \cite[\S1]{MebNar} we see that $\cO(X)$ satisfies the conditions found in \S1.1.2 of {\it loc. cit.} Let $\partial_1,\ldots,\partial_n$ denote a free generating set for $\cT(X)$ as a $\cO(X)$-module so we may consider \[\cD(X)=\cO(X)[\partial_1,\ldots,\partial_n]\] as a filtered $K$-algebra with associated graded ring \[\gr \cD(X)\cong \cO(X)[T_1,\ldots, T_n]\] with $\cO(X)$ in degree $0$ and $T_1,\ldots,T_n$ in degree $1$, the principal symbols of $\partial_1,\ldots,\partial_n$ respectively. 
	
The filtration of $\sL(X)$ whose $0$th filtered part is $\sL(X)$ and whose $-1$st filtered part is $0$ is a good filtration, so that \[\gr \sL(X)\cong \cO(X)[T_1,\ldots,T_n]/(T_1,\ldots,T_n).\] 
	
	Since $X$ is also connected, $\cO(X)$ is an integral domain by the proof of \cite[Proposition 4.2]{ABB}.  Thus any non-zero proper $\cD(X)$-module quotient of $\sL(X)$ must have dimension $<n$. However, by \cite[Th\'eor\`eme 1.1.4, Corollaire 1.2.3]{MebNar}, no such $\cD(X)$-module can exist and so $\sL(X)$ is a simple $\cD(X)$-module as claimed. 
\end{proof}

\begin{prop}\label{simple} Suppose $X$ is connected and that $[\sL]\in \PicCon(X)$. Then $\sL$ is simple as a $\cD$-module. 
\end{prop}
\begin{proof} Since $\sL$ is a line bundle, there is an admissible cover $\cU$ of $X$ consisting of $K$-affinioid subdomains such that the line bundle $\sL|_U$ is trivial for all $U \in \cU$. By passing to a refinement, we may also assume that for each $U\in \cU$, $U$ is connected and that $\cT|_U$ is a free $\cO_U$-module. 
	
 Suppose that $\sM$ is a subobject of $\sL$ as a $\cD$-module and consider \[\cV_1 := \{U \in \cU : \sM(U)=\sL(U)\}\mbox{ and }\cV_2=\{U\in \cU: \sM(U)=0 \}.\] Then $\cU$ is the disjoint union of $ \cV_1$ and $\cV_2$ by Lemma \ref{LXsimple}. Now if $U \in \cV_1$ and $V \in \cV_2$ with $U \cap V \neq \emptyset$, then \[\cL(U\cap V)\cong \cO(U\cap V)\otimes_{\cO(U)}\sM(U)\cong \sM(U\cap V)\cong \cO(U\cap V)\otimes_{\cO(V)}\sM(V)=0,  \] a contradiction.  Since $X$ is connected if follows that $\cU=\cV_1$ or $\cU=\cV_2$. Hence $\sM=\sL$ or $\sM=0$ as required. 
\end{proof}
Whenever $\sM$ is a $\cD$-module on $X$, we have at our disposal the $K$-vector space of \emph{global horizontal sections}
\[ \sM(X)^{\cT(X) = 0} := \{m \in \sM(X) : \partial \cdot m = 0 \qmb{for all} \partial \in \cT(X)\}.\]
Similarly, when $x \in X$, we have the vector space of \emph{local horizontal sections}
\[ \sM_{X,x}^{\cT_{X,x} = 0} := \{m \in \sM_{X,x} : \partial \cdot m = 0 \qmb{for all} \partial \in \cT_{X,x}\}.\]
These vector spaces are related as follows.
\begin{lem} \label{GlobHorSec} Let $\sM$ be a $\cD$-module on $X$.
\be
\item There is an injective $K$-linear map
\[ \Hom_{\cD}(\cO,\sM) \hookrightarrow \sM(X)^{\cT(X)=0}\]
given by $\varphi \mapsto \varphi(X)(1)$. This map is functorial in $\sM$ and commutes with restriction to admissible open subspaces $U$ of $X$.
\item Suppose that $\cT$ is generated as an $\cO$-module by its global sections. Then the map in part (a) is an isomorphism, and for every $x \in X$, restriction sends $\sM(X)^{\cT(X)=0}$ into $\sM_{X,x}^{\cT_{X,x}=0}$.
\ee\end{lem}
\begin{proof} (a) Recall that for any $\cO$-module $\frF$ on \emph{any} ringed space $(\frX,\cO)$, the rule $\varphi \mapsto \varphi(\frX)(1)$ gives a functorial bijection $\Hom_{\cO}(\cO, \frF) \stackrel{\cong}{\longrightarrow} \frF(\frX)$, whose inverse sends the global section $f \in \frF(\frX)$ to the $\cO$-linear map $r_f : \cO \to \frF$ sending $a \in \cO(\mathfrak{U})$ to $a \cdot f|_{\mathfrak{U}}$. In our setting, this bijection sends the $\cD$-linear homomorphisms to the global horizontal sections, because $1 \in \cO(X)$ is killed by $\cT(X)$. The first statement follows, and the second is clear.

(b) Suppose that $f \in \sM(X)$ is such that $\cT(X)\cdot f = 0$. Then $\cT(U) \cdot f|_U = 0$ for all admissible open $U \subseteq X$ because $\cT$ is generated by global sections. It follows that the map $r_f(U)  : \cO(U) \to \sM(U)$ is $\cD(U)$-linear, and hence $r_f \in \Hom_{\cD}(\cO,\sM)$. 

The second statement holds because global horizontal sections of $\sM$ restrict to local horizontal sections of $\sM$ whenever $\cT$ is generated by global sections.
\end{proof}
Before we can proceed, we recall some well-known facts about the local structure of smooth rigid spaces. 
\begin{lem} \label{local} Let $K(x) := \cO_{X,x}/\frm_{X,x}$ denote the residue field of $x \in X$.
\be \item $\cO_{X,x}$ is a regular local ring.
\item There exist local parameters $t_1,\cdots,t_d \in \cO_{X,x}$ and derivations $v_1,\cdots,v_d \in \cT_{X,x}$ such that $v_i(t_j) = \delta_{i,j}$ for all $i,j=1,\cdots,d$. 
\item The $\frm_{X,x}$-adic completion $\h{\cO_{X,x}}$ of $\cO_{X,x}$ is isomorphic to the formal power series ring $K(x)[[t_1,\cdots,t_d]]$.
\item The natural map $\widehat{\cO_{X,x}}^{\cT_{X,x} = 0} \to K(x)$ is an isomorphism.
\item The restriction map $\cO(X) \to \h{\cO_{X,x}}$ is injective whenever $X$ is connected.
\ee
\end{lem}
\begin{proof} (a) This follows from the smoothness of $X$.

(b) This follows from the proof of \cite[Theorem A.5.1]{HTT}. 

(c) Since $K$ is a field of characteristic zero, the local ring $\cO_{X,x}$ is equi-characteristic. Hence it admits a coefficient field by \cite[Theorem 28.3]{MatRingThy}. The result now follows by using part (a) together with \cite[Theorem 11.22]{AMac}.

(d) This follows immediately from (b) and (c).

(e) Suppose that $f \in \cO(X)$ maps to zero in $\h{\cO_{X,x}}$. Since $\cO_{X,x}$ is a regular local ring by (a), it embeds into $\h{\cO_{X,x}}$ by \cite[Corollary 10.19]{AMac}, and hence the germ $f_x \in \cO_{X,x}$ of $f$ at $x$ is zero. Hence there exists an affinoid subdomain $U$ of $X$ containing $x$ such that $f|_U = 0$. Let $Z = V(f \cO_X)$ denote the closed $K$-analytic subspace of $X$ cut out by $f$; then $U \subseteq Z$. Since $X$ is smooth, it is normal in the sense of \cite{ConradIrred}. Since it is also connected, \cite[Lemma 2.1.4]{ConradIrred} ensures that $Z = X$. Since $X$ is reduced, it follows that $f = 0$. 
\end{proof}

\begin{prop}\label{GeomConGlobHorO} Suppose that $X$ is geometrically connected and that $\cT$ is generated as an $\cO$-module by its global sections. Then 
\[ \cO(X)^{\cT(X)=0}=K.\]
\end{prop}
\begin{proof} Pick a point $x \in X$ and consider the completion $\widehat{\cO_{X,x}}$ of the local ring $\cO_{X,x}$ at $x$. Since $X$ is connected, the restriction map $\cO(X) \to \widehat{\cO_{X,x}}$ is injective by Lemma \ref{local}(e).
This restriction map sends the $K$-subalgebra $L := \cO(X)^{\cT(X) = 0}$ of $\cO(X)$ into $\widehat{\cO_{X,x}}^{\cT_{X,x} = 0}$ by Lemma \ref{GlobHorSec}(b), which is isomorphic to the residue field $K(x)$ by Lemma \ref{local}(d).  Since $K(x)$ is finite dimensional over $K$, we see that $L$ is a finite field extension of $K$. If $L$ was a proper field extension of $K$, then base changing $X$ to $L$ would yield non-trivial idempotents in $L \otimes L \subset \cO(X) \otimes L = \cO(X_L)$ and show that $X_L$ is not connected. Since $X$ was assumed to be geometrically connected, we conclude that $L = \cO(X)^{\cT(X) = 0}$ is in fact equal to $K$.
\end{proof}
As a consequence, we can characterise the trivial line bundle with trivial connection in terms of its horizontal sections under certain assumptions on $X$.

\begin{cor}\label{GlobHorL} Suppose that $X$ is geometrically connected and that $\cT$ is generated as an $\cO$-module by its global sections. Then for every $[\sL]\in \PicCon(X)$, the following are equivalent:
\begin{enumerate}[{(}i{)}]
\item $\sL$ is the trivial line bundle with trivial flat connection,
\item $\dim_K \sL(X)^{\cT(X) = 0} = 1$,
\item $\sL$ has a non-zero global horizontal section.
\end{enumerate}
\end{cor}
\begin{proof} (i) $\Rightarrow$ (ii). This follows from Proposition \ref{GeomConGlobHorO}.

(ii) $\Rightarrow$ (iii). This is clear.

(iii) $\Rightarrow$ (i). By Lemma \ref{GlobHorSec}(b), there is a non-zero $\cD$-linear map $\cO\to \sL$. Since $X$ is connected and since $\cO$ and $\sL$ are both simple $\cD$-modules by Proposition \ref{simple}, this map must be an isomorphism.
\end{proof}

Proposition \ref{GeomConGlobHorO} also implies that an isomorphism between two line bundles with flat connection is unique up to scalars.

\begin{cor}\label{IsoLL}  Suppose that $X$ is geometrically connected and that $\cT$ is generated as an $\cO$-module by its global sections. Let $\varphi_1, \varphi_2 : \sL_1 \to \sL_2$ be two isomorphisms between two line bundles with flat connection on $X$. Then there is a scalar $\lambda \in K^\times$ such that $\varphi_2 = \lambda \varphi_1$.
\end{cor}
\begin{proof} By tensoring $\varphi_1$ and $\varphi_2$ by $\sL_2^{\otimes -1}$ we may assume that $\sL_2 = \cO$. But then $\varphi_1 \circ \varphi_2^{-1} : \cO \to \cO$ is a $\cD$-linear isomorphism, so by Lemma \ref{GlobHorSec}(a) and Proposition \ref{GeomConGlobHorO}, it is given by multiplication by a non-zero scalar in $K$.
\end{proof}

\begin{prop} \label{PicConinvlim} Suppose that $(X_n)_{n\geq 0}$ is an increasing admissible covering of $X$ by geometrically connected affinoid subdomains. 
\be \item $X$ is also geometrically connected.
\item If $\cT$ is generated as an $\cO$-module by its global sections, then the restriction maps $\PicCon(X)\to \PicCon(X_n)$ induce a natural isomorphism \[ \PicCon(X)\stackrel{\cong}{\longrightarrow} \invlim \PicCon(X_n). \]
\ee
 \end{prop}

\begin{proof} (a) By base-changing to $\bfC$, it is enough to show that $X$ is connected. If it isn't, then we can find two non-empty families of non-empty admissible open subsets $\cU$ and $\cV$ of $X$, such that $\cU\cup \cV$ is an admissible covering of $X$ and \[\bigcup_{U\in \cU}U \cap \bigcup_{V\in \cV}V=\emptyset.\] 
Since the admissible covering $(X_n)$ was assumed to be increasing, we can find $n\geq 0$, $U\in \cU$ and $V\in \cV$ such that $X_n \cap U$ and $X_n \cap V$ are both non-empty.  But then $\{U\cap X_n:U\in \cU\}$ and $\{V\cap X_n:V\in \cV\}$ together form an admissible cover of $X_n$ that disconnects it, giving the required contradiction. 
	
(b) Suppose that $[\sL]\in \PicCon(X)$ is such that $[\sL|_{X_n}]=0$ in $\PicCon(X_n)$ for all $n$. Since $\cT$ is generated as an $\cO$-module by its global sections, $\sL^{\cT=0}$ is a subsheaf of $\sL$ with $\sL(X_n)^{\cT=0}=K$ for all $n$ by Proposition \ref{GeomConGlobHorO}. It follows that all restriction maps $\sL(X_{n+1})^{\cT=0}\to \sL(X_n)^{\cT=0}$ are isomorphisms. Hence by the sheaf condition on $\sL^{\cT=0}$, $\sL(X)^{\cT=0}=K$. Using part (a) together with Corollary \ref{GlobHorL} , we now see that $[\sL]=0\in \PicCon(X)$.

To see that the homomorphism is also surjective we consider a family of line bundles with connection $[\sL_n]_{n\geq 0}\in \prod_{n\geq 0}\PicCon(X_n)$ such that \[[\sL_n|_{X_m}]=[\sL_m] \mbox{ in } \PicCon(X_m) \qmb{whenever} n \geq m.\]
Choose isomorphisms of $\cD$-modules $\varphi_m\colon \sL_{m+1}|_{X_m}\stackrel{\cong}{\longrightarrow} \sL_m$ for all $m \geq 0$. Then whenever $n \geq m$ we can define an isomorphism of sheaves $\varphi_{n,m}\colon \sL_n|_{X_m}\to \sL_m$ by \[\varphi_{n,m}:=\varphi_{n-1}|_{X_m}\circ\cdots \circ \varphi_m|_{X_m}.\] Then the construction in \cite[\S 4.4]{FvdPut} gives a sheaf $\sL$ of $\cD$-modules on $X$ together with isomorphisms of $\cD$-modules $\sL|_{X_n}\stackrel{\cong}{\longrightarrow}\sL_n$. Because each $\sL(X_n)$ is free of rank $1$ over $\cO(X_n)$, we see that $\sL$ must be a line bundle on $X$. Hence $[\sL] \in \PicCon(X)$ is the required preimage of $[\sL_n]\in \invlim \PicCon(X_n)$.\end{proof}
	
The following result will also be useful.

\begin{cor} \label{dtors} Suppose that $X$ is geometrically connected and that $\cT$ is generated as an $\cO$-module by its global sections. Then for every positive integer $d$, the abelian group $\cO(X)^\times/K^\times$ has no $d$-torsion.\end{cor}

 \begin{proof}
	Suppose that $f\in \cO(X)^\times$ with $f^d\in K^\times$. Then for every $\partial\in \cT(X)$ we see that $0=\partial(f^d)=df^{d-1}\partial(f)$. Since $df^{d-1}\in \cO(X)^\times$ it follows that $\partial(f)=0$. Thus $f \in K^\times$ by Proposition \ref{GeomConGlobHorO}.
\end{proof}
\begin{lem}\label{GactsPicCon}  \hsp Suppose that a group $G$ acts on $X$. Then $G$ acts naturally on $\PicCon(X)$ by abelian group automorphisms via \[ g\cdot [\sL]  = [g_\ast \sL]\] where $g_\ast \sL$ is a $\cD$-module via the ring isomorphsim $(g^{-1})^\cD\colon \cD\stackrel{\cong}{\longrightarrow}g_\ast\cD$. 
	
Moreover $\Con(X)$ is a $G$-stable subgroup of $\PicCon(X)$. 
\end{lem}
\begin{proof}
	 We can check that $g \cdot ([\sL] [\sM]) = (g \cdot [\sL]) (g \cdot [\sM])$ for any $[\sL], [\sM] \in \PicCon(X)$, because $g_\ast (\sL \otimes_{\cO} \sM)$ is naturally isomorphic to $(g_\ast \sL) \otimes_{g_\ast\cO} (g_\ast \sM)$ as a $g_\ast\cD$-module.
	
	Since $g_\ast\sL$ is trivial as a line bundle whenever $\sL$ is trivial as a line bundle the last part is immediate.   	
\end{proof}

\begin{rmk} \label{GactsPicConRem} We note more generally, in the context of Lemma \ref{GactsPicCon}, that if $U$ is an admissible open subspace of $X$, each $g\in G$ induces a group isomorphism $\PicCon(U)\to \PicCon(g\cdot U)$; $[\sL]\mapsto [g_\ast \sL]$ where again $g_\ast\sL$ is a $\cD_{gU}$-module via the ring isomorphism $(g^{-1})^{\cD}\colon \cD_{gU}\to g_\ast\cD_{U}$. Moreover these restrict to isomorphisms $\Con(U)\to \Con(g\cdot U)$. \end{rmk} 

We will now construct some connections on the trivial line bundle by using units.
\begin{lem} \label{Lud} Suppose that $d$ is a positive integer and $u\in \cO(X)^\times$. Then there is a unique element $[\sL_{u,d}]$ of $\Con(X)$ such that $\sL_{u,d}$ has a free generator $v$ as a $\cO$-module with $\partial(v)= \frac{1}{d} \frac{\partial(u)}{u} v$ for all $\partial\in \cT$. \end{lem}

\begin{proof} Suppose that $\sL=\cO v$ is a line bundle with a flat connection satisfying  \[\partial(v)= \frac{1}{d} \frac{\partial(u)}{u} v\] for all $\partial\in \cT$. For all $f\in \cO$ and $\partial\in \cT$, necessarily \[ \partial(fv)=\left(\partial(f)+ \frac{1}{d}\frac{\partial(u)}{u}f\right)v \] so as $\cD$ is generated by $\cO$ and $\cT$ there is at most one element of $\Con(X)$ with the property given in the statement. To prove the existence of such a line bundle with flat connection it suffices to show that for all $\partial_1,\partial_2\in \cT$ \[ (\partial_1\partial_2-\partial_2\partial_1)(v)=[\partial_1 \partial_2](v)\] where $[--]$ denotes the Lie bracket on $\cT$. 

But \begin{eqnarray*} \partial_1\partial_2(v) & = & \partial_1\left(\frac{1}{d}\frac{\partial_2(u)}{u}v\right)\\ 
	                                          & = & \frac{1}{d}\partial_1\left(\frac{\partial_2(u)}{u}\right)v + \frac{1}{d}\frac{\partial_2(u)}{u}\frac{1}{d}\frac{\partial_1(u)}{u}v \\
	                                          & = & \frac{1}{d}\frac{u\partial_1\partial_2(u)-\partial_2(u)\partial_1(u)}{u^2}v + \frac{1}{d}\frac{\partial_2(u)}{u}\frac{1}{d}\frac{\partial_1(u)}{u}v 
\end{eqnarray*}
and so \[ (\partial_1\partial_2-\partial_2\partial_1)(v)=\frac{1}{d} \frac{[\partial_1 \partial_2](u)}{u}v\] as required.
\end{proof}
\begin{prop}\label{PicConDtor} Suppose that $X$ is geometrically connected and that $\cT$ is generated as an $\cO$-module by its global sections. For every non-zero integer $d$, there is a homomorphism of abelian groups
	$\cO(X)^\times\to  \Con(X)$ given by $u\mapsto [\sL_{u,d}]$. The kernel of this homomorphism is $K^\times\cO(X)^{\times d}$ and its image is $\Con(X)[d]$. \end{prop}
\begin{proof}  Suppose that $u_1,u_2\in \cO(X)^\times$. Then if, for $i=1,2$, $v_i$ generates $\sL_{u_i,d}$ with $\partial(v_i)=\frac{1}{d}\frac{\partial(u_i)}{u_i}$, $v_1\otimes v_2$ is a generator of $\sL_{u_1,d}\otimes \sL_{u_2,d}$ and \begin{eqnarray*} \partial(v_1\otimes v_2) & = &\partial(v_1)\otimes v_2+v_1\otimes\partial(v_2) \\ & = & \frac{1}{d}\left(\frac{\partial(u_1)}{u_1}+ \frac{\partial(u_2)}{u_2}\right)(v_1\otimes v_2) \\ & = & \frac{1}{d}\frac{\partial(u_1u_2)}{u_1u_2}(v_1\otimes v_2) \end{eqnarray*} for all $\partial\in \cT$. Thus $u\mapsto [\sL_{u,d}]$ does define a homomorphism.
	
	Note that $u\in \cO(X)^\times$ is in the kernel of the homomorphism if and only if $[\sL_{u,d}]=[\cO.v]=[\cO]$ in $\Con(X)$. That is, $u$ is in the kernel if and only if there is $w\in \cO(X)^\times$  such that $\partial(wv)=0$ for all $\partial\in \cT(X)$. But \[\partial(wv)=\left(\frac{\partial(w)}{w}+ \frac{1}{d}\frac{\partial(u)}{u}\right)wv. \] Now \[\frac{\partial(uw^{d})}{uw^d}=\frac{\partial(u)}{u}+ d\frac{\partial(w)}{w}\] so $\partial(wv)=0$ if and only if $\partial(uw^{d})=0$. Thus by Proposition \ref{GeomConGlobHorO}, the kernel is precisely $K^\times\cO(X)^{\times d}$ as claimed. Moreover $[\sL_{u,d}^{\otimes d}]=[\sL_{u^d,d}]=[\cO]$ so each $[\sL_{u,d}]$ is indeed $d$-torsion.

	Given $[\sL] \in \Con(X)[d]$, we use the hypothesis that $\sL$ is trivial as a line bundle to pick a generator $v \in \sL(X)$ as an $\cO(X)$-module, and we choose a $\cD$-linear isomorphism $\psi \colon \sL^{\otimes d} \stackrel{\cong}{\longrightarrow} \cO$ using the fact that $d \cdot [\sL] = 0$ in $\Con(X)$. We claim that if $\psi(v^{\otimes d})=u$ then $\partial(v)= \frac{1}{d}\frac{\partial(u)}{u}$ for all $\partial\in \cT$ and so $[\sL]=[\sL_{u,d}]$. For this we compute that $\partial(v^{\otimes d})=\frac{\partial(u)}{u}v^{\otimes d}$ for all $\partial\in \cT$. But if $\partial(v)=av$ then $\partial(v^{\otimes d})=dav^{\otimes d}$ so $a=\frac{1}{d}\frac{\partial(u)}{u}$ as claimed.
\end{proof}

\begin{defn} Suppose that $X$ is geometrically connected and that $\cT$ is generated as an $\cO$-module by its global sections.
 We define \[\theta_d\colon \Con(X)[d]\to \cO(X)^\times/K^\times\cO(X)^{\times d}\] to be the inverse of the isomorphism induced by the homomorphism $u \mapsto [\sL_{u,d}]$ in Proposition \ref{PicConDtor}.
\end{defn}
The proof of surjectivity in Proposition \ref{PicConDtor} shows that $\theta_d\left([\cO v]\right)$ is determined by the image of $v^{\otimes d}$ under any $\cD$-linear isomorphism $\psi\colon (\cO v)^{\otimes d}\stackrel{\cong}{\longrightarrow} \cO$, via
\begin{equation}\label{ExplthetaDef} \theta_d([\cO v]) = \psi(v^{\otimes d}) K^\times\cO(X)^{\times d}. \end{equation}
\begin{prop}\label{PicConDtorG} Suppose that $X$ is geometrically connected and that $\cT$ is generated as an $\cO$-module by its global sections. Let $G$ be a group acting on $X$. Then for every non-zero integer $d$, $\theta_d$ is a $G$-equivariant isomorphism \[  \theta_d\colon \Con(X)[d]\stackrel{\cong}{\longrightarrow} \cO(X)^\times/K^\times\cO(X)^{\times d}.\] 
\end{prop}
\begin{proof}
	Let $g \in G$, $[\sL] \in \Con(X)[d]$ and fix a $\cD$-linear isomorphism $\psi : \sL^{\otimes d} \stackrel{\cong}{\longrightarrow} \cO$. The map $g^{\cO} : \cO \to g^\ast \cO$ is a $\cD$-linear isomorphism, so $g^{\cO} \circ \psi : \sL^{\otimes d} \to g^\ast\cO$ is also a $\cD$-linear isomorphism. After identifying $g_\ast (\sL^{\otimes d})$ with $(g_\ast \sL)^{\otimes d}$ and $g_\ast(g^\ast\cO)$ with $\cO$, we obtain a $g_\ast \cD$-linear isomorphism
	\[ \psi' := g_\ast( g^{\cO} \circ \psi) : (g_\ast \sL)^{\otimes d}  \stackrel{\cong}{\longrightarrow}  \cO.\]
	Recall Lemma \ref{GactsPicCon} that $g \cdot [\sL] = [g_\ast \sL]$, where $\cD$ acts on $g_\ast \sL$ via the ring isomorphism $(g^{-1})^{\cD} : \cD \stackrel{\cong}{\longrightarrow} g_\ast \cD$. So $\psi'$ becomes an $\cD$-linear isomorphism in this way, and we can use it to compute $\theta_d([g_\ast \sL])$ as follows: let $v \in \sL(X)$ be such that $\sL(X) = \cO(X)v$; then by definition of $\psi'$ we have $\psi'(v^{\otimes d}) = g \cdot \psi(v^{\otimes d})$, so
	\[ \theta_d(g \cdot [\sL]) = \theta_d([g_\ast \sL]) = \psi'(v^{\otimes d})K^\times \cO(X)^{\times d} = g \cdot \psi(v^{\otimes d})K^\times \cO(X)^{\times d} = g \cdot \theta_d([\sL])\]
	as required.
\end{proof}
\begin{rmk} \label{groupoidactPicCon} Proposition \ref{PicConDtorG} can be viewed as saying that \[[g_\ast\sL_{u,d}]=[\sL_{g\cdot u,d}]\in \Con(X)[d].\] More generally if $U$ is an admissible open subset of $X$ and $u\in \cO(U)^\times$ then \[ [g_\ast\sL_{u,d}]=[\sL_{g\cdot u,d} ]\in \Con(g\cdot U).\] \end{rmk}

\subsection{Equivariant line bundles with flat connections}\label{PicConG}
We now turn to a discussion of \emph{equivariant} line bundles with flat connection. In this section we will assume that $G$ is a topological group acting continuously on a smooth rigid $K$-analytic space $X$ in the sense of \cite[Definition 3.1.8]{EqDCap}. We first consider $G$-equivariant line bundles. Our next definition, Definition \ref{GLB} below, will require some preparation. 

\begin{lem}\label{IntConLocFrech} Let $\sM$ be a coherent $\cO$-module on $X$. Suppose that $\{g^{\sM}\}_{g \in G}$ is a $G$-equivariant structure on $\sM$. Then for every affinoid subdomain $U$ of $X$ and $g\in G$, the structure map\[ g^\sM\colon \sM(U)\to \sM(gU) \] is continuous with respect to canonical $K$-Banach topologies on the domain and codomain.
\end{lem}
\begin{proof} Let $U$ be an affinoid subdomain of $X$. By Kiehl's Theorem --- see, e.g. \cite[Theorem 4.5.2]{FvdPut} --- $\sM(U)$ is a finitely generated module over the affinoid algebra $\cO(U)$ because $\sM_{|U}$ is a coherent $\cO_U$-module. Recalling from \cite[Proposition 2.1]{ST} that every finitely generated module $M$ over an affinoid algebra $A$ carries a canonical $K$-Banach-space topology, we see that $\sM(U)$ carries a canonical $K$-Banach space topology. Fixing $g \in G$, we can regard $\sM(gU)$ as an $\cO(U)$-module via a twisted action, by defining $a \ast v := (g\cdot a)v$ for all $a \in \cO(U)$ and $v \in \sM(gU)$. Then $\sM(gU)$ is still a finitely generated $\cO(U)$-module, and the structure map $g^{\sM}(U) : \sM(U) \to \sM(gU)$ is now a $\cO(U)$-linear homomorphism between two finitely generated $\cO(U)$-modules. It is therefore automatically continuous by \cite[Corollary 1.2.4]{FvdPut}.
\end{proof}

\begin{lem}\label{phizM} Let $\sM$ be a coherent $\cO_X$-module. Suppose that $\{g^\sM\}_{g\in G}$ is a $G$-equivariant structure on $\sM$ and that $L$ is a finite field extension of $K$. Then for every $z\in X(L)$ there is a natural group homomorphism \[\phi_{z,\sM}\colon G_z\to \Aut_L(\sM(z))\] where $\sM(z):=L\otimes_{\cO_{X,z}}\sM_z$ denotes the fibre of $\sM$ at $z$. 
\end{lem}
\begin{proof} Let $g \in G_z$. The $G$-equivariant structure on $\cO$ gives us a local $K$-algebra automorphism $g^{\cO}_z : \cO_{X,z} \to \cO_{X,z}$, whereas the $G$-equivariant structure on $\sM$ gives a $K$-linear automorphism $g^{\sM}_z : \sM_z \to \sM_z$, satisfying $g^{\sM}_z( a.m) = g^{\cO}_z(a) \cdot g^{\sM}_z(a)$ for all $a \in \cO_{X,z}$ and $m \in \sM_z$. It is now straightforward to check that setting 
\[g \cdot (\lambda \otimes m) := \lambda \otimes g^{\sM}_z(a) \qmb{for all} g \in G_z, \lambda \in L, m \in \sM_z\]
gives a well-defined $L$-linear action of $G_z$ on $L \otimes_{\cO_{X,z}} \sM_z$. \end{proof}
Suppose that $M$ is any $K$-Banach space. Then the $K$-algebra of bounded $K$-linear endomorphisms $\cB(M)$ is also a $K$-Banach algebra through the operator norm $||T||:=\sup\limits_{v\in V\backslash \{0\}} \frac{|Tv|}{|v|}$, so its unit group $\cB(M)^\times$ becomes a topological group --- using the geometric series, one can check that the inversion map on $\cB(M)^\times$ is continuous. If $\cM \subset M$ is the unit ball in $M$, then the congruence subgroups of $\cB(M)^\times$ 
\[\Gamma_n(\cM) := \{\gamma \in \cB(M)^\times :  (\gamma - 1)(\cM) \subseteq \pi_F^n \cM\}\]
form a fundamental system of open neighbourhoods of the identity in $\cB(M)^\times$.  Since any isomorphism of $K$-Banach spaces $M \stackrel{\cong}{\longrightarrow} N$ induces an isomorphism of topological groups $\cB(M)^\times \stackrel{\cong}{\longrightarrow} \cB(N)^\times$ via `conjugation', we see that the topology on $\cB(M)^\times$ only depends on the topology on $M$ and not on any particular choice of $K$-Banach norm on $M$.

Let $\sM$ be a coherent $\cO_X$-module and let $\{g^{\sM}\}_{g \in G}$ be a $G$-equivariant structure on $\sM$. For each affinoid subdomain $U$ of $X$ and each $g\in G_U$, the maps $g^{\sM}(U) : \sM(U) \to \sM(U)$ induce, by Lemma \ref{IntConLocFrech} a homomorphism
\[ G_U \to \cB(\sM(U))^\times.\]

\begin{defn}\label{GLB} A \emph{$G$-equivariant line bundle on $X$} is a $G$-equivariant $\cO_X$-module $\sL$ on $X$ such that \be \item $\sL$ is invertible as an $\cO_X$-module, and \item the action map $G_U\to \cB(\sL(U))^\times$ is continuous for every affinoid subdomain $U$ of $X$.\ee  \end{defn}

\begin{lem}\label{TrivGLB} $\cO_X$ with its usual equivariant structure is a $G$-equivariant line bundle on $X$.
\end{lem}
\begin{proof} Let $U$ be an affinoid subdomain of $X$. Consider the sup norm $|\cdot |_U$ on $\cO(U)$ whose unit ball in $\cA := \cO(U)^\circ$. For each $n \geq 0$, the congruence subgroup $\Gamma_n(\cO(U)^\circ)$ of $\cB(\cO(U))^\times$ contains the group $\cG_{\pi_F^n}(\cA)$ appearing on \cite[p. 19]{EqDCap}. Through the Raynaud generic fibre functor $\rig$, $\cG_{\pi_F^n}(\cA)$ can be identified with a subgroup of the group of $K$-linear automorphisms $\Aut_K(U, \cO_U)$ of the $G$-ringed topological space $(U, \cO_U)$. These subgroups of $\Aut_K(U, \cO_U)$ form a filter base for a certain topology on $\Aut_K(U,\cO_U)$ --- see \cite[Theorem 3.1.5]{EqDCap} --- and the action map $G_U \to \Aut_K(U, \cO_U)$ is continuous with respect to this topology by \cite[Definition 3.1.8]{EqDCap}, because $G$ is assumed to act continuously on $X$. It follows that the action map $G_U \to \cB(\cO(U))^\times$ is continuous as required.
\end{proof}

A \emph{morphism} between two $G$-equivariant line bundles $\sL$ and $\sM$ is a morphism of $G$-equivariant $\cO$-modules in the sense of \cite[Definition 2.3.1(c)]{EqDCap}. Given any such morphism $\varphi : \sL \to \sM$ and an affinoid subdomain $U$ of $X$, the map $\varphi(U) : \sL(U) \to \sM(U)$ is then an $\cO(U)$-linear homomorphism between two finitely generated $\cO(U)$-modules, and is therefore automatically continuous.

\begin{defn} \label{PicGX} We let $\Pic^G(X)$ denote the set of isomorphism classes of $G$-equivariant line bundles on $X$.\end{defn}

\begin{lem} \label{TensorHomPicG}  Let $\sL$ and $\sM$ be $G$-equivariant line bundles on $X$. Then so are $\sL \otimes_{\cO} \sM$ and $\sL^{\otimes -1} = \mathpzc{Hom}_{\cO}(\sL, \cO)$.\end{lem}
\begin{proof}
We can easily verify that the usual formula for tensor product and contragredient representations satisfies Definition \ref{GLB}.
\end{proof}
	
With respect to these operations $\Pic^G(X)$ is an abelian group with unit given by the structure $\{g^\cO\}$ on $\sL=\cO_X$. 

\begin{prop} \label{phiz} Suppose that $L$ is a finite extension of $K$ and $z\in X(L)$. There is a natural group homomorphism $\phi_z\colon \Pic^G(X)\to \Hom(G_z,L^\times)$ given by \[ \phi_z([\sL])=\phi_{z,\sL}.\] 
\end{prop}

\begin{proof} We note that if $\sL$ is a line bundle on $X$ with a $G$-equivariant structure then $\sL(z)$ is a one-dimensional vector space over the residue field of the local ring $\cO_{X,z}$, and so, by Lemma \ref{phizM}, $\phi_{z,\sL}$ can be viewed as a homomorphism $G_z\to L^\times$.  

Next we show that $\phi_{z,\sL}$ is continuous. To this end we choose an affinoid subdomain $U$ of $X$ such that $z\colon \Sp(L)\to X$ factors through $U$. Since $G_z\cap G_U$ is an open subgroup of $G_z$ it suffices to show that $\phi_{z,\sL}|_{G_U\cap G_z}$ is continuous. Since the natural map $\cB(\sL(U))^\times \to L^\times$ is continuous this follows from Definition \ref{GLB}(b).

It remains to show that $\phi_z$ is a group homomorphism. This is immediate, because whenever $\sL_1$ and $\sL_2$ are elements of $\Pic^G(X)$, there is a canonical isomorphism \[(\sL_1\otimes_{\cO} \sL_2)(z)\to \sL_1(z)\otimes_L \sL_2(z)\] which is compatible with the $G$-actions. 
\end{proof}

This discussion leads us on to the following definition.

\begin{defn}\label{GLBFC} A \emph{$G$-equivariant line bundle with flat connection on $X$} is a $G$-equivariant $\cD_X$-module $\sL$, such that when $\sL$ is viewed as a $G$-equivariant $\cO_X$-module by restriction, it is a $G$-equivariant line bundle on $X$.\end{defn}

It follows easily from Lemma \ref{TrivGLB} that $\cO_X$ equipped with the trivial connection and its usual equivariant structure is a $G$-equivariant line bundle with connection. 

We can also typically put other $G$-equivariant structures on the trivial line bundle with trivial connection by considering $\Hom(G,K^\times)$, the abelian group of \emph{continuous} group homomorphisms from $G$ to $K^\times$. 

\begin{defn}\label{defnOchi} Given $\chi\in \Hom(G,K^\times)$ we define a new $G$-equivariant structure on $\cO$ equipped with the trivial connection: for each affinoid subdomain $U$ of $X$ and for each $g \in G$ we define $K$-linear continuous maps
	\[ g^{\cO_\chi}(U) \colon \cO(U) \longrightarrow \cO(gU)\]
	by $f \mapsto \chi(g) g^{\cO}(f)$.
\end{defn}
This family $\{g^{\cO_\chi}\}$ then defines a $G$-equivariant line bundle with connection on $X$ in the sense of Definition \ref{GLBFC} that we will denote by $\cO_\chi$. Note that condition (b) in Definition \ref{GLB} follows from the assumption that $\chi : G \to K^\times$ is continuous.

\begin{lem}\label{TensorHomPicConG} Let $\sL$ and $\sM$ be $G$-equivariant line bundles with flat connection. Then so are $\sL \otimes_{\cO} \sM$ and $\sL^{\otimes -1} = \mathpzc{Hom}_{\cO}(\sL, \cO)$.
\end{lem}
\begin{proof}  We saw at the start of $\S \ref{ELBsec}$ that $\sL \otimes_{\cO} \sM$ and $\sL^{\otimes -1}$ are $\cD$-modules on $X$. The usual formula for the tensor product and contragredient representations allow us to see that they also carry standard $G$-equivariant $\cD$-module structures and Lemma \ref{TensorHomPicG} shows that this makes them $G$-equivariant line bundles.\end{proof}

\begin{defn} We denote the set of isomorphism classes of $G$-equivariant line bundles with flat connection on $X$ by $\PicCon^G(X)$.
\end{defn}

In view of Lemma \ref{TensorHomPicConG}, the operations $-\otimes_{\cO}-$ and $(-)^{\otimes -1}$ endow $\PicCon^G(X)$ with the structure of an abelian group. The unit element in this group is given by the isomorphism class of $\cO_X$ equipped with the trivial connection together with its usual $G$-equivariant structure.

\begin{defn}\label{ConGdefn} We define $\Con^G(X)$ by 
	\[\Con^G(X) := \ker(\PicCon^G(X) \to \Pic(X)),\]
	the group of isomorphism classes of $G$-equivariant line bundles with flat connection on $X$ that are trivial after forgetting the connection and the $G$-action. 
\end{defn}

We record that $\Con$ is functorial in a natural way. For $g\in G$ we write $c_g\colon G\to G$ to denote conjugation by $g$; $c_g(x)=gxg^{-1}$. 

\begin{prop}\label{GactsPicConeq} Let $U$ be a geometrically connected admissible open subset of $X$, and suppose that $H$ is a closed subgroup of $G_U$. Then for each $g\in G$, \be \item the map $[\sL]\mapsto [g_\ast\sL]$ induces a natural isomorphism 
\[g\colon \PicCon^H(U)\quad \stackrel{\cong}{\longrightarrow} \quad\PicCon^{^{g}H}(g U);\] \item for $L$ a finite extension of $K$ and $z\in U(L)$, the following diagram commutes: \[ \xymatrix{ \PicCon^H(U) \ar[r]^{\phi_z} \ar[d]_g & \Hom(H_z, L^\times) \\  \PicCon^{^gH}(gU) \ar[r]_{\phi_{g\cdot z}}  & \Hom({}^g H_z, L^\times).\ar[u]_{c_g^\ast}  }\]  \ee\end{prop}

\begin{proof} (a) We've already seen in Remark \ref{GactsPicConRem} that $\sL\mapsto g_\ast\sL$ induces an isomorphism $\PicCon(U)\to \PicCon(gU)$. 
	
	It remains to see how the $H$-equivariant structure on $\sL\in \PicCon^H(U)$ induces an ${}^gH$-equivariant structure on $g_{\ast} \sL$: for each $h\in H$ we have an isomorphism $h^\sL\colon \sL\to h^\ast \sL$. This induces an isomorphism \[(ghg^{-1})^{g_\ast\sL}\colon g_\ast\sL\to (ghg^{-1})^{\ast}g_\ast\sL=g_\ast h^\ast\sL\] given by $(ghg^{-1})^{g_\ast \sL}=g_\ast(h^\sL)$. It is easy to verify that this induces the desired isomorphism $\PicCon^H(U)\to \PicCon^{{}^gH}(gU)$. 
	
	(b) Fix $[\sL] \in \PicCon^H(U)$ and consider the stalk $(g_\ast \sL)_{g \cdot z}$ of $g_\ast \sL$ at $g \cdot z \in gU$. There is a natural bijection between the affinoid subdomains of $gU$ containing $g\cdot z$, and the affinoid subdomains of $U$ containing $z$, given by $V \mapsto g^{-1}V$. This gives a $K$-linear isomorphism $\tau_g : \sL_z \to (g_\ast \sL)_{g \cdot z}$ which is appropriately equivariant with respect to the $H_z$-action on $\sL_z$ and the $H_{g\cdot z} = {}^g H_z$-action on $(g_\ast \sL)_{g \cdot z}$:
\[ \tau_g( h \cdot m ) = c_g(h) \cdot \tau_g(m) \qmb{for all} h \in H_z, m \in \sL_z.\]
Now let $h \in H_z$. Using Lemma \ref{phizM}, we see that the scalar $\phi_z([\sL])(h) \in L^\times$ is completely determined by the following equation inside $\sL(z)$:
	\[ 1 \otimes h \cdot m = \phi_z([\sL])(h) \otimes m \qmb{for all} m \in \sL_z.\]
Since $c_g(h) = ghg^{-1}$ lies in $H_{g\cdot z}$, we have a similar equation inside $(g_\ast \sL)(g \cdot z)$:
\[ 1 \otimes c_g(h) \cdot m = \phi_{g\cdot z}([g_\ast \sL])(c_g(h)) \otimes m  \qmb{for all} m \in (g_\ast \sL)_{g \cdot z}. \]
Note that the map $\tau_g$ satisfies $\tau_g(a \cdot m) = g^{\cO}_z(a) \cdot \tau_g(m)$ for all $a \in \cO_{X,z}, m \in \sL_z$. Therefore $1 \otimes \tau_g : L \otimes_K \sL_z \to L \otimes_K (g_\ast \sL)_{g \cdot z}$ descends to a well-defined $L$-linear map $\sL(z) \to (g_\ast \sL)(g \cdot z)$. Applying this map to the first equation and comparing the result with the second shows that
\[ \phi_z([\sL])(h) = \phi_{g \cdot z}([g_\ast \sL])(c_g(h)) = (c_g^\ast \circ \phi_{g\cdot z} \circ g)([\sL])(h)\qmb{for all} h \in H_z.\]
This implies the commutativity of the diagram in the statement.\end{proof} 

  Forgetting the $G$-equivariant structure gives us a functor $\underline{\omega}$ from the category of $G$-equivariant line bundles with flat connection on $X$ and isomorphisms between them to the category of line bundles with flat connection on $X$ and isomorphisms between them. Moreover $\underline{\omega}$ induces a group homomorphism \[\omega\colon \PicCon^G(X) \longrightarrow \PicCon(X).\]
Recall\footnote{see \cite[Definition 2.3]{Kiehl}} that $X$ is said to be \emph{quasi-Stein} if there is an admissible affinoid covering $(X_n)_{n=0}^\infty$ of $X$ with $X_0 \subseteq X_1 \subseteq X_2 \subseteq \cdots $ such that for each $n \geq 0$, the restriction map $\cO(X_{n+1}) \to \cO(X_n)$ has dense image. If $X$ is quasi-Stein, then the global sections functor $\Gamma(X,-)$ gives a fully faithful embedding from the category of coherent $\cO$-modules on $X$ into the category $\cO(X)$-modules; the essential image is the category of \emph{coadmissible $\cO(X)$-modules} in the sense of \cite{ST}. In particular, every coherent $\cO$-module (such as the tangent sheaf $\cT$) is generated by its global sections.

\begin{prop}\label{PicSeq} Suppose that $X$ is geometrically connected and quasi-Stein. There is an exact sequence of abelian groups
	\[ 0 \to \Hom(G, K^\times) \to \PicCon^G(X) \stackrel{\omega}{\to} \PicCon(X)^G\]
	with the first non-trivial map given by $\chi\mapsto \cO_\chi$.
\end{prop}
\begin{proof} It is easy to verify that $\chi\mapsto \cO_\chi$ defines a group homomorphism from  $\Hom(G,K^\times)$ to $\PicCon^G(X)$. Using Lemma \ref{GlobHorSec}(b), we see that for any $G$-equivariant line bundle with flat connection $\sL$, $\Hom_{\cD}(\cO, \sL) \cong \sL(X)^{\cT(X) =0}$ is a $K$-linear $G$-representation: if $g \in G$ and $\cT(X) \cdot v = 0$ for some $v \in \sL(X)$, then $\partial \cdot (g \cdot v) = g \cdot ((g^{-1} \cdot \partial) \cdot v) = 0$ for all $\partial \in \cT(X)$ so that $g \cdot v \in \sL(X)^{\cT(X) = 0}$ again. 
	
	Suppose that $\chi \in \Hom(G, K^\times)$ is such that $\cO_\chi$ is isomorphic to $\cO$ as a $G$-equivariant line bundle with flat connection. Considering the global horizontal sections, we obtain an isomorphism of continuous $G$-representations
	\[\cO(X)^{\cT(X) = 0} \quad \cong \quad \cO_\chi(X)^{\cT(X) = 0}.\]
	By Proposition \ref{GeomConGlobHorO}, both of these $K$-vector spaces are $1$-dimensional and spanned by $1 \in \cO(X)$. However the $G$-action on the first is trivial, whereas the $G$-action on the second is through the character $\chi$. Hence $\chi$ is the trivial character, and the map $\Hom(G, K^\times) \to \PicCon^G(X)$ is injective.
	
	 If $\sL$ is $G$-equivariant, then $(g^{-1})^{\sL} : \sL \to g_\ast \sL$ is a $\cD$-linear isomorphism for all $g \in G$ which means that the class $[\sL]$ in $\PicCon(X)$ is fixed by this $G$-action i.e. the image of the map $\PicCon^G(X)\to \PicCon(X)$ is indeed contained in $\PicCon(X)^G$. 
	 
	Finally suppose that $\sL$ is a $G$-equivariant line bundle with flat connection on $X$ which becomes trivial after forgetting the $G$-structure. Then we can find a $\cD$-linear isomorphism $\varphi : \cO \stackrel{\cong}{\longrightarrow} \sL$. Applying the functor of global horizontal sections,  we deduce from Proposition \ref{GeomConGlobHorO} that the $K$-linear $G$-representation $\sL(X)^{\cT(X) = 0}$ is in fact one-dimensional. Let $v := \varphi(X)(1) \in \sL(X)^{\cT(X)=0}$; since $1 \in \cO(X)$ generates $\cO$ as an $\cO$-module, we see that $v \in \sL(X)$ generates $\sL$ as an $\cO$-module: $\sL = \cO \cdot v$. Let $\chi \in \Hom(G, K^\times)$ describe the $G$-action on $v$, so that $g \cdot v = \chi(g) v$ for all $g \in G$. This gives us a character $\chi : G \to K^\times$;  we will first show that $\chi$ is continuous. To see this, choose a non-empty affinoid subdomain $U$ of $X$ and let $\cA$ be an affine formal model in $\cO(U)$, so that $\cA \cdot v_{|U}$ is the unit ball in $\sL(U) = \cO(U) \cdot v_{|U}$ with respect to some $K$-Banach norm defining the canonical topology on $\sL(U)$. By Definition \ref{GLB}(b), for any $n \geq 0$ we can find an open subgroup $H_n$ of $G_U$ such that $h \cdot v_{|U} \equiv v_{|U} \mod \pi_F^n \cA \cdot v_{|U}$ for all $h \in H_n$. Hence $\chi(h) \equiv 1 \mod \pi_F^n K^\circ$ for all $h \in H_n$, and $\chi : G \to K^\times$ is therefore continuous as claimed.
	
	With $\chi \in \Hom(G,K^\times)$ in hand, consider the $G$-equivariant line bundle with connection $\sL \otimes_{\cO} \cO_{\chi^{-1}},$ and the $\cD$-linear isomorphism $\varphi_\chi : \cO \to \sL \otimes_{\cO} \cO_{\chi^{-1}}$ given by $\varphi_\chi(f) : =\varphi(f)\otimes 1$. We claim that $\varphi_\chi$  is $G$-equivariant; given this claim, it follows immediately that $\sL \cong \cO_\chi$ as a $G$-equivariant $\cD$-module, and then $[\sL] \in \PicCon^G(X)$ lies in the image of $\Hom(G, K^\times)$ as required.
	
	To establish the claim, we first check that $\varphi_\chi(X) : \cO(X) \to (\sL \otimes \cO_{\chi^{-1}})(X)$ is $G$-equivariant: for $g \in G$ and $f \in \cO(X)$ we have
	\[g\cdot\varphi_\chi(X)(f) = g\cdot(f v\otimes 1) =(g\cdot f)(g\cdot v)\otimes \chi^{-1}(g) = (g\cdot f)v\otimes 1= \varphi_\chi(X)(g\cdot f).\] 
	By replacing $\sL$ by $\sL \otimes \cO_{\chi^{-1}}$, it now remains to show that if $\varphi : \cO \to \sL$ is a $\cD$-linear isomorphism such that $\varphi(X) : \cO(X) \to \sL(X)$ is $G$-equivariant, then $\varphi$ is also $G$-equivariant. To this end, fix $g \in G$, and consider the morphisms $g^\ast(\varphi) \circ g^{\cO} : \cO \to g^\ast \sL$ and $g^{\sL} \circ \varphi : \cO \to g^\ast \sL$. By precomposing the $g^\ast \cO$-module structure on $g^\ast \sL$ with the ring isomorphism $g^{\cO} : \cO \to g^\ast \cO$ we may regard $g^\ast \sL$ to be an $\cO$-module. Then it is coherent, and the two maps are $\cO$-linear. Since they have the same global sections by our computation above and since $X$ is \emph{quasi-Stein}, we conclude using \cite[Corollary 3.3]{ST} that the two maps are equal. This means that $\varphi$ is $G$-equivariant.
\end{proof}
In order to understand $\PicCon^{G^0}(\Omega)$ the following result will prove to be useful.

\begin{prop}\label{PicConamal}Suppose that  $X$ is quasi-Stein and geometrically connected and $G$ is equal to an amalgamated product $A \ast_CB$ of its open subgroups $A$ and $B$ along their common subgroup $C$. Then the homomorphism
	\[ (p_A,p_B)\colon \PicCon^G(X)\to \PicCon^A(X)\underset{\PicCon^C(X)}{\times}{}\PicCon^B(X) \] induced by restriction of equivariant structures is an isomorphism.
\end{prop}
\begin{proof} Suppose that $[\sL]\in \PicCon^G(X)$ lies in the kernel of $(p_A,p_B)$. Then \[\omega([\sL])=[\cO]\in \PicCon(X).\] By Proposition \ref{PicSeq}, $[\sL]=[\cO^\chi]$ for some $\chi\in \Hom(G,K^\times)$. However, also by Proposition \ref{PicSeq}, $\chi|_A$ and $\chi|_B$ are both trivial and so, as $G$ is generated by $A$ and $B$, $\chi$ is trivial. Thus $[\sL]=[\cO]\in \PicCon^G(X)$ and $(p_A,p_B)$ is injective. 
	
	Suppose now that $[\sL_A]\in \PicCon^A(X)$, $[\sL_B]\in \PicCon^B(X)$ and that there is an isomorphism $\theta\colon \sL_A|_C\stackrel{\cong}{\longrightarrow} \sL_B|_C$ of $C$-equivariant line bundles with flat connection obtained by restriction of equivariant structures from $A$ and $B$ to $C$. We transport the $B$-equivariant structure on $\sL_B$ along $\theta$ to $\sL_A$. In this way, the $\cD$-module $\sL := \underline{\omega}(\sL_A)$ can be equipped with $A$-equivariant and $B$-equivariant structures whose restrictions to $C$ agree.
	
	 	  
By Proposition \ref{GAmodStrAsSections}, for every subgroup $H$ of $G$ there is a bijection between the set of all $H-\cD$-module structures on $\sL$ extending the given $\cD$-module structure on $\sL$, and the set $\cS_{\cD}(H,\sL)$. This bijection is given by
	\[ \{g^\sL\}_{g\in H}\mapsto [g\mapsto (g,g^{\sL})\in \Aut_{\cD}(\sL/X/H)]. \]
On the other hand, by Theorem \ref{amalgequiv} restriction induces a bijection 
	\[ \cS_{\cD}(G,\sL)\stackrel{\cong}{\longrightarrow} \cS_{\cD}(A,\sL)\times_{\cS_{\cD}(C,\sL)}\cS_{\cD}(B,\sL).\]
It follows that there exists a $G-\cD$-module $\sL_G = (\sL, \{g^{\sL}\}_{g \in G})$ whose restriction to $A$ is $\sL_A$, and whose restriction to $B$ is the transport of $\sL_B$ to $\sL_A$ along $\theta$. Since $A$ and $B$ are open in $G$, the action map $G_U\to \cB(\sL(U))^{\times}$ is continuous for every affinoid subdomain $U$ of $X$, because the restrictions of this map to both $A_U$ and $B_U$ are continuous. This shows that $[\sL_G]\in \PicCon^G(X)$. By construction, we have $p_A([\sL_G]) = [\sL_A]$ and $p_B([\sL_G]) = [\sL_B]$. Hence $(p_A,p_B)$ is surjective. \end{proof}




\begin{lem}\label{PicConGlim} If $X$ is a geometrically connected quasi-Stein space with an increasing admissible covering $(X_n)$ by $G$-stable affinoid subdomains, then the restriction maps $\PicCon^G(X)\to \PicCon^G(X_n)$ induce an isomorphism of groups \[ \PicCon^G(X)\cong \invlim \PicCon^G(X_n). \] 
\end{lem}

\begin{proof}
	Certainly the restriction maps induce a group homomorphism \[\alpha\colon \PicCon^G(X)\to \invlim \PicCon^G(X_n).\]  
	
	By Proposition \ref{PicConinvlim}, if $[\sL]\in \ker\alpha$ then $\omega([\sL])=[\cO_X]\in \PicCon(X)$. Thus $[\sL]=[\cO_X^\chi]$ for some continuous character $\chi\colon G\to K^\times$ by Proposition \ref{PicSeq}. But then $[\sL|_{X_n}]=[\cO_{X_n}^\chi]$ for any $n\geq 0$ and so $\chi$ is the trivial character since $\alpha([\sL])$ is trivial. Thus $\alpha$ is injective. 
	
	Let $([\sL_n])\in \prod \PicCon^{G}(X_n)$ be a compatible family of isomorphism classes of equivariant line bundles with connection under restriction so that for each $n$ we can find an isomorphism \[\theta_{n+1,n}\colon \sL_{n+1}|_{X_{n}}\to \sL_{n}\] of $G$-equivariant line bundles with connection on $X_n$. 
	
	Thus each $l\geq m$, we can define $\theta_{m,l}\colon \sL_l|_{X_m}\to \sL_{m}$ to be the composite of the restrictions of $\theta_{l,l-1},\theta_{l-1,l-2},\ldots, \theta_{m+1,m}$ to $X_m$. Now $(\sL_l,\theta_{m,l})$ forms gluing data for the cover $(X_m)_{m\geq n}$ of $X$. The resulting sheaf $\sL$ is an $G$-equivariant line bundle with connection on $X$ with $\alpha([\sL])=([\sL_n])$ and so $\alpha$ is surjective. 
\end{proof}

\subsection{Cocycles and equivariant line bundles on affinoids}\label{SecCocycELBC}
In this technical subsection, we will explain how isomorphism classes of $G$-equivariant structures on the trivial line bundle over a $K$-affinoid variety $X$ can be classified through the language of continuous $1$-cocycles of $G$ acting on $\cO(X)^\times$. This material will be crucial to the proof of one of the main results in $\S \ref{MainProofSec}$, namely Theorem \ref{phizisolocal}. 

We assume throughout $\S \ref{SecCocycELBC}$ that $X$ is a smooth and geometrically connected $K$-\textbf{affinoid} space, with a topological group $G$ acting continuously on it. 

\begin{lem} \label{GOZ1} \hfill
	\be
	\item The set of $G$-equivariant structures on a trivial line bundle $\sL=\cO\cdot v$ is in natural bijection with $Z^1(G,\cO(X)^\times)$ under a function that sends $\{g^{\sL}:g\in G\}$ to the function $\alpha : G \to \cO(X)^\times$ determined by the rule \[g^\sL(v)=\alpha(g)v \qmb{for all} g\in G.\] 
    \item The bijection in (a) induces an isomorphism 
    \[  \theta^G_X : \ker\left(\Pic^G(X)\to \Pic(X)\right)\stackrel{\cong}{\longrightarrow} H^1(G,\cO(X)^\times).\] \ee
\end{lem}

\begin{proof} (a) Suppose that for each $g \in G$, we have a morphism of sheaves of $K$-vector spaces $g^{\sL} : \sL \to g^\ast \sL$ such that
\[g^\sL(fv)=g^{\cO}(f)g^{\sL}(v) \qmb{for all}g\in G, f\in \cO.\] 
This data is completely determined by the function  $\alpha\colon G\to \cO(X)$ given by \[g^\sL(v)=\alpha(g)v \qmb{for all} g \in G.\]We first claim that $\{g^\sL:g\in G\}$ is a $G$-equivariant structure on $\sL$ if and only only $\alpha$ is a $1$-cocycle with values in the group $\cO(X)^\times$.
	
We see that for all $g,h\in G$  \begin{eqnarray}  \label{alphagalphah} h^{\ast}(g^\sL)h^{\sL}(v)=h^{\ast}(g^\sL)(\alpha(h)v) & = & (g\cdot \alpha(h))\alpha(g)v \mbox{ and} \\  (gh)^{\sL}(v) & = & \alpha(gh)v. \label{alphagh} \end{eqnarray}

Now if  $\{g^\sL:g\in G\}$ defines a $G$-equivariant structure on $\sL$ then $\alpha(1)=1$ and $h^{\ast}(g^\sL)h^{\sL}(v)=(gh)^{\sL}(v)$. Thus  by (\ref{alphagalphah}) and (\ref{alphagh}) $\alpha(gh)=(g\cdot \alpha(h))\alpha(g)$ for all $g,h\in G$ and in particular \[ 1=\alpha(gg^{-1})=(g\cdot \alpha(g^{-1}))\alpha(g)\] for all $g\in G$. Thus $\alpha$ is a $1$-cocycle with values in $\cO(X)^\times$. Conversely, if $\alpha$ is a $1$-cocycle then by (\ref{alphagalphah}) and (\ref{alphagh}) again, for all $g,h\in G$, \[ (gh)^{\sL}(v)=h^{\ast}(g^\sL)h^{\sL}(v) \mbox{ and so }\]  \[(gh)^{\sL}=h^{\ast}(g^\sL)h^{\sL}.\] Moreover $\alpha(1^2)=\alpha(1)^2$ and so, since $X$ is connected, $\alpha(1)=1$ and $1^\sL=\mathrm{id}_{\sL}$. 
 
 It remains to observe that $\alpha$ is continuous if and only if  for every affinoid subdomain $U$ of $X$, the action map $G_U\to \cB(\sL(U))^\times$ is continuous. This holds because \[ g\cdot(fv)=g^\cO(f)\alpha(g)v \qmb{ for all }g\in G_U, f\in \cO(U)\] and because $G$ acts continuously on $X$.
 
 (b) Suppose now that $\sL_1=\cO\cdot v_1$ and $\sL_2=\cO\cdot v_2$ are two $G$-equivariant line bundles corresponding to $1$-cocycles $\alpha_1$ and $\alpha_2$ respectively.
 
Let $\varphi \colon \sL_1\to \sL_2$ be an isomorphism of the underlying line bundles. Then $\varphi(v_1)=f v_2$ for some $f\in \cO(X)^\times$, so for all $g\in G$ we have \[\varphi(g^{\sL_1}(v_1))= \varphi(\alpha_1(g)v_1)=\alpha_1(g) fv_2 \] whereas
  \[ g^{\sL_2}(\varphi(v_1))=g^{\sL_2}(fv_2)=g^{\cO}(f)\alpha_2(g)v_2. \] 
Hence $\varphi$ defines an isomorphism of $G$-\emph{equivariant} line bundles if and only if   \[\alpha_2(g)= \frac{g^{\cO}(f)}{f} \alpha_1(g)\qmb{ for all }g\in G.\]  Thus the map in (a) induces a bijection \[\theta^G_X\colon \ker \left(\Pic^G(X)\to \Pic(X)\right)\to H^1(G,\cO(X)^\times)\]which is a group homomorphism because $g\cdot (v_1\otimes v_2)= \alpha_1(g)v_1\otimes \alpha_2(g)v_2=(\alpha_1\alpha_2)(g)(v_1\otimes v_2)$ for all $g\in G$. 
\end{proof}

Recall the map $\phi_z$ from Proposition \ref{phiz}; by abuse of notation, we will also denote its pre-composition with the forgetful map $\Con^G(X) \to \Pic^G(X)$ by $\phi_z$.

\begin{prop}\label{defphiz} Suppose that $L$ is a finite field extension of $K$.
	 \be \item The isomorphism from Lemma \ref{GOZ1}(b) induces a homomorphism \[\phi^G_X\colon \Con^{G}(X)\to H^1(G,\cO(X)^\times)\] by pre-composition with the forgetful map \[\Con^G(X)\to \ker (\Pic^G(X)\to \Pic(X)).\]
		\item For every $z\in X(L)$ and every $[\sL] \in \Con^G(X)$, we have 
\[z \circ (\res^G_{G_z} \phi^G_X([\sL]) ) = \phi_z([\sL]).\]
 
\item For every $z\in X(L)$ and every $\chi\in \Hom(G,K^\times)$, we have \[\phi_z([\cO_\chi])=\chi|_{G_z}.\]
\item Let $Y \subseteq X$ be a $G$-stable affinoid subdomain, with $z\in Y(L)\subseteq X(L)$. Then the following diagram is commutative: 
\[ \xymatrix{ \Con^G(X) \ar[drr]^{\phi_{z}} \ar[d]_{ (-)|_Y }& & \\ \Con^G(Y) \ar[rr]_{\phi_{z}} && \Hom(G_z,L^\times).}\]
  \ee 
\end{prop}

\begin{proof}
	
	(a) Since the forgetful map $\Con^G(X)\to \ker(\Pic^G(X)\to \Pic(X))$ is a group homomorphism, the function $\phi^G_X$ that is the composite of this forgetful map with $\theta^G_X$ is also a homomorphism.
	
	(b) Suppose that $\phi^G_X([\sL])=[\alpha]\in H^1(G,\cO(X)^\times)$ so that we can write $\sL=\cO\cdot v$ with $g\cdot v=\alpha(g)v$ for all $g\in G$.  Then working inside $L\otimes_{\cO(X)}\sL(X)$ we have \[ \phi_z(g)(1\otimes v)= g\cdot (1\otimes v)=1\otimes \alpha(g)v=(z\circ \alpha)(g)\otimes v \qmb{for all} g\in G_z.\]
	

	(c) This follows from Lemma \ref{GOZ1}(a) and Definition \ref{defnOchi}. 
	
	(d) If $[\sL]=[\cO_X\cdot v]\in \Con^G(X)$, then $[\sL|_{Y}]=[\cO_Y\cdot v]\in \Con^G(Y)$ and $g^{\sL|_Y}(v)=g^{\sL}(v)$ for all $g\in G$. Now use Lemma \ref{GOZ1} together with part (a).
\end{proof}
	 

\begin{defn}\label{ZGXude} Recall from $\S \ref{ConvNotn}$ the map $\delta_G\colon \cO(X)^\times\to Z^1(G,\cO(X)^\times)$, given by $\delta_G(u)(g)=g\cdot u/u$. For each $u\in \cO(X)^\times$ and $d,e\geq 1$, we define
	\[\mathcal{Z}^{G,X}_{u,d,e}:=\left\{ \alpha\in Z^1(G,\cO(X)^\times): \alpha^{de}=\delta_G(u^e)\right\} .\]
\end{defn}
This special set of $1$-cocycles will be useful for our explicit construction of torsion equivariant line bundles with flat connection.

Recall from Lemma \ref{Lud} the line bundle with connection $\sL_{u,d}$: it is the free $\cO_X$-module on the $1$-element set $\{v\}$, and the action of $\cT(X)$ is determined by
\begin{equation}\label{Ludform} \partial(v)=\frac{1}{d}\frac{\partial(u)}{u}v \qmb{for all} \partial\in \cT(X).\end{equation}

\begin{lem}\label{1cocycles} Let $u\in \cO(X)^\times$ and let $d,e\geq 1$. \be\item For each $\alpha \in \cZ^{G,X}_{u,d,e}$, there is a $(de)$-torsion $G$-equivariant line bundle with connection $\sL^\alpha_{u,d}$ on $X$ such that \[\omega([\sL^\alpha_{u,d}]) = \sL_{u,d}\qmb{and} \phi_X^G([\sL^\alpha_{u,d}])=[\alpha]\in H^1(G,\cO(X)^\times).\] 
	\item If $u,w\in \cO(X)^\times$, $\alpha\in \mathcal{Z}^{G,X}_{u,d,e}$ and $\beta\in \mathcal{Z}^{G,X}_{w,d,e}$, then $\alpha\beta\in \mathcal{Z}^{G,X}_{uw,d,e}$ and 
	\[\sL_{u,d}^\alpha\otimes \sL_{w,d}^\beta\cong \sL_{uw,d}^{\alpha\beta}.\]
	\item For each $\alpha \in \mathcal{Z}^{G,X}_{u,d,e}$, the following are equivalent:
	\begin{enumerate}[{(}i{)}]\item $[\sL_{u,d}^\alpha] = [\cO]$ in $\Con^G(X)[de]$,
		\item there is $f\in \cO(X)^\times$ such that $u/f^d\in K^\times$ and $\alpha=\delta_G(f)$.
	\end{enumerate}
	\item  The map $\alpha \mapsto [\sL_{u,d}^\alpha]$ defines a bijection
	\[\cZ^{G,X}_{u,d,e} \stackrel{\cong}{\longrightarrow} \left\{[\sL]\in \Con^G(X)[de]:\omega([\sL])=[\sL_{u,d}] \right\}.\]
	\ee\end{lem}
\begin{proof} (a) We equip $\sL_{u,d}$ with the $G-\cO_X$-module structure associated to $\alpha$ by Lemma \ref{GOZ1}. In particular the action map $G_U \to \cB(\sL(U))^\times$ is continuous for every affinoid subdomain $U$ of $X$. We check that this is in fact a $G-\cD_X$-module. Using $\alpha^{de}=\delta_G(u^e)$, we compute
	\[ \frac{\partial(\alpha(g))}{\alpha(g)}= \frac{1}{d}\left(\frac{\partial(g\cdot u)}{g\cdot u}- \frac{\partial(u)}{u}\right) \qmb{for each} \partial\in \cT_X, g\in G.  \]
	Using this together with (\ref{Ludform}) we compute \begin{eqnarray*} (g\cdot \partial)(g\cdot v) & = & (g\cdot \partial)(\alpha(g)) v + \alpha(g) (g\cdot\partial)(v) \\
		& = & \frac{1}{d}\left(\frac{(g\cdot \partial)(g\cdot u)}{g\cdot u}- \frac{(g\cdot\partial) u}{u} \right)\alpha(g)v + \alpha(g)\frac{1}{d}\frac{(g\cdot \partial)(u)}{u} v\\ & = & \frac{1}{d}\frac{g\cdot (\partial(u))}{g\cdot u}g\cdot v \\ & = & g\cdot (\partial (v)). \end{eqnarray*}	
	Thus we obtain a $G$-equivariant line bundle with connection on $X$ that we denote $\sL^\alpha_{u,d}$. It is evident that $\omega\left([\sL_{u,d}^\alpha]\right)=[\sL_{u,d}]$.
	
	Using (\ref{Ludform}), we see that $v^{\otimes de}\mapsto u^e$ defines an isomorphism $\psi : \sL_{u,d}^{\otimes de}\stackrel{\cong}{\longrightarrow} \cO_X$ of line bundles with flat connection. To establish that $\sL_{u,d}^\alpha$ is $(de)$-torsion, it suffices to show this isomorphism is $G$-linear. Since for all $g \in G$ we have
	\[g \cdot (v^{\otimes de})=\alpha^{de}(g)v^{\otimes de}\qmb{and} g\cdot (u^e)=\delta_G(u^e)(g)u^e;\]
	 we have $\alpha^{de}=\delta_G(u)^e$ and the $G$-linearity follows.
	
	(b) Since $\alpha^{de}=\delta_G(u^e)$ and $\beta^{de}=\delta_G(w^e)$ and $\delta_G\colon \cO(X)^\times\to Z^1(G,\cO(X)^\times)$ is a group homomorphism, we see that $(\alpha\beta)^e=\delta_G((uw)^e)$. Hence $\alpha\beta\in \mathcal{Z}_{uw,d,e}^{G,X}$.  
	
	Using Proposition \ref{PicConDtor}, we see that $[\sL_{u,d}\otimes \sL_{v,d}]=[\sL_{uv,d}]$ in $\Con(X)[d]$. Using the definition of the $G$-equivariant structure on $\sL_{u,d}^\alpha$ in part (a) and the definition of the product in $\Con^G(X)$ given in the proof of Lemma \ref{TensorHomPicConG}, we also see that $[\sL_{u,d}^\alpha\otimes \sL_{w,d}^\beta] = [\sL_{uw,d}^{\alpha\beta}]$ in $\Con^G(X)$.
	
	(c) For any $f \in \cO(X)^\times$, there is an isomorphism $\cO \stackrel{\cong}{\longrightarrow} \sL_{f^d,d}^{\delta_G(f)}$ of $G$-equivariant line bundles with connection on $X$, sending $1$ to $f^{-1} v$. This gives the equality $[\sL_{f^d,d}^{\delta_G(f)}]=[\cO]$ in $\Con^G(X)$. Using part (b), we then also have
	\begin{equation}\label{DivideCocycle} [\sL^\alpha_{u,d}] = [\sL^\alpha_{u,d}] \cdot [\cO] = [\sL^\alpha_{u,d}] \cdot [\sL^{\delta_G(f)}_{f^d,d}]^{-1} = [\sL_{u/f^d,d}^{\alpha/\delta_G(f)}] .\end{equation}
	Suppose now that $u = \lambda f^d$ for some $\lambda \in K^\times$ and some $f \in \cO(X)^\times$ such that $\alpha = \delta_G(f)$. Using (\ref{DivideCocycle}), we then have $[\sL^\alpha_{u,d}] = [\sL_{\lambda,d}^1] = [\cO]$ in $\Con^G(X)$.  
	
	Conversely, suppose that $[\sL^\alpha_{u,d}] = [\cO]$ in $\Con^G(X)$. Then $\omega([\sL_{u,d}^\alpha]) = [\sL_{u,d}] = [\cO]$ in $\Con(X)^G$, so using Proposition \ref{PicConDtor} we can find $f\in \cO(X)^\times$ and $\lambda \in K^\times$ such that $u= \lambda f^d$. Then $(\ref{DivideCocycle})$ implies that $[\cO]=  [\sL^\alpha_{u,d}] = [\sL_{\lambda,d}^{\alpha/\delta_G(f)}]$ in $\Con^G(X)$. Hence $G$ must fix the basis vector $v$ of $\sL_{\lambda,d}^{\alpha/\delta_G(f)}(X)^{\cT(X) = 0}$, so $\alpha = \delta_G(f)$.
	
	(d) Suppose that $[\sL] \in \Con^G(X)[de]$ is such that $\omega([\sL]) = [\sL_{u,d}]$. Consider the function $\alpha\colon G\to \cO(X)^\times$ defined by $g \cdot v=\alpha(g)v$. Then $\alpha\in Z^1(G,\cO(X)^\times)$ by Lemma \ref{GOZ1}.
	
	Since $(de) \cdot [\sL] = 0$ in $\Con^G(X)$, there is an isomorphism of $G$-equivariant line bundles with flat connection $\varphi : \sL^{\otimes de}\stackrel{\cong}{\longrightarrow} \cO_X$. On the other hand we also have the isomorphism $\psi : \sL_{u,d}^{\otimes de}\stackrel{\cong}{\longrightarrow} \cO_X$ of line bundles with flat connection constructed in the proof of part (a) above. Since $\omega([\sL]) = \sL_{u,d}$,  Corollary \ref{IsoLL} implies that $\varphi = \lambda \psi$ for some $\lambda \in K^\times$, and then $\varphi(v^{\otimes de}) = \lambda u^e$.  Since $\varphi$ is $G$-linear, we have
	\[g \cdot \lambda u^e = g \cdot \varphi(v^{\otimes de}) = \varphi(g \cdot v^{\otimes de}) = \alpha(g)^{de} \lambda u^e \qmb{for all} g \in G,\]
	so $\alpha^{de}=\delta_G(u^e)$.  Hence $\alpha \in \cZ^{G,X}_{u,d,e}$ and $[\sL] = [\sL^\alpha_{u,d}]$ in $\Con^G(X)$.
	
	Finally, suppose that $\alpha, \beta \in \cZ^{G,X}_{u,d,e}$ are such that $[\cL^\alpha_{u,d}] = [\cL^\beta_{u,d}]$. Then $[\cL^{\alpha/\beta}_{1,d,e}] = [\cO]$ by part (b), so $\alpha/\beta = \delta_G(f)$ for some $f \in \cO(X)^\times$ such that $1 / f^d \in K^\times$. Corollary \ref{dtors} implies that $f\in K^\times$.  Then $\delta_G(f) = 1$ so $\alpha = \beta$ as required.
\end{proof}
We interrupt the flow of $\S \ref{SecCocycELBC}$ to state and prove a useful elementary Lemma which is valid with far less restrictive assumption on the affinoid variety $X$ than those imposed at the start of $\S \ref{SecCocycELBC}$: it holds if $X$ is only assumed to be reduced.
\begin{lem}\label{Ddivis} Let $\varpi:=p^{-\frac{1}{p-1}}\in \bR_{>0}$.
\be \item If $d$ is an integer such that $p \nmid d$, then the $d$th power map \[(-)^d\colon \cO(X)^{\times\times}\to \cO(X)^{\times\times}\] is an isomorphism of topological groups. 	
\item If $r\in (0,\varpi/p)$ then every element of $\cO(X)^{\times\times}_r$ has a $p$th root in $\cO(X)^{\times\times}$. \ee
\end{lem}
\begin{proof} (a) If $a \in \cO(X)^{\circ \circ}$ then the binomial expansion
\[(1 + a)^{1/d} = \sum\limits_{n=0}^\infty \binom{1/d}{n} a^n\]
converges to an element of $\cO(X)^{\times\times}$ because $|a| < 1$ and because $\binom{1/d}{n} \in \Zp \subset K^\circ$ for all $n \geq 0$ as a consequence of the assumption that $p \nmid d$. That the map $a\mapsto \sum\limits_{n=0}^\infty \binom{1/d}{n} a^n$ is continuous is evident. 

(b) Similary if $|a|\leq r<\varpi/p$, the binomial expansion \[ (1+a)^{1/p}=\sum\limits_{n=0}^{\infty} \binom{1/p}{n}a^n \] converges to an element of $\cO^{\times\times}$ since \[v_p\left(p^n\binom{1/p}{n}\right)=-v_p(n!) \geq -\frac{n}{p-1}\] so that for $n\geq 1$ \[ \left|\binom{1/p}{n}a^n \right|_X\leq \left(pr/\varpi \right)^n\] and $pr/\varpi<1$. Thus $\sum\limits_{n=0}^{\infty} \binom{1/p}{n}a^n$ is the required $p$th root of $1+a$. 
\end{proof}

\begin{notn}\label{TXnotn} We will write $T(X)$ to denote the discrete abelian group \[ T(X):=\cO(X)^{\times}/\cO(X)^{\times\times} \] and $\pi_{T(X)}$ to denote the natural projection map $\pi_{T(X)}\colon \cO(X)^\times\to T(X)$.  \end{notn}

Since $\cO(X)^{\times\times}$ is open in $\cO(X)^\times$, $\pi_{T(X)}$ is continuous. We also observe that $T(-)$ defines a functor from affinoid varieties to abelian groups. Our next result gives conditions on the data $u,d,e,\alpha$ that determine when $[\sL^\alpha_{u,d,e}] \in \Con^G(X)$ is in fact the trivial element, in the case where $d,e$ are both coprime to $p$.

	
\begin{prop} \label{Ludatriv} Let  $d,e\geq 1$ be integers coprime to $p$.  Then for every $u \in \cO(X)^\times$ and $\alpha \in \cZ^{G,X}_{u,d,e}$, the following are equivalent:
	\begin{enumerate}[{(}i{)}]
		\item there exists $v\in \cO(X)^\times$ such that $u/v^d\in K^\times$ and $\alpha=\delta_G(v)$,
		\item there exists $v\in \cO(X)^\times$ and $\lambda\in K^\times$ such that 
		\[\pi_{T(X)}(\lambda v^d)=\pi_{T(X)}(u) \qmb{and} \pi_{T(X)}\circ \alpha=\pi_{T(X)}\circ \delta_G(v).\]
		\item $[\sL^{\alpha}_{u,d,e}]=[\cO]\in \Con^G(X)$. 
\end{enumerate} \end{prop}

\begin{proof}
	The equivalence of (i) and (iii) is a special case of Lemma \ref{1cocycles}(c).
	
	The implication (i)$\Rightarrow$(ii) is immediate since we can take $\lambda=u/v^d$. 
	
	Suppose that $v\in \cO(X)^\times$ and $\lambda\in K^\times$ such that $\pi_{T_X}(\lambda v^d)=\pi_{T(X)}(u)$ and $\pi_{T(X)}\circ \alpha=\pi_{T(X)}\circ \delta_G(v)$.  Since $\ker \pi_{T(X)}=\cO(X)^{\times\times}$, using Lemma \ref{Ddivis}, we can find $\varepsilon \in \cO(X)^{\times\times}$ such that $\varepsilon^d = \lambda v^d/u$. Setting $v':=v/\varepsilon'$, we have \[u/v'^d\in K^\times\qmb{and}\pi_{T(X)}\circ \delta_G(v')=\pi_{T(X)}\circ\delta_G(v)=\pi_{T(X)}\circ\alpha.\] 
	
	To deduce that (i) holds it remains to prove that $\alpha=\delta_G(v')$. Since $u/v'^d\in K^\times$, we have $\delta_G(u)=\delta_G(v')^d$. Because $\alpha\in \mathcal{Z}^{G,X}_{u,d,e}$, $\alpha^{de}=\delta_G(u^e)=\delta_G(v'^{de})$ shows that $\alpha/\delta_G(v')$ takes values in $\mu_{de}(K)$. However $\alpha/\delta_G(v')$ also takes values in $\ker \pi_{T(X)}=\cO(X)^{\times\times}$. We're now done because $\cO(X)^{\times\times}\cap \mu_{de}(K)$ is trivial.
\end{proof}
\begin{prop} \label{approxcocycle} Assume the hypotheses of Proposition \ref{Ludatriv} hold. Suppose also that $G$ is compact and that the exponent of every finite abelian $p'$-quotient of $G$ divides $e$.
	
	For every $\beta\in Z^1(G,\cO(X)^\times)$ and $u \in \cO(X)^\times$ such that $\pi_{T(X)}\circ \left(\beta^{-d}\delta_G(u)\right)$ takes values in $K^\times/K^{\times\times}$ there is a unique $\alpha\in \mathcal{Z}^{G,X}_{u,d,e}$ such that \[\pi_{T(X)} \circ \alpha=\pi_{T(X)} \circ \beta\in Z^1(G,T(X)).\] 
\end{prop}
\begin{proof} Let $\eta$ be the $1$-cocycle $\eta:=\delta_G(u)\beta^{-d}$. The assumption on $\eta$ implies that   
 $\pi_{T(X)} \circ \eta\in Z^1(G,K^\times/K^{\times\times})$. Since $G$ is compact and $K^\times/K^{\times\times}$ is a discrete group with trivial $G$-action, $\pi_{T(X)}\circ \eta\in \Hom(G,K^\times/K^{\times\times})$ has finite image and so takes values in the torsion subgroup of $K^{\times}/K^{\times\times}$. 
 
 Since $K^\times/K^{\times\times}$ has no $p$-torsion, $\pi_{T(X)}\circ \eta$ factors through a finite abelian $p'$-quotient of $G$ and thus  $(\pi_{T(X)}\circ \eta)^e=1$ by our assumption on $G$. That is $\eta^e$ takes values in $\cO(X)^{\times\times}$. By Lemma \ref{Ddivis}(a), $\eta^{e}$ has a $de^{\rm{th}}$ root $\gamma$ in  $Z^1(G,\cO(X)^{\times\times})$: $\gamma^{de} = \eta^e$. 
 
 Now  $\delta_G(u)^{e}= (\eta \beta^d)^e = (\gamma j)^{de}$, so $\alpha:=\gamma \beta$ satisfies $\alpha^{de}=\delta_G(u)^e$.  Moreover $\pi_{T(X)} \circ \alpha=\pi_{T(X)} \circ \beta$ as required.
	
	Suppose now that $\alpha'\in \mathcal{Z}_{u,d,e}^{G,X}$ satisfies $\pi_{T(X)}\circ\alpha'=\pi_{T(X)}\circ \beta=\pi_{T(X)}\circ \alpha$. Then $\alpha/\alpha'$ takes values in $\cO(X)^{\times\times}$ but also $(\alpha/\alpha')^{de}=1$. Then $\alpha'=\alpha$ since $p \nmid de$.\end{proof}


\section{Applications to Drinfeld's upper half plane}
\subsection{Subdomains of the rigid analytic affine line}
We will write $\A := \A_K := \A_K^{1,\an}$ to denote the rigid $K$-analytic affine line, equipped with a fixed choice of local coordinate $x \in \cO(\A)$.  We write $\bP^1$ to denote the rigid $K$-analytic projective line. 
\begin{defn}\label{Cheese}A \emph{$K$-cheese} is an affinoid subdomain of $\bA$ of the form
\[C_K(\mathbf{\alpha}, \mathbf{s}) := \Sp K \left\langle \frac{x - \alpha_0}{s_0}, \frac{s_1}{x - \alpha_1}, \cdots, \frac{s_g}{x - \alpha_g}\right\rangle\]
for some $\mathbf{\alpha} := (\alpha_0,\ldots, \alpha_g) \in K^{g+1}$ and $\mathbf{s} := (s_0,\ldots, s_g) \in (K^\times)^{g+1}$, which satisfy
\begin{itemize}
\item $|s_i| \leq |s_0|$ for all $i = 1,\ldots, g$,
\item $|\alpha_i - \alpha_0| \leq |s_0|$ for all $i = 1,\ldots,g$, and
\item $|\alpha_i - \alpha_j| \geq \max\{|s_i|, |s_j|\}$ whenever $1 \leq i < j \leq g$.
\end{itemize}
When there is no risk of confusion, we will simplify the notation to $C(\mathbf{\alpha},\mathbf{s})$.

We call the open discs \begin{eqnarray*} D_\infty & := & \{z\in \bP^1({\bf{C}}):|z-\alpha_0|>|s_0|\} \mbox { and }\\ D_i & := &\{z\in \bP^1({\bf C}): |z-\alpha_i|<|s_i|\}\mbox{ for }i=1,\ldots,g\end{eqnarray*} the \emph{holes} of $C(\alpha,\mathbf{s})$ and we write  \[ h(C(\alpha,\mathbf{s})):=\{D_1,\ldots,D_g,D_\infty\}\] to denote the set of holes of $C(\alpha,\mathbf{s})$.  
\end{defn}
Of course the $\bf{C}$-points of $C(\mathbf{\alpha}, \mathbf{s})$ are obtained by removing the $g+1$ holes from $\bP^1({\bf C})$. The conditions on the parameters $\mathbf{\alpha}$ and $\mathbf{s}$ are there to ensure that the holes are pairwise disjoint. We also require that $\mathbf{\alpha}$ and $\mathbf{s}$ are defined over $K$.

\begin{rmk}\label{ovlpdiscs} \hsp \be
\item Given two open discs $D_1, D_2$ in $\bP^1(\bfC)$ with $D_1 \cap D_2 \neq \emptyset$, it must necessarily be the case that either $D_1 \subseteq D_2$, or $D_2 \subseteq D_1$.
\item The union and the intersection of two cheeses $C_1,C_2$ are also cheeses, unless $C_1 \cap C_2 = \emptyset$.
\ee \end{rmk}

\begin{lem} \label{rYX} Suppose that $X$ and $Y$ are $K$-cheeses and $\varphi : \bP^1 \to \bP^1$ is a $K$-analytic automorphism such that $\varphi(Y)\subseteq X$.  Then there is a unique function 
\[\varphi^X_Y\colon h(X)\to h(Y)\]
such that for every $D\in h(X)$, $\varphi^X_Y(D)$ is the largest hole of $Y$ containing $\varphi^{-1}(D)$.\end{lem}
\begin{proof} By \cite[p. 33]{Lut}, the automorphism $\varphi$ is necessarily a M\"{o}bius transformation. Hence $\varphi$, as well as $\varphi^{-1}$, maps open discs in $\bP^1(\bfC)$ to open discs in $\bP^1(\bfC)$. 

Let $D\in h(X)$. Then $\varphi^{-1}(D)\cap Y(\bfC)\subseteq \varphi^{-1}(D\cap X(\bfC))=\emptyset$. Thus the open disc $\varphi^{-1}(D)$ is contained in the union of the holes of $Y$. Remark \ref{ovlpdiscs}(a) then implies that $\varphi^{-1}(D)$ is contained in a unique hole $\varphi_Y^X(D)$ of $Y$.
\end{proof}

\begin{notn} Suppose that $X,Y$ are $K$-cheeses with $Y \subseteq X$. We will denote the function $\id^X_Y : h(X) \to h(Y)$ associated with the identity map $\id : \bP^1 \to \bP^1$ by $\iota^X_Y : h(X) \to h(Y)$.
\end{notn}

The following lemma will be useful later.
\begin{lem} \label{fibrecheeses} Suppose that $X$ and $Y$ are $K$-cheeses with non-empty intersection. Then there is a natural bijection 
\[ \iota^{X\cup Y}_X \times \iota^{X\cup Y}_Y : h(X\cup Y)\quad\longrightarrow \quad h(X) \underset{h(X\cap Y)}{\times}{} h(Y) \] given by $D\mapsto (\iota^{X\cup Y}_X(D), \iota^{X\cup Y}_Y(D))$.
\end{lem}
\begin{proof} Note that the map in the statement of the Lemma is well-defined because 
\[\iota^X_{X\cap Y} \circ \iota^{X\cup Y}_X=\iota^{X\cup Y}_{X\cap Y}=\iota^Y_{X\cap Y}\circ \iota^{X\cup Y}_Y.\]
De Morgan's laws imply that $h(X\cup Y)$ is precisely the set of non-empty intersections $A\cap B$ with $A\in h(X)$ and $B\in h(Y)$. Let $D \in h(X \cup Y)$; hence there exist $A \in h(X)$ and $B \in h(Y)$ with $D = A \cap B$. But then $\iota_X^{X\cup Y}(D)=A$ and $\iota^{X\cup Y}_Y(D)=B$, so $D = A$ or $D= B$ by Remark \ref{ovlpdiscs}(a). Hence $\iota^{X\cup Y}_X \times \iota^{X\cup Y}_Y$ is injective.
	
We now show that $\iota^{X\cup Y}_X \times \iota^{X\cup Y}_Y$ is surjective. Suppose that $A\in h(X)$ and $B\in h(Y)$ are such that $\iota^{X}_{X\cap Y}(A)=\iota^Y_{X\cap Y}(B)=:E$. This means that $A$ and $B$ are contained in the same hole $E$ of $X\cap Y$. Since $E$ is a hole of $X\cap Y$, by de Morgan's laws again we see that $E$ is the union of the holes of $X$ contained in $E$ together with the holes of $Y$ contained in $E$. But no open disc in $\bP^1(\bf{C})$ is a finite union of proper open subdiscs; hence $E\in h(X)\cup h(Y)$. It follows that either $E = A$ or $E = B$. Since $E$ contains both $A$ and $B$, it follows that either $A \subseteq B$ or $B \subseteq A$. Hence $D := A \cap B$ is non-empty and is therefore a hole of $X \cup Y$ as we saw above. It is now clear that 
\[(A,B)=\left(\iota^{X\cup Y}_X(D),\iota^{X\cup Y}_Y(D)\right)\] lies in the image of $\iota^{X\cup Y}_X \times \iota^{X\cup Y}_Y$. \end{proof}
We are interested in these cheeses because \emph{every} connected affinoid subdomain of the affine line $\bA$ is a $K$-cheese whenever $K$ is algebraically closed by \cite[Corollary 2.4.7]{LutJacobians}. We will prove Theorem \ref{RationalCheese} below which carries out a Galois descent of this statement down to our base field $K$ which may fail to be algebraically closed.

\begin{lem}\label{GalDescAff} Let $\cG_K := \Gal(\overline{K}/K)$ and let $A$ be a $K$-affinoid algebra. Then the natural map $A \to (A \h\otimes \bfC)^{\cG_K}$ is an isomorphism of $K$-Banach algebras.
\end{lem}
\begin{proof} Since $A$ is a quotient of a Tate algebra, $A$ is of countable type as a $K$-Banach space: it has a dense $K$-linear subspace of countable dimension. Assume first that $\dim_K A = \infty$. By \cite[Proposition 1.2.1(3)]{FvdPut}, we can find a $K$-Banach space isomorphism $\varphi : A \to c_0(K)$; this means that $\varphi$ is a bounded $K$-linear map which has a bounded $K$-linear inverse. Now consider the commutative diagram

\[ \xymatrix{ A \ar[rr]\ar[d]_{\varphi} && (A \h\otimes\bfC )^{\cG_K} \ar[d]^{\varphi \h\otimes 1} \\ c_0(K) \ar[rr] && (c_0(K) \h\otimes \bfC)^{\cG_K}.}\]
The arrow in the bottom row is an isomorphism because 
\[(c_0(K) \h\otimes \bfC)^{\cG_K} = c_0(\bfC)^{\cG_K} = c_0(\bfC^{\cG_K}) = c_0(K)\]
by the Ax-Sen-Tate theorem --- see, for example, \cite[Proposition 2.1.2]{BriCon}; the proof given there works for any complete non-Archimedean field of characteristic zero. The case $\dim_KA < \infty$ is handled in a similar manner.
\end{proof}

\begin{defn} When $X$ is an affinoid subdomain of $\bA$ and $K'$ is a finite extension of $K$ we say that $X$ is \emph{split over $K'$} if $X_{K'}$ is a finite union of pairwise disjoint cheeses.
\end{defn}

We will write $\sqrt{|K|^\times}$ to denote the divisible subgroup of $\R^\times$ generated by $|K^\times|$; this is the same as $|\overline{K}^\times|$. 
\begin{thm}\label{RationalCheese}
Let $X$ be an affinoid subdomain of the affine line $\A$. Then there is a finite extension $K'$ of $K$ such that $X$ splits over $K'$. 
\end{thm}
\begin{proof} Recall that $X$ is \emph{geometrically connected} if the base change $X_\bfC$ is connected. Suppose first that $X$ is geometrically connected. Then $X_\bfC$, being connected, is a cheese $C_{\bfC}(\alpha,\mathbf{s})$ by \cite[Corollary 2.4.7]{LutJacobians}. Since $\overline{K}$ is dense in $\bfC$, we can choose the centres $\alpha_0,\alpha_1,\ldots,\alpha_g$ to be first in $\overline{K}$, and then find a large enough finite extension $K'$ of $K$ such that $\alpha_i \in K'$ for all $i$. Since $|\overline{K}^\times| = |\bfC^\times| = \sq{K}$, we may enlarge $K'$ if necessary and arrange that $s_i \in K'$ for all $i$ as well. Let $Z  := C_{K'}(\alpha,\mathbf{s})$ be the same cheese but defined over $K'$. Choose a large enough closed disc $D$ defined over $K'$ which contains both $X' := X_{K'}$ and $Z$, and fix a coordinate $y$ on $D$. Then there is an isomorphism of $\bfC$-affinoid varieties $X' \times_{K'} \bfC \cong Z \times_{K'} \bfC$ compatible with the inclusions $X' \hookrightarrow D$ and $Z \hookrightarrow D$. Now consider the induced $\bfC$-algebra isomorphism 
\[ \psi : \cO(X') \h\otimes_{K'} \bfC = \cO(X' \times_{K'} \bfC) \stackrel{\cong}{\longrightarrow} \cO(Z \times_{K'} \bfC) = \cO(Z) \h\otimes_{K'}\bfC.\]
Because $X'$ is an affinoid subdomain of $D$, the $\bfC$-algebra $\cO(X') \h\otimes_{K'} \bfC$ contains a dense $\bfC$-subalgebra generated by rational functions in $y$ with coefficients in $K'$, and $\psi$ must send these rational functions to $\cG_{K'}$-invariants in the target. Hence $\psi$ respects the natural $\cG_{K'} := \Gal(\overline{K}/K')$-actions on both sides. Taking $\cG_{K'}$-invariants and applying Lemma \ref{GalDescAff} we deduce a $\cO(D)$-algebra isomorphism $\cO(X') \cong \cO(Z)$, so that $X' = Z$ is a cheese.

Returning to the general case, it will now be enough to show that there is some finite extension $K''$ of $K$ such that every connected component of $X_{K''}$ is geometrically connected. To see this, consider again $X_\bfC$, and let $\{e_1,\ldots,e_n\}$ be the primitive idempotents of $\cO(X_\bfC)$. Since $\cG_K$ acts continuously on $\cO(X_\bfC)$ the stabiliser $H_i$ in $\cG_K$ of each $e_i$ is closed. On the other hand, $\cG_K$ preserves $\{e_1,\ldots,e_n\}$ so each $H_i$ has finite index in $\cG_K$. Hence each $H_i$ is also open in $\cG_K$. We can therefore find a finite extension $K''$ of $K$ such that $\cG_{K''}$ fixes each $e_i$ pointwise. Then $e_i \in \cO(X_\bfC)^{\cG_{K''}} = (\cO(X) \h\otimes \bfC)^{\cG_{K''}} = \cO(X) \h\otimes K'' = \cO(X_{K''})$ again by Lemma \ref{GalDescAff}. It follows that every connected component of $X_{ K''}$ is geometrically connected as required. \end{proof}

\begin{prop}\label{CheeseSchauder} Let $C=C_K(\alpha;\mathbf{s})$ be a $K$-rational cheese and $\xi=\xi_0$ a coordinate on $\mathbb{A}^1$ such that $D_\infty=\{z\in \bP^1(\bfC):|\xi(z)|>1\}$. 
	
	For $i=1,\ldots,g$ let $\xi_i=\frac{c_i}{\xi-\xi(\alpha_i)}$ with $c_i\in K^\times$ such that $|\xi_i|=1$. Then the set $\{1,\xi_i^j: j\geq 1, 0\leq i\leq g \}$ is an orthonormal Schauder basis for the $K$-Banach space $\cO(C)$, in the sense of \cite[\S 2.7.2]{BGR}. 
\end{prop}
\begin{proof} This is a straightforward rephrasing of \cite[Proposition 2.4.8(a)]{LutJacobians}.\end{proof}
	
\begin{prop}\label{CheeseUnits} Let $X = C(\mathbf{\alpha},\mathbf{s})$ be a cheese. Then the map 
\[\mathbb{Z}^g \longrightarrow \frac{\cO(X)^\times}{K^\times \cdot \cO(X)^{\times\times}}\] 
defined by
\[(n_1,\ldots,n_g) \quad \mapsto \quad (x-\alpha_1)^{n_1}\cdots (x - \alpha_g)^{n_g} \cdot K^\times \cdot \cO(X)^{\times\times}\]
is an isomorphism of abelian groups.
\end{prop}
\begin{proof}This is \cite[Proposition 2.4.8(b)]{LutJacobians}.\end{proof}

\begin{prop}\label{PicCheese} For every cheese $X$, $\Pic(X)=0$.\end{prop}

\begin{proof} This follows from \cite[Proposition 8.2.3(1)]{FvdPut} and \cite[Corollary 3.8]{vdPut82}.
\end{proof}

\begin{prop}\label{admquasiStein} Suppose that $X$ is an admissible subdomain of $\bA$ and $\{X_n\}_{n=0}^{\infty}$  is an admissible cover of $X$ by cheeses such that $X_n\subset X_{n+1}$ for all $n$ and each map $\iota^{X_n}_{X_{n+1}}\colon h(X_{n+1})\to h(X_n)$ is surjective. Then $X$ is geometrically connected, smooth and quasi-Stein.
\end{prop}

\begin{proof} Since $X_{\bfC}$ has an admissible cover by the cheeses $X_{n,\bfC}$, $X$ is geometrically connected and smooth. 
	
	 By \cite[Exercise 2.6.2]{FvdPut} all the maps $\cO(X_{n+1})\to \cO(X_n)$ have dense image and so $X$ is quasi-Stein. One argument to complete the exercise is as follows. For any cheese $C:=C(\alpha,\mathbf{s})$ the sub-$K$-algebra of rational functions \[\cO_{\mathrm{rat}}(C):=K[x,(x-\alpha_1)^{-1},\ldots,(x-\alpha_d)^{-1}]\] is dense in $\cO(C)$, by Proposition \ref{CheeseSchauder}. Moreover, the condition $h(X_{n+1})\to h(X_n)$ is surjective guarantees that the centres of the holes in $X_n$ can all be chosen to not lie in $X_{n+1}$ so that $\cO_{\mathrm{rat}}(X_n)\subset \cO(X_{n+1})$. 
	\end{proof}

\subsection{Drinfeld's upper half-plane}

In this section we study $\Omega_F$ the \emph{Drinfeld upper half plane}, which is a rigid $F$-analytic space whose underlying set consists of the $\Gal(\overline{F}/F)$-orbits in $\Omega_F(\overline{F})=\bP^1(\overline{F}) \backslash \bP^1(F)$.

It is straightforward to see that for any finite extension $L$ of $F$, $\Omega_F(L)$ can be identified with $L\backslash F$ and we will often silently make this identification. The rigid space $\Omega_F$ comes naturally equipped with an action of $GL_2(F)$ by M\"obius transformations: \[ \begin{pmatrix} a & b \\ c & d \end{pmatrix}\cdot z= \frac{az+b}{cz+d}\] and the same formula induces an action on each set $\Omega_F(L)$.

We  recall from \cite[\S I.1.2]{BC} the $GL_2(F)$-equivariant reduction map $\lambda\colon \Omega_F(\bfC) \to |\mathcal{BT}|$ from $\Omega_F(\bfC) = \bP^1(\bfC) \backslash \bP^1(F)$ to the geometric realisation $|\mathcal{BT}|$ of the Bruhat-Tits tree $\mathcal{BT}$ associated with $PGL_2(F)$. Since $\lambda \circ \sigma = \lambda$ for any $\sigma \in \Gal(\overline{F}/F)$, this map factors through the rigid $F$-analytic space $\Omega_F$, giving us a map
\[ \lambda : \Omega_F \to |\mathcal{BT}|.\]
Any point in $\Omega_F$ is the $\Gal(\overline{F}/F)$-orbit $[z]$ of some $z \in \Omega_F(\overline{F})$; then we have $\lambda([z]) = \lambda(z)$. We abuse notation and also call $\lambda : \Omega_F \to |\mathcal{BT}|$  the \emph{reduction map}. 

\begin{lem}\label{stabrel} Let $z \in \Omega_F(\overline{F})$. Then $GL_2(F)_z \leq  GL_2(F)_{[z]}  \leq  GL_2(F)_{\lambda\left([z]\right)}$.
\end{lem}

\begin{proof}
This is a consequence of the $GL_2(F)$-equivariance of $z\mapsto [z]$ and $\lambda$. 
\end{proof}

\begin{prop} \label{TreeToCheese} Suppose that $\cT$ is a finite subtree of $\mathcal{BT}$. Then $\lambda^{-1}(|\cT|)$ is an $F$-cheese contained in $\Omega_F$.
\end{prop}
\begin{proof} Since the union of two non-disjoint $F$-cheeses is an $F$-cheese and $\cT$ is connected, by an induction on the number of edges of $\cT$, it suffices to prove the result when $\cT$ is a single vertex or has two vertices connected by an edge. Both of these cases can be deduced from the discussion in \cite[\S I.2.3]{BC}.
\end{proof}

\begin{defn} \label{KcheeseCT} For every finite subtree $\cT$ of $\mathcal{BT}$, we define
\[ C_{\cT} := \lambda^{-1}(|\cT|) \times_F K.\]
\end{defn}
Note that $C_{\cT}$ is a $K$-cheese contained in $\Omega := \Omega_F \times_F K$, by Proposition \ref{TreeToCheese}.
\begin{defn} \label{nbdtree} Suppose $\cT$ is a finite subtree of $\mathcal{BT}$. \be \item The \emph{neighbourhood of $\cT$} is the subset $N(\cT)$ of the set of edges of $\mathcal{BT}$ with precisely one vertex in $\cT$: \[  N(\cT):=\{(ss')\in E(\mathcal{BT}):s\in \cT, s'\not\in \cT\}.\]
\item For $e\in N(\cT)$ we write $s_\cT(e)$ to denote the vertex of $e$ in $\cT$ and $t_\cT(e)$ to denote the vertex of $e$ not in $\cT$; \[s_{\cT}((ss')):=s\mbox{ and } t_{\cT}((ss')):=s'\mbox{ for }s\in \cT, s'\not\in \cT\]
\ee \end{defn} 
\begin{lem}\label{mapnbdtreeslem} Let $\cT' \subseteq \cT$ be finite subtrees of $\mathcal{BT}$ and let $e\in N(\cT)$. Then there is a unique $f\in N(\cT')$ such that the unique path in $\mathcal{BT}$ from $t_\cT(e)$ to $t_{\cT'}(f)$ contains no vertices of $\cT'$. 
\end{lem}

\begin{proof} Let $w$ be any vertex of $\cT'$. Since $\mathcal{BT}$ is a tree it contains a unique path from $t_\cT(e)$ to $w$. Since $w$ is a vertex of $\cT'$ and $t_\cT(e)$ is not, there is precisely one edge $f$ in this path contained in $N(\cT')$. We can then truncate the path to a path from $t_\cT(e)$ to $t_{\cT'}(f)$ that contains no vertices of $\cT'$. 

If $f'$ is an element of $N(\cT)\backslash\{f\}$ then the unique path from $t_{\cT'}(f)$ to $t_{\cT'}(f')$ must pass through a vertex of $\cT'$ so there is no path from $t_\cT(e)$ to $t_{\cT'}(f')$ that contains no vertices of $\cT'$. 
\end{proof}
\begin{defn} \label{mapnbdtreesdef} If $\cT' \subseteq \cT$ are finite subtrees of $\mathcal{BT}$, then \[\iota^{\cT}_{\cT'}\colon N(\cT)\to N(\cT')\] is the map that sends $e\in N(\cT)$ to $f \in N(\cT')$ given by Lemma \ref{mapnbdtreeslem}.
	\end{defn}

\begin{ex}\label{leafvertex} Suppose that $\cS$ is a subtree of $\cB\cT$ consisting of two vertices $s$ and $s'$ and the single edge $(ss')$, and $\{s\}$ is the subtree of $\cS$ with $s$ as its only vertex. 
Then there exist $2q$ edges $e_1,\ldots,e_q,f_1,\ldots,f_q \in E(\cB\cT)$ such that:
\be \item  $N(\cS)=\{e_1,\ldots,e_q,f_1,\ldots,f_q\}$,
\item $s_\cS(e_i)=s$ and $s_{\cS}(f_i)=s'$ for each $i=1,\ldots,q$,
\item $N(\{s\})=\{e_1,\ldots,e_q,(ss')\}$, and
\item $\iota^{\cS}_{\{s\}}(e_i)=e_i$ and $\iota^{\cS}_{\{s\}}(f_i)=(ss')$ for each $i=1,\ldots, q$.
\ee\end{ex}

\begin{lem}\label{iotalem} Suppose that $\cT_1\subseteq \cT_2\subseteq \cT_3$ are finite subtrees of $\cB\cT$.
	\be\item  $\iota^{\cT_3}_{\cT_1}=\iota^{\cT_2}_{\cT_1} \circ \iota^{\cT_3}_{\cT_2}$.
	\item If $e\in N(\cT_2)$ and $s_{\cT_2}(e)\in V(\cT_1)$, then $\iota_{\cT_1}^{\cT_2}(e)=e$. 
	\item If $e\in N(\cT_2)$ and $s_{\cT_2}(e)\not\in V(\cT_1)$, then $\iota_{\cT_1}^{\cT_2}(e)\in E(\cT_2)\backslash E(\cT_1)$.
	\ee
\end{lem}

\begin{proof}
(a)	If $e\in N(\cT_3)$ then the path $\cP$ in $\cB\cT$ from $t_{\cT_3}(e)$ to $t_{\cT_1}\left(\iota^{\cT_3}_{\cT_1}(e)\right)$ that contains no vertices of $\cT_1$ can be decomposed as a union of two subpaths with a single vertex in common (and possibly no edges): one of these subpaths  goes from $t_{\cT_3}(e)$ to the last vertex $s$ in $\cP$ that does not lie in $\cT_2$ and the other goes from $s$ to $t_{\cT_1}\left(\iota^{\cT_3}_{\cT_1}(e)\right)$. Then these paths show that $\iota_{\cT_3}^{\cT_2}(e)$ is the unique element $f$ of $N(\cT_2)$ such that $t_{\cT_2}(f)=s$ and $\iota^{\cT_2}_{\cT_1}(f)$ is $\iota^{\cT_3}_{\cT_1}(e)$ as required.

(b)  Since $\cT_1\subseteq \cT_2$, the condition $s_{\cT_2}(e)\in \cT_1$ gives that $e\in N(\cT_1)$ and the path from $t_{\cT_2}(e)$ to $t_{\cT_1}(e)=t_{\cT_2}(e)$ has no edges and so contains no vertices of $\cT_1$. 

(c) First $N(\cT_1)\cap E(\cT_1)=\emptyset$ so $\iota^{\cT_2}_{\cT_1}(e)\not\in E(\cT_1)$.  Since $s_1:=s_{\cT_2}(e)\in \cT_2$ and  $s_2:=s_{\cT_1}\left(\iota_{\cT_1}^{\cT_2}(e)\right)\in \cT_1\subseteq \cT_2$, the unique path in $\cB\cT$ from $s_1$ to $s_2$ lies in $\cT_2$. The condition $s_1\not\in \cT_1$ ensures this path contains at least one edge $f:=\left(t_{\cT_1}\left(\iota_{\cT_1}^{\cT_2}(e)\right) s_2\right)$. Moreover $f\in N(\cT_1)$. Adding the edge $(t_{\cT_2}(e)s_1)$ to the start of the path and removing $f$ from its end gives the path that shows that $\iota_{\cT_1}^{\cT_2}(e)=f$. 
\end{proof}
\begin{lem} \label{fibretree} Suppose that $\cS$ and $\cT$ are finite subtrees of $\cB\cT$ such that \[E(\cS)=\{(ss')\}\mbox { and }V(\cS)\cap V(\cT)=\{s\}.\] There is a natural bijection 
\[ \iota^{\cS\cup \cT}_\cS \times \iota^{\cS\cup \cT}_{\cT} : N(\cS\cup \cT)\to N(\cS) \underset{N(\cS\cap \cT)}{\times}{} N(\cT) \] given by $e\mapsto (\iota^{\cS\cup \cT}_\cS(e), \iota^{\cS\cup \cT}_{\cT}(e))$. 
		\end{lem}
		
\begin{proof} The map $\xi := \iota^{\cS\cup \cT}_\cS \times \iota^{\cS\cup \cT}_{\cT}$ in the statement is well-defined, because
\[\iota^{\cS}_{\cS\cap \cT}\circ\iota^{\cS\cup \cT}_{\cS}=\iota^{\cS\cup \cT}_{\cS\cap \cT}=\iota^{\cT}_{\cS\cap \cT}\circ\iota^{\cS\cup \cT}_{\cT} \] by Lemma \ref{iotalem}(a). 

Next we show that $\xi$ is injective. To this end, suppose that $e_1, e_2$ are two elements of $N(\cS \cup \cT)$ such that $\xi(e_1) = \xi(e_2)$. Let $v_i=s_{\cS\cup \cT}(e_i)$ for $i=1,2$. 
Suppose first that both of $v_1,v_2$ lie in $\cS$. In this case, $\iota^{\cS \cup \cT}_{\cS}(e_1) = e_1$ and $\iota^{\cS \cup \cT}_{\cS}(e_2) = e_2$, by Lemma \ref{iotalem}(b), and we deduce by looking at the first component of $\xi(e_1) = \xi(e_2)$ that $e_1 = e_2$. The case where both $v_1,v_2$ lie in $\cT$ is entirely similar. Suppose for a contradiction that $e_1 \neq e_2$. Then without loss of generality, we can now assume that $v_1$ lies in $V(\cT) \backslash V(\cS)$ and $v_2$ lies in $V(\cS)\backslash V(\cT)$. Since $\cS$ is a single leaf with $V(\cS) \cap V(\cT) = \{s\}$, this forces $v_2 = s'$. Therefore since $\iota^{\cS \cup \cT}_{\cS}(e_2) = e_2$ by Lemma \ref{iotalem}(b), the only vertex of $\iota^{\cS \cup \cT}_{\cS}(e_2)$ in $\cS$ is $s'$. On the other hand, because $v_1 \notin V(\cS)$,  $\iota^{\cS \cup \cT}_{\cS}(e_1)\in E(\cT)$ by Lemma \ref{iotalem}(c). This contradicts $\iota^{\cS \cup \cT}_{\cS}(e_2) = \iota^{\cS \cup \cT}_{\cS}(e_1)$ because $s' \notin V(\cT)$.

Finally we show that $\xi$ is surjective. Suppose that $(e,f)\in N(\cS)\underset{N(\cS\cap \cT)}{\times}{}N(\cT)$. 

We first consider the case where $s_{\cT}(f)\neq s$, so that $t_{\cT}(f)\not\in \cS\cup \cT$. It follows that $f\in N(\cS\cup \cT)$ and we claim $\xi(f)=(e,f)$. That $\iota^{\cS\cup \cT}_{\cT}(f)=f$ follows from Lemma \ref{iotalem}(b) because $s_\cT(f)\in \cT$. Consider the following element $g$ of $N(\cS \cap \cT)$: \[g := \iota^{\cS}_{\cS\cap \cT}(e)=\iota^{\cT}_{\cS\cap \cT}(f).\]
Since $s_{ \cT}(f)\not\in\cS\cap \cT$, $g\in E(\cT)$ by Lemma \ref{iotalem}(c). In particular $g \neq (ss')$. Since $g = \iota^{\cS}_{\cS \cap \cT}(e)$, this implies that $g=e$ by Example \ref{leafvertex}(d). Now $s_{\cS\cup\cT}(f)\not\in \cS$ so $h := \iota^{\cS\cup \cT}_{\cS}(f)\in E(\cT)\backslash E(\cS)$ by Lemma \ref{iotalem}(c) again. Hence $\iota^{\cS}_{\cS \cap \cT}(h) = h$ by Example \ref{leafvertex}(d). Using Lemma \ref{iotalem}(a) several times, we now see that
\[\iota^{\cS\cup \cT}_{\cS}(f)= h = \iota^{\cS}_{\cS \cap \cT}(h) = \iota^{\cS}_{\cS \cap \cT} \iota^{\cS\cup \cT}_{\cS}(f) = \iota^{\cS\cup \cT}_{\cS\cap \cT}(f)=\iota^\cT_{\cS\cap \cT}\iota^{\cS\cup\cT}_{\cT}(f)=g. \]Hence $\iota^{\cS\cup \cT}_\cS(f)=g=e$ as required.

Next we consider the case where $s_{\cT}(f)=s$ so that, by Lemma \ref{iotalem}(b),  $\iota^{\cT}_{\cS \cap \cT}(f) = f$, and hence $f = \iota_{\cS \cap \cT}^{\cT}(f) = \iota^{\cS}_{\cS \cap \cT}(e)$.  This splits into two subcases. 

Suppose first that  $f=(ss')$. Then $\iota^{\cS}_{\cS\cap \cT}(e)=f =(ss')$ implies by Example \ref{leafvertex}(d) that $s_{\cS}(e) = s'$. Therefore $t_{\cS}(e) \notin V(\cT)$ which means that $e \in N(\cS \cup \cT)$. Then $\iota^{\cS\cup \cT}_{\cS}(e) = e$ by Lemma \ref{iotalem}(b) and $\iota^{\cS \cup \cT}_{\cT}(e) = (ss') = f$  by Lemma \ref{iotalem}(c), so $\xi(e) = (e,f)$ as required. 

Finally, suppose that $f \neq (ss')$. Then $t_{\cT}(f) \notin V(\cS)$, so $f \in N(\cS \cup \cT)$. Then $\iota^{\cS}_{\cS \cap \cT}(e) = f \neq (ss')$ implies that  $e = \iota^{\cS}_{\cS \cap \cT}(e) = f$ by Example \ref{leafvertex}(d). Hence $\iota^{\cS \cup \cT}_{\cT}(f) = f$ and $\iota^{\cS \cup \cT}_{\cS}(f) = \iota^{\cS \cup \cT}_{\cS}(e) = e$, and so $\xi(f) = (e,f)$ as required. 
\end{proof}

		\begin{prop} \label{holesnbd} Let $\cT$ be a finite subtree of $\mathcal{BT}$. Then there is a $G^0_\cT$-equivariant bijection \[h_\cT\colon N(\cT)\to h\left(C_\cT\right)  \] such that following diagram is commutative for every substree $\cT'$ of $\cT$:
\begin{equation}\label{hTfunc}\xymatrix{ N(\cT) \ar[rr]^{\iota^{\cT}_{\cT'}} \ar[d]_{h_\cT} & & N(\cT') \ar[d]^{h_{\cT'}} \\	
	h(C_\cT) \ar[rr]_{\iota^{C_\cT}_{C_{\cT'}}} && h(C_{\cT'}). }\end{equation}
\end{prop} 
\begin{proof} 
	Suppose $\cS$ and $\cT$ are disjoint finite subtrees of $\mathcal{BT}$. It follows from Proposition \ref{TreeToCheese} that $C_{\cS}$ and $C_{\cT}$ are disjoint $K$-cheeses, so $C_{\cS}$ is contained in a unique hole of $C_{\cT}$. In particular, if $e\in N(\cT)$, then $C_{\{t_\cT(e)\}}$ and $C_\cT$ are disjoint $K$-cheeses, so $C_{\{t_\cT(e)\}}$ is contained in a unique hole $h_\cT(e)$ of $C_\cT$.

	Since $\lambda$ is $G^0$-equivariant, if $g\in G^0_\cT$ then $t_\cT(g\cdot e)=gt_\cT(e)$ and so \[\lambda^{-1}(t_\cT(ge))=g\lambda^{-1}(t_\cT(e)) \qmb{and} h_\cT(ge)=g h_\cT(e).\] Thus $e\mapsto h_\cT(e)$ defines a $G^0_\cT$-equivariant function.
	
Suppose that $\cT'\subseteq \cT$ is a subtree and that $\iota_{\cT'}^{\cT}(e)=f$. Then the path from $t_\cT(e)$ to $t_{\cT'}(f)$ in $\mathcal{BT}$ is a tree, $\cS$ say, that is disjoint from $\cT'$. Then $C_{\{t_\cT(e)\}}$ and $C_{\{t_{\cT'}(f)\}}$ and $C_{\cS}$ are all contained in the same hole, $D$ say, of $C_{\cT'}$. It follows that $h_{\cT'}\circ\iota^\cT_{\cT'}(e)=D=\iota^{C_\cT}_{C_\cT'}\circ h_\cT(e)$ and that the diagram (\ref{hTfunc}) is commutative.

	To show that $h_\cT$ is always a bijection, we induct on the number of edges of $\cT$. If $\cT$ consists of a single vertex (no edges), or a single edge, then the result is a simple consequence of \cite[I.2.3]{BC}. In the general case, we decompose $\cT$ as $\cS\cup \cT'$ where $\cS$ is a single leaf of $\cT$ and $\cT'$ is $\cT$ with $\cS$ removed. Using the diagram (\ref{hTfunc}) twice, we obtain the following commutative diagram:   
	\[ \xymatrix{ N(\cT) \ar[rr]^(0.35){\iota^{\cT}_{\cS} \times \iota^{\cT}_{\cT'} } \ar[d]_{h_\cT} & & N(\cS) \underset{N(\cS\cap \cT')}{\times}{} N(\cT') \ar[d]^{h_{\cS}\times h_{\cT'}} \\	
		h(C_\cT) \ar[rr]_(0.35){\iota^{C_\cT}_{C_{\cS}} \times \iota^{C_\cT}_{C_{\cT'}}} && h(C_\cS) \underset{h(C_{\cS\cap \cT'})}{\times}{} h(C_{\cT'}).}\] 
Now, the horizontal arrows in this diagram are bijections by Lemma \ref{fibretree} and Lemma \ref{fibrecheeses} respectively. Since $h_{\cS} \times h_{\cT'}$ is a bijection by the induction hypothesis, it follows that $h_{\cT}$ is a bijection as well.
	 \end{proof} 
\begin{rmk} We note that the bijectivity of $h_\cT$ in Proposition \ref{holesnbd} is more conceptually clear than our proof suggests. If $D$ is in $h(C_{\cT})$ then $\overline{\lambda(D\cap \Omega_F)}$ is a connected component $X_D$ of $|\cB\cT|\backslash |\cT|$. There is precisely one edge $e_D$ in $N(\cT)$ such that the interior of $|e_D|$ is contained in $X_D$. Then the inverse of $h_\cT$ sends $D$ to $e_D$. However it is not straightforward to make this argument rigorous in the context of this paper.  \end{rmk}

\begin{defn} \label{defnOmegan} Let $s_0$ be the vertex of $\mathcal{BT}$ fixed by $GL_2(\cO_F)$ and let $n\geq 0$.
\be \item $\cT_n \subset \mathcal{BT}$ is the subtree whose vertices have distance at most $n$ from $s_0$.
\item $\Omega_n$ is the cheese $\Omega_n:=C_{\cT_n}.$
\ee   \end{defn} 

\begin{rmk}\label{TreeStabs} Since $ G^0_{\cT_n}=G^0_{s_0}=GL_2(\cO_F)$ for all $n\geq 0$,  $\Omega_n$ is $GL_2(\cO_F)$-stable for all $n\geq 0$. \end{rmk}

\begin{rmk} \label{treecover} For any family $\{\cT_j\}_{j\in J}$ of finite subtrees of $\mathcal{BT}$ such that $\bigcup_{j\in J} |\cT_j|=|\mathcal{BT}|$, the family of cheeses $\{C_{\cT_j}\}_{j\in J}$ forms an admissible cover of $\Omega$. 
\end{rmk}

\begin{lem}Let $n \geq 0$. \label{fibres} \be \item $GL_2(\cO_F)$ acts transitively on $h(\Omega_n)$.
\item The fibres of the maps $\iota^{\Omega_{n+1}}_{\Omega_n}\colon h(\Omega_{n+1})\to h(\Omega_n)$ all have size $q$.
\ee \end{lem}
\begin{proof}
(a) By Proposition \ref{holesnbd} and Remark \ref{TreeStabs}, it suffices to prove that $GL_2(\cO_F)$ acts transitively on $N(\cT_n)$ for each $n\geq 0$. But $N(\cT_n)$ consists of all edges between vertices of distance $n$ from $s_0$ and vertices of distance $n+1$ from $s_0$. This holds because $GL_2(\cO_F)$ acts transitively on the set of vertices of distance $n+1$ from $s_0$.

(b) Note that $|h(\Omega_n)|=|N(\cT_n)|=(q+1)q^n$ since $\mathcal{BT}$ is a $(q+1)$-regular tree. The fibres of $\iota^{\Omega_{n+1}}_{\Omega_n}\colon h(\Omega_{n+1})\to h(\Omega_n)$ all have the same size, by part (a).   \end{proof}

We introduce some other admissible covers of $\Omega$ by $K$-cheeses, for later use.

Recall that $w=\begin{pmatrix} 0 & 1 \\ \pi_F & 0\end{pmatrix}\in GL_2(F)$ and  $w\cdot s_0$ is a vertex of $\mathcal{BT}$ adjacent to $s_0$.

\begin{defn}\label{defnPsin} Let $n\geq 0$. \be \item Let $e_0$ be the unique edge of $\mathcal{BT}$ with vertices $s_0$ and $w\cdot s_0$.
\item Let $\cS_n$ be the subtree of $\mathcal{BT}$ consisting of vertices a distance at most $n$ from either $s_0$ or $ws_0$.  \item Let $\Psi_n$ be the cheese $\Psi_n:=C_{\cS_n}$. \ee \end{defn}

 \begin{lem} For each $n\geq 1$, \be \item$\Psi_{n}=\Omega_n\cup w\Omega_n$, and \item $\Psi_{n-1}=\Omega_{n}\cap w\Omega_{n}$.  \ee \end{lem}
 
 \begin{proof} Let $n \geq 1$. It is clear that $\cS_n=\cT_n\cup w\cT_n$. We claim that $\cS_{n-1}=\cT_n\cap w\cT_n$.   For the forward inclusion, because $w^2$ acts trivially on $\mathcal{BT}$, it is enough to show that $\cT_{n-1} \subseteq w\cT_n$. Let $d$ be the distance function on $V(\mathcal{BT})$ and let $x \in V(\cT_{n-1})$. Then $d(x,s_0) \leq n-1$, so $d(x,ws_0) \leq d(x,s_0) + d(s_0,ws_0) \leq (n-1) + 1 = n$ and hence $x \in V(w\cT_n)$. For the reverse inclusion, it is enough to show that $\cT_n \cap w \cT_n \subseteq \cT_{n-1}$. Suppose that $x \in V(\cT_n \cap w \cT_n)$ so that $d(x,s_0) \leq n$ and $d(x,ws_0) \leq n$. By considering the unique path in $\mathcal{BT}$ passing through $x,s_0$ and $ws_0$, we see that we must have either $d(x,s_0) \leq n-1$ or $d(x,ws_0) \leq n-1$, and hence $x \in V(\cT_{n-1})$.
 
Both parts now follow easily. \end{proof}

 \begin{rmk}\label{ITreeStabs} Since,  for each $n\geq 0$, $G^0_{\cS_n}=G^0_{e_0}=I$ is the Iwahori subgroup from Notation \ref{IwahoriDefn}(b), each cheese $\Psi_n=C_{\cS_n}$ is $I$-stable.
 	\end{rmk}
 
\begin{lem}\label{iomegapsi} Suppose that $n\geq 0$. 
	
	\be \item $h(\Psi_n)$ has precisely two $I$-orbits, each of size $q^{n+1}$.
	
	\item The map $\iota^{\Psi_{n+1}}_{\Psi_{n}}\colon h(\Psi_{n+1})\to h(\Psi_{n})$ is surjective, with all fibres of size $q$.
	\ee
\end{lem}

\begin{proof}
(a) By Proposition \ref{holesnbd} it suffices to show that $N(\cS_n)$ has precisely two $I$-orbits each of size $q+1$. But $N(\cS_n)$ consists of those edges of $\cS_{n+1}$ that are not edges of $\cS_{n}$. These fall into those that connect vertices of distance $n$ and $n+1$ from $s_0$ (and distance $n+1$ and $n+2$ from $w\cdot s_0$) and those that connect vertices of distance $n$ and $n+1$ from $w\cdot s_0$ (and distance $n+1$ and $n+2$ from $s_0$). These two sets of edges in $N(\cS_n)$ are its $I$-orbits. 

(b) Using Proposition \ref{holesnbd} again, it suffices to prove the same thing about the fibres of the $I$-equivariant function $N(\cS_{n+1})\to N(\cS_n)$. This is straightforward to verify since $\mathcal{BT}$ is $q+1$-regular. 
\end{proof}
\begin{lem}	\label{coverOmega} The following collections of affinoid subdomains form admissible covers of $\Omega$:
	\be \item $\{\Omega_n\}_{n\geq 0}$;
	\item $\{w\Omega_n\}_{n\geq 0}$;
	\item $\{\Psi_n\}_{n\geq 0}$.
	\ee
\end{lem}
\begin{proof} 
	Each part is an easy consequence of Remark \ref{treecover}.
	\end{proof}

\begin{prop}\label{BasicUH} $\Omega$ is a smooth, geometrically connected, quasi-Stein rigid $K$-analytic space.
\end{prop}
\begin{proof} 
	We've seen that the chain $\Omega_0\subseteq \Omega_{1}\subseteq \cdots $ is an admissible cover of $\Omega$ by an increasing union of cheeses. Moreover the maps $h(\Omega_{n+1})\to h(\Omega_n)$ are all surjective by Lemma \ref{fibres}. Thus $\Omega$ is a smooth, geometrically connected, quasi-Stein rigid $K$-analytic space by Proposition \ref{admquasiStein}.
\end{proof}

\subsection{Units, measures and flat connections on \ts{\Omega}}

 Recall, for $\varphi\in \Aut(\bP^1)$ and cheeses $X$ and $Y$ with $\varphi(Y)\subseteq X$, the map $\varphi^X_Y\colon h(X)\to h(Y)$ from Lemma \ref{rYX} together with the notation $D_\infty$ to denote the element of $h(X)$ containing the point $\infty\in \bP^1(\bfC)$. 

\begin{prop}\label{UnitsMeasCheese} Let $X = C(\alpha, \mathbf{s})$ and $Y$ be cheeses and $\varphi\in \Aut(\bP^1)$ with $\varphi(Y)\subseteq X$. Then there is commutative diagram
\[\xymatrix{  1  \ar[r]& K^\times \cdot  \cO(X)^{\times\times} \ar[r] \ar[d] & \cO(X)^\times  \ar[r]^{\mu_X} \ar[d]  & M_0(h(X),\bZ)  \ar[r]\ar[d]_{\varphi^X_{Y,\ast}} & 0 \\
		1 \ar[r] & K^\times\cdot \cO(Y)^{\times\times} \ar[r]  & \cO(Y)^\times 
  \ar[r]_{\mu_Y} &  M_0(h(Y),\bZ) \ar[r] & 0}\]	
whose rows are short exact sequences of abelian groups and whose non-labelled vertical arrows are induced by the composite of  the restriction $\cO(X)\to \cO(\varphi(Y))$ and $\varphi^\sharp\colon \cO(\varphi(Y))\to \cO(Y)$. 

The map $\mu_X$ is characterised by $\mu_X(x-\alpha)=\delta_D-\delta_{D_\infty}$ for $D\in h(X)$ and $\alpha\in D(K)$.
\end{prop}
\begin{proof} For each $i=1,\ldots, g := g_X$, let $D_i\in h(X)$ be the open disc containing $\alpha_i$. Given $u \in \cO(X)^\times$, use Proposition \ref{CheeseUnits} to find integers $n_1,\ldots,n_g$ such that
 \[u \equiv (x-\alpha_1)^{n_1}\cdots (x - \alpha_g)^{n_g}  \quad\mod\quad K^\times\cdot \cO(X)^{\times\times}\]
 and define the measure $\mu_X(u) \in M_0(h(X),\bZ)$ by 
 \[\mu_X(u):=\sum_{i=1}^g n_i(\delta_{D_i} - \delta_{D_\infty}).\]
 The top row is then exact by Proposition \ref{CheeseUnits}. We note that  $\mu_X$ does not depend on the choice of the centres $\alpha_1,\ldots,\alpha_g$ of the holes of the cheese $X$. 
 
 Since $Y$ is also a cheese the bottom row is also exact. The commutativity of the left-hand square is clear. 
 
 To see the right-hand square commutes it suffices by the argument just given to show that for all $i=1,\cdots,g$, we have  
  \[\varphi^X_{Y,\ast}\mu_X(x-\alpha_i)=\mu_Y(\varphi^{\sharp}(x-\alpha_i)).\]
 
Now, $\varphi^X_{Y,\ast} \mu_X(x-\alpha_i)=\delta_{\varphi^X_Y(D_i)}- \delta_{\varphi^X_Y(D_\infty)}$, and $\varphi^{\sharp}(x-\alpha_i)$ is a rational function with divisor $(\varphi^{-1}(\alpha_i))-(\varphi^{-1}(\infty))$. So, since $\varphi_Y^X(D_i)\in h(Y)$ contains $\varphi^{-1}(\alpha_i)\in \varphi^{-1}(D_i)$ and $\varphi_Y^X(D_\infty)\in h(Y)$ contains $\varphi^{-1}(\infty)\in \varphi^{-1}(D_\infty)$, we see that
  \[ \mu_Y(\varphi^{\sharp}(x-\alpha_i))=\delta_{\varphi^X_Y(D_i)}- \delta_{\varphi^X_Y(D_\infty)} = \varphi^X_{Y,\ast} \mu_X(x-\alpha_i). \qedhere\]
 \end{proof}

Because of Proposition \ref{PicConDtorG}, we are interested in the groups $\frac{\cO(X)^\times}{K^\times} \underset{\bZ}{\otimes}{} \frac{\frac{1}{d} \bZ}{\bZ}$ for positive integers $d$. 


\begin{cor}\label{UnitsModD} Let $X$ be a cheese and let $d$ be an integer. 
\be \item The map $\mu_X$ induces a surjective homomorphism 
\[ \mu_{X,d} : \frac{\cO(X)^\times}{K^\times\cO(X)^{\times d}}  \twoheadrightarrow M_0\left(h(X), \bZ/d \bZ \right).\] 
\item If $G\to \Aut(\bP^1)_X$ is a group homomorphism,  then $\mu_{X,d}$ is $G$-equivariant. 
\item If $p\nmid d$ then $\mu_{X,d}$ is an isomorphism. \ee
\end{cor}
\begin{proof} (a) Proposition \ref{UnitsMeasCheese} gives us an exact sequence of abelian groups
\[1 \to \cO(X)^{\times\times}/K^{\times\times} \to \cO(X)^\times / K^\times \stackrel{\mu_X}{\longrightarrow} M_0(h(X), \bZ) \to 0.\]
Tensoring this sequence with $\bZ/d\bZ$ gives an exact sequence
\begin{equation}\label{muXd} \frac{\cO(X)^{\times\times}}{K^{\times\times}}\underset{\bZ}{\otimes} \bZ/d\bZ\to \frac{\cO(X)^\times}{K^\times} \underset{\bZ}{\otimes}{} \bZ/d\bZ \stackrel{\mu_X \otimes 1}{\longrightarrow} M_0(h(X), \bZ) \underset{\bZ}{\otimes}{} \bZ/d\bZ \to 0.\end{equation}
The second term is $\cO(X)^\times/K^\times\cO(X)^{\times d}$ and the third term is $M_0(h(X), {\bZ}/d\bZ )$ by Lemma \ref{M0modD}. 

(b) This part follows easily from Proposition \ref{UnitsMeasCheese}.  

(c) Since $p\nmid d$,  the first term in (\ref{muXd}) vanishes by Lemma \ref{Ddivis}, \end{proof}
\begin{cor}\label{ConCheeseMeas} Let $X$ be a cheese, $d$ is an integer such that $p\nmid d$ and suppose that $G\to \Aut(\bP^1)_X$ is a group homomorphism. Then  
\[ \mu_{X,d}\circ\theta_d\colon \Con(X)^G[d]\to M_0\left(h(X),\bZ/d\bZ\right)^G\] is an isomorphism.   
\end{cor}
\begin{proof} 
This follows immediately from Proposition \ref{PicConDtorG} and Corollary \ref{UnitsModD}.
\end{proof}
We will now use Corollary \ref{ConCheeseMeas} to investigate how the group $\Con^G(X)[d]$ changes when we vary $X$ and $G$. More precisely, we have the following
\begin{prop}\label{RestCon} Let $Y \subseteq X$ be cheeses such that $\iota^X_Y : h(X) \to h(Y)$ is surjective and let $d \geq 1$ be an integer such that $p \nmid d$.  
	\be \item Suppose that 
	\begin{enumerate}[{(}i{)}]
		\item each fibre of $\iota^X_Y\colon h(X)\to h(Y)$ has size coprime to $d$, and
		\item the $G$-orbits in $h(X)$ are unions of these fibres.
	\end{enumerate}
	Then the following restriction map is injective:
	\[\Con(X)[d]^G \hookrightarrow \Con(Y)[d].\]
	\item Suppose that additionally to the assumptions in (a),
	\begin{enumerate}[{(}i{)}]
		\setcounter{enumii}{2}
		\item $H$ is a closed subgroup of $G_Y$, and
		\item the restriction map $\Hom(G,\mu_d(K))\to \Hom(H,\mu_d(K))$ is injective.
	\end{enumerate}
	Then the following restriction map is injective: 
	\[\Con^G(X)[d]\hookrightarrow \Con^H(Y)[d].\]
	\item Suppose that additionally to the assumptions in (b),
	\begin{enumerate}[{(}i{)}]
		\setcounter{enumii}{4}
		\item $\iota^X_Y\colon h(X)\to h(Y)$ induces a bijection between the $G$-orbits in $h(X)$ and the $H$-orbits in $h(Y)$,  
		\item the map $\Hom(G,\mu_d(K))\to \Hom(H,\mu_d(K))$ is surjective, and
		\item $\omega\colon \Con^{G}(X)[d]\to \Con(X)[d]^G$ is surjective. 
	\end{enumerate}
	Then the following restriction map is an isomorphism: 
	\[\Con^G(X)[d] \stackrel{\cong}{\longrightarrow} \Con^H(Y)[d].\]
	\ee
\end{prop}

\begin{proof} (a) By Corollary \ref{ConCheeseMeas} there is a commutative diagram \begin{eqnarray}\label{rest} \xymatrix{ \Con(X)[d]^G \ar[rr]^{\mu_{X,d}\circ\theta_d}_\cong \ar[d] &&  M_0\left(h(X),{\bZ}/d{\bZ}\right)^G \ar[d]^{\iota^X_{Y,\ast}}\\
			\Con(Y)[d] \ar[rr]_{\mu_{Y,d}\circ\theta_d}^\cong && M_0\left(h(Y),\bZ/d\bZ\right) }\end{eqnarray}
	whose left-vertical arrow is restriction and whose horizontal arrows are isomorphisms. Thus it suffices to prove that $\iota^X_{Y,\ast}\colon M_0\left(h(X),\bZ/d\bZ\right)^G \to M_0\left(h(Y),\bZ/d\bZ\right)$ is injective. Suppose that $\nu$ is in the kernel. Then for $D\in h(Y)$, 
	\[ 0=\iota^X_{Y,\ast}\nu(\{D\})=\nu((\iota^X_Y)^{-1}\{D\}). \]
	Since $\iota^X_Y$ is surjective by assumption, we may choose some $D' \in h(X)$ such that $\iota^X_Y(D') = D$. Because $\nu$ is $G$-invariant, assumption (ii) implies that
	\[\nu\left((\iota^X_Y)^{-1}\{D\}\right)=|(\iota^X_Y)^{-1}(D)|\cdot \nu(\{D'\}).\]
	Then assumption (i) gives $\nu(\{D'\})=0$, so $\nu=0$ as required. 
	
	(b) Using Lemma \ref{PicSeq} together with Proposition \ref{PicCheese} and (iii), we have the commutative diagram \begin{eqnarray}\label{rest2}
		\xymatrix{ 0 \ar[r] &\Hom(G,\mu_d(K)) \ar[d]\ar[r]&  \Con^G(X)[d] \ar[r]\ar[d] & \Con(X)[d]^G \ar[d]  \\ 0 \ar[r] & \Hom(H,\mu_d(K))\ar[r] & \Con^H(Y)[d] \ar[r] & \Con(Y)[d]^H } \end{eqnarray} with exact rows, whose vertical arrows are given by restriction. Using (iv), part (a) and the Four Lemma, we see that the middle arrow is injective.
	
	(c) Assumption (v) implies that the right vertical map in (\ref{rest}) has image equal to $M_0\left(h(Y),\bZ/d\bZ\right)^H$, so the right vertical arrow in (\ref{rest2}) is an isomorphism. Using this together with (vii) gives that both of the rightmost horizontal arrows in diagram (\ref{rest2}) are surjective. We can now use (vi) and the Five Lemma to finish the proof.
\end{proof}

The following technical Lemma is needed for the important Corollary \ref{KilledByq+1} below. Recall the $K$-cheeses $C_{\cT}$ from Definition \ref{KcheeseCT}.
\begin{lem} \label{subtreeunitfilt} Suppose that $\cT'\subseteq \cT$ are finite subtrees of $\mathcal{BT}$ with $N(\cT')\subseteq \cT$. Then for every $n\geq 1$ and $f\in K^\times \cO(C_\cT)^{\times\times}_{|\pi_F|^n}$, we have $f|_{C_{\cT'}}\in K^\times\cO(C_{\cT'})^{\times\times}_{|\pi_F|^{n+1}}$.
\end{lem}

\begin{proof}
	 We claim first that for each $a\in F$ and $n\in\{\pm 1\}$, \begin{equation} \left|(x-a)^n \right|_{C_{\cT'}}\leq |\pi_F|\left|(x-a)^n\right|_{C_\cT}.\label{estlinearfunct} \end{equation}
	 
	  By a change of coordinate induced by an element of $GL_2(F)$ we may reduce the proof of this claim to proving $|\pi_F|^{-1}\leq |x|_{C_\cT}$ in the particular case $|x|_{C_{\cT'}}=1$. Now if $|x|_{C_{\cT'}}=1$, then $s_0\in \cT'$ and so, by hypothesis, \[\cT_1\subseteq \cT'\cup N(\cT')\subseteq \cT.\] Then because $\Omega_1=C_{\cT_1}$, we have $|x|_{C_\cT}\geq |x|_{\Omega_1}\geq |\pi_F|^{-1}$ which proves the claim.
	  
	Now suppose that $f\in K^\times \cO(C_{\cT})_{|\pi_F|^n}^{\times\times}$ so that $f=\lambda(1+h)$ for some $\lambda\in K^\times$ and $h\in \pi^n_F\cO(C_{\cT})^\circ$. We have to show that $1+h\in K^\times \cO(C_{\cT'})^{\times\times}_{|\pi_F|^{n+1}}$.  By Proposition \ref{CheeseSchauder}, we can write \[1+ h= (1+\lambda_0) + \sum_{i=0}^g \sum_{j\geq 1} \lambda_{ij}\xi_i^j\] with $\lambda_0, \lambda_{ij}\in \pi_F^n\cO_K$ and $\xi_0,\ldots,\xi_g$ each of the form $c(x-a)$ or $\frac{c}{x-a}$ with $a,c\in F$ and $c\neq 0$ and $|\xi_i|_{C_\cT}=1$. Since $(1+\lambda_0)\in K^{\times\times}$, by considering $(1+\lambda_0)^{-1}(1+h)$ we may further assume that $\lambda_0=0$ and then it suffices to prove that $|h|_{C_{\cT'}}\leq |\pi_F^{n+1}|$.  
	Now by (\ref{estlinearfunct}), for all suitable $i,j$ we have \[|\xi_i^j|_{C_{\cT'}}=|\xi_i|^j_{C_{\cT'}}\leq |\pi_F|^j\leq |\pi_F|\] so the result follows by the ultrametric inequality. \end{proof}

	
	 


Recall the $K$-cheeses $\Omega_n$ from Definition \ref{defnOmegan}. 
\begin{cor} \label{omegaunitfilt} Suppose that $n,m\geq 0$. Then for all $f\in K^\times \cO(\Omega_{n+m})^{\times\times}$,
\[f|_{\Omega_n}\in K^\times\cO(\Omega_n)^{\times\times}_{|\pi_F|^m}.\]
\end{cor}
\begin{proof}
	Since $N(\cT_{n+k})\subseteq \cT_{n+k+1}$ for all $n,k\geq 0$, this follows from Lemma \ref{subtreeunitfilt} by a straightforward induction on $m$. 
\end{proof}

\begin{prop}\label{propomegan} Write $A := GL_2(\cO_F)$.
	\be \item For all $n\geq 0$, the restriction map \[\Con(\Omega_{n+1})^{A}[p']\to \Con(\Omega_{n})^{A}[p']\] is an isomorphism. These groups are cyclic of order $q+1$. 
	
\item There is $m\geq 1$ such that the restriction maps \[ \Con(\Omega_{m+n})^{A}[p]\to \Con(\Omega_n)^{A}[p]\] are zero for all $n\geq 0$. 
\ee \end{prop}
\begin{proof} (a) Suppose $d$ is an integer coprime to $p$ and that $d$ is a multiple of $(q+1)=|h(\Omega_0)|$. Then by Corollary \ref{ConCheeseMeas} and Proposition \ref{UnitsMeasCheese}, for each $n\geq 1$, there is a commutative diagram \[ \xymatrix{ \Con(\Omega_n)^{A}[d] \ar[r] \ar[d] & M_0(h(\Omega_n),\bZ/d\bZ)^{A} \ar[d] \\ \Con(\Omega_{n-1})^{A}[d] \ar[r] & M_0(h(\Omega_{n-1}),\bZ/d\bZ)^{A}}\] whose horizontal maps are isomorphisms. Since $A$ acts transitively on each $h(\Omega_n)$, by Lemma \ref{fibres}(a), we see by Proposition \ref{M0Gfinite}(b) that  $M_0(h(\Omega_n),\bZ/d\bZ)^{A}$ is cyclic of order $\gcd(d,|h(\Omega_n)|)=q+1$  and generated by the image of $\frac{d}{q+1}\Sigma_{h(\Omega_n)}$. Moreover by Lemma \ref{fibres}(b) together with Proposition \ref{M0Gfinite}(c), the right-hand vertical map sends the image of $\frac{d}{q+1}\Sigma_{h(\Omega_n)}$ in $M_0(h(\Omega_n),\bZ/d\bZ)^{A}$ to the image of $\frac{qd}{q+1}\Sigma_{h(\Omega_{n-1})}$ in $M_0(h(\Omega_{n-1}),\bZ/d\bZ)^{A}$. Since $q$ is coprime to $d$, it follows that the map is an isomorphism. Part (a) now follows easily.

(b) We take $m\geq 1$ such that $|\pi_F^{m-1}|<\varpi$ and let $N=n+m$. Suppose that \[[\sL]\in \Con(\Omega_{N})^{A}[p].\]  We will show that $[\sL|_{\Omega_n}]=[\cO]\in \Con(\Omega_n)$. By Proposition \ref{PicConDtorG} there is $u\in \cO(\Omega_{N})^\times$ such that 
\[ \theta_p([\sL])=uK^\times \cO(\Omega_{N})^{\times p}\in \left(\frac{\cO(\Omega_{N})^\times}{K^\times \cO(\Omega_{N})^{\times p}}\right)^{A}. \] It suffices to show that $u|_{\Omega_n}\in K^\times\cO(\Omega_n)^{\times p}$. 
Now $\mu_{\Omega_N,p}(u)\in M_0(h(\Omega_{N}),\bZ/p\bZ)^{A}$ by Corollary \ref {UnitsModD}(b), and by Proposition \ref{M0Gfinite}(b,c) and Lemma \ref{fibres}(b), the natural map induced by the inclusion $\Omega_N\subset \Omega_{N-1}$ \[M_0(h(\Omega_{N}),\bZ/p\bZ)^{A}\to M_0(h(\Omega_{N-1}),\bZ/p\bZ)\]is zero. It follows, using Proposition \ref{UnitsMeasCheese}, that there is $v\in \cO(\Omega_{N-1}^\times)$ such that $\mu_{\Omega_{N-1}}\left(u|_{\Omega_{N-1}}v^p\right)=0$. Writing \[w:=u|_{\Omega_{N-1}}v^p,\] to prove $u|_{\Omega_n}\in K^\times\cO(\Omega_n)^{\times p}$ it suffices to show that $w|_{\Omega_n}\in K^\times\cO(\Omega_n)^{\times p}$.  Since $\mu_{\Omega_{N-1}}(w)=0$, Proposition \ref{UnitsMeasCheese} now implies that $w\in K^\times\cO(\Omega_{n+m-1})^{\times\times}$.
 
 Now by Corollary \ref{omegaunitfilt}, \[w|_{\Omega_n}\in K^\times\cO(\Omega_n)^{\times\times}_{|\pi_F|^{m-1}}.\] Our assumption that $|\pi_F^{m-1}|<\varpi$ now allows us to deduce from Lemma \ref{Ddivis}(b) that $w|_{\Omega_n}\in K^\times\cO(\Omega_n)^{\times p}$ as required. 
\end{proof}

We can now compute $\PicCon(\Omega)_{\tors}^{GL_2(\cO_F)}$.

\begin{cor}\label{KilledByq+1} The group $\PicCon(\Omega)^{GL_2(\cO_F)}_{\tors}$ is cyclic of order $q+1$.  
\end{cor}
\begin{proof}
By Proposition \ref{BasicUH}, Proposition \ref{PicConinvlim}, and Proposition \ref{PicCheese} we have
\[\PicCon(\Omega)\cong \invlim \Con(\Omega_n).\]
Since each $\Omega_n$ is $A:=GL_2(\cO_F)$-stable by Remark \ref{TreeStabs}, and the functors taking $A$-invariants and taking the $d$-torsion subgroup each commute with limits it follows that for each $d\geq 1$ we have 
 \[ \PicCon(\Omega)^{A}[d]\cong \invlim \Con(\Omega_n)^{A}[d].\]
By Proposition \ref{propomegan}(b), $\invlim \Con(\Omega_n)^{A}[p]=0$. So $\PicCon(\Omega)^{A}$ has no $p$-torsion.

By Proposition \ref{propomegan}(a), we can see that for each $d$ that is a multiple of $q+1$, $\invlim \Con(\Omega_n)^{A}[d]$ is cyclic of order $q+1$. The result follows.
\end{proof}
\subsection{Proof of Theorem A}\label{MainProofSec}
We now return to the setting of $\S \ref{SecCocycELBC}$, and start working towards our proof of Theorem A. Recall the cheeses $\Omega_n$ from Definition \ref{defnOmegan}(b) and the map $\phi_z$ from Proposition \ref{phiz}. 

$\Omega_{F,n}$ will denote the version of $\Omega_n$ obtained when $K = F$.

 \begin{thm} \label{phizisolocal} Let $L$ be an unramified quadratic extension of $F$. Then for every $z \in \Omega_{F,0}(L)$ and every $n\geq 0$, the map \[\phi_z\colon \Con^{GL_2(\cO_F)}(\Omega_n)[p']\to \Hom(GL_2(\cO_F)_z,K(z)^\times)[p']\] is an isomorphism.
\end{thm}

\begin{proof} Note that because $z \in \Omega_{F,0}(L)$, we may view it as a point of $\Omega_{F,0}(K(z))  = \Omega_0(K(z)) \subseteq \Omega_n(K(z))$. Hence the map $\phi_z$ from Proposition \ref{phiz} makes sense in this setting. Write $A:=\GL_2(\cO_F)$. By Proposition \ref{PicSeq} together with the left exactness of the endofunctor $(-)[p']$ on abelian groups, there is an exact sequence 
	 	 \[ 0\to \Hom(A,K^\times)[p']\to \Con^{A}(\Omega_n)[p']\to \Con(\Omega_n)^{A}[p']. \]  The group $\Con(\Omega_n)^{A}[p']$ is cyclic of order $q+1$ by Proposition \ref{propomegan}(a), whereas $\Hom(A,K^\times)[p']$ is cyclic of order $q-1$ by Lemma \ref{Hommup'}(a). Thus, \[\left|\Con^A(\Omega_n)[p']\right|\leq q^2-1.\] 
Since $z\in \cO_{F,0}(L) \subseteq \Omega_F(\overline{F})$ by assumption, we can apply Lemma \ref{stabrel} to see that $A_z=G^0_z$. Also, $z \in \Omega_{F,0}(L) \subseteq \Omega_F(L) = L \backslash F$ implies that $F(z) = L$ is a quadratic extension of $F$, so Lemma \ref{HomG0z} can be applied to deduce that $\Hom(A_z,K(z)^\times)[p']$ is cyclic of order $q^2-1$. Now it suffices to show that the image of $\phi_z$ contains a generator of this cyclic group. To this end we will construct an element $[\sL]$ of $\Con^A(\Omega_n)[p']$ such that, in the notation of Lemma \ref{HomG0z}, $\phi_z([\sL])=\widehat{j_z}^{q^n}$. To do this, we will first construct a suitable unit $u \in \cO(\Omega_n)^\times$, then an appropriate $1$-cocycle $\alpha \in \cZ^{A,\Omega_n}_{u,q+1,q-1}$ and then the required equivariant line bundle $\sL$ is given by an application of Lemma \ref{1cocycles}(a).
	
Consider the function $j\colon A\to \cO(\Omega_n)^{\times}$ given by \[ j\left(\begin{pmatrix} a & b \\ c & d\end{pmatrix}\right)= {a-cx}. \] We compute that \begin{eqnarray*} j\left(\begin{pmatrix} a_1 & b_1 \\ c_1 & d_1\end{pmatrix}\right) \begin{pmatrix} a_1 & b_1 \\ c_1 & d_1\end{pmatrix}\cdot j\left(\begin{pmatrix} a_2 & b_2 \\ c_2 & d_2\end{pmatrix}\right) & = & ({a_1-c_1x})\left(a_2-c_2{\frac{d_1x-b_1}{a_1-c_1x}}\right) \\ & = & {(a_1a_2+b_1c_2)-(c_1a_2+c_2d_2)x} \\ & = & j\left(\begin{pmatrix} a_1 & b_1 \\ c_1 & d_1\end{pmatrix}\begin{pmatrix} a_2 & b_2 \\ c_2 & d_2 \end{pmatrix}\right)\end{eqnarray*} and see that $j\in Z^1(A,\cO(\Omega_n)^\times)$. The reason for considering this $1$-cocycle $j$ is that 
     \[\mu_{\Omega_n}(j(g))=\delta_{gD_\infty}-\delta_{D_\infty} \qmb{for all} g\in A, \qmb{and} z\circ j|_{A_z} \equiv \widehat{j_z} \qmb{mod} L^{\times\times}.\]   
Now we define $\nu:=|h(\Omega_n)|\delta_{D_\infty}-\Sigma_{\Omega_n}\in M_0(h(\Omega_n),\bZ)$. Applying Proposition \ref{UnitsMeasCheese}, we find $u\in \cO(\Omega_n)^\times$ such that $\mu_{\Omega_n}(u)=\nu$.
	Then we calculate that
	\[\delta_A(\nu)(g)=g\cdot\nu - \nu = |h(\Omega_n)|(\delta_{gD_\infty}-\delta_{D_\infty}) \qmb{for all} g \in A.\] 
	Therefore inside $Z^1(A, M_0(h(\Omega_n),\bZ))$ we have the equality
	\[ \mu_{\Omega_n}\circ j^{|h(\Omega_n)|}= \delta_{A}(\nu)=\mu_{\Omega_n}\circ \delta_{A}(u).\] 
	Since $|h(\Omega_n)| = q^n(q+1)$, this means that $j^{-q^n(q+1)} \delta_{A}(u)$ takes values in $\ker \mu_{\Omega_n}$. Now Proposition \ref{UnitsMeasCheese} tells us that $\ker \mu_{\Omega_n}=K^\times \cdot \cO(\Omega_n)^{\times\times}$. So we may rephrase this as saying that 
	\[\pi_{T(\Omega_n)}\circ(j^{-q^n(q+1)}\delta_G(u))\] 
	takes values in $K^\times/K^{\times\times}$. Since $\Omega_n$ is geometrically connected, $A$ is compact and every finite abelian $p'$-quotient of $A$ has exponent dividing $q-1$, by Remark \ref{p'exp}, we may apply Proposition \ref{approxcocycle} with $(d,e,u,\beta)=(q+1,q-1,u,j^{q^n})$ to deduce that there exists an $\alpha\in \cZ^{A,\Omega_n}_{u,q+1,q-1}$ such that 
	\begin{equation}\label{alphajqn}\pi_{T(\Omega_n)}\circ \alpha=\pi_{T(\Omega_n)}\circ j^{q^n}.\end{equation}	
	By Lemma \ref{1cocycles}(a), there is a $(q^2-1)$-torsion $A$-equivariant line bundle with connection $\sL^\alpha_{u,q+1}$ on $\Omega_n$, such that $\phi_{\Omega_n}^A([\sL^\alpha_{u,q+1}])=[\alpha]$ inside $H^1(A,\cO(\Omega_n)^\times)$. 
	
	To see what $\phi_z$ does to this $[\sL^\alpha_{u, q+1}]$, we apply Proposition \ref{defphiz}(b) to find that
\[\phi_z([\sL^\alpha_{u,q+1}])= z \circ (\res^A_{A_z} \phi^A_{\Omega_n}([\sL^\alpha_{u,q+1}]) ) = z\circ \alpha|_{A_z}.\]  
Applying the functor $T(-)$ from Notation \ref{TXnotn} to the morphism of affinoid varieties $z : \Sp L \hookrightarrow \Omega_n$ and using equation (\ref{alphajqn}), we see that
\[z\circ\alpha|_{A_z} \equiv z\circ j^{q^n}|_{A_z} \mod K(z)^{\times\times}.\] 
 	But $z\circ j\left(\begin{pmatrix} a & -cN(z) \\ c & a-c\tr(z)\end{pmatrix} \right) = a-cz$, so as $z\circ \alpha|_{A_z}$ takes values in $\mu_{q^2-1}(K(z)^{\times})$, we conclude that inside $\Hom(A_z,K(z)^{\times})[p']$ we have
	\[\phi_z([\sL^\alpha_{u,q+1}])=z\circ \alpha|_{A_z} = \widehat{j_z}^{q^n}\]as claimed earlier. This is a generator because $q^n$ is coprime to $q^2-1$. 
	\end{proof}
Using our next Lemma, we will be able to use Theorem \ref{phizisolocal} to shed light on our main group of interest, namely $\PicCon^{G^0}(\Omega)_{\tors}$. See Corollary \ref{fibreproduct} below for a description of the $p'$-torsion part of $\PicCon^{G_0}(\Omega)$.

\begin{lem}\label{PicConGOmega} Let $A = \GL_2(\cO_F)$ and $B = {}^w \GL_2(\cO_F)$. \be \item $\PicCon^A(\Omega)\quad \cong\quad \invlim \Con^A(\Omega_n)$.
		\item $\PicCon^{G^0}(\Omega)\quad \cong\quad\PicCon^{A}(\Omega)\underset{\PicCon^I(\Omega)}{\times}{}\PicCon^{B}(\Omega)$. \ee
\end{lem}
\begin{proof} We note by Proposition \ref{BasicUH} and Lemma \ref{coverOmega}(a), $\Omega$ is a smooth, geometrically connected, quasi-Stein space with admissible cover  $\{\Omega_n\}$. Thus	(a) follows from Lemma \ref{PicConGlim} together with Remark \ref{TreeStabs} and Proposition \ref{PicCheese}.
	
	(b) This follows from Proposition \ref{PicConamal} and Theorem \ref{AmalgTrees}. \end{proof}

Recall the $K$-cheeses $\Psi_n$ from Definition \ref{defnPsin}. Since $\Pic(\Psi_n)=0$ by Proposition \ref{PicCheese}, there are restriction maps $\PicCon^I(\Omega)\to \Con^I(\Psi_n)$ for all $n\geq 0$. 

\begin{cor} \label{restpsi0} The restriction map $\PicCon^I(\Omega)[p']\to \Con^I(\Psi_0)[p']$ is injective.
\end{cor}

\begin{proof} By  Lemma \ref{PicConGlim} and Lemma \ref{coverOmega}(c), it suffices to show that the restriction map $\Con^I(\Psi_{n+1})[p']\to \Con^I(\Psi_{n})[p']$ is injective for all $n \geq 1$. Fixing $n \geq 1$, this is equivalent to $\Con^I(\Psi_{n+1})[d]\to \Con^I(\Psi_{n})[d]$ being injective for each $d$ coprime to $p$. We will prove this using Proposition \ref{RestCon}(b). 	

Condition (i) of Proposition \ref{RestCon} follows from Lemma \ref{iomegapsi}(c). Condition (ii) holds since the induced map on $I$-orbits $h(\Psi_{n+1})/I\to h(\Psi_n)/I$ is surjective and hence injective by Lemma \ref{iomegapsi}(c),(a).   Conditions  (iii) and (iv) are trivial since in this case $G=H=I$. Thus $\Con^{I}(\Psi_n)[d]\to \Con^I(\Psi_{n-1})[d]$ is injective.\end{proof}

\begin{prop} \label{twistw} For every $[\sL]\in \PicCon^{GL_2(\cO_F)}(\Omega)[p']$ there is an integer $k$ such that the restriction $\sL|_I$ satisfies 
\[ [\sL|_I] \cdot w[\sL|_I]=[\cO_{\widehat{\det}^k}] \qmb{in} \PicCon^I(\Omega).\]
\end{prop}
\begin{proof} We restrict $\sL|_I$ further to $\Psi_0$, forming $[\sL|_{I,\Psi_0}] \in \Con^I(\Psi_0)[p']$. By Corollary \ref{restpsi0}, it suffices to show that inside in $\Con^I(\Psi_0)$ we have 
\[[\sL|_{I,\Psi_0}] \cdot w[\sL|_{I,\Psi_0}]=[\cO_{\widehat{\det}^k}] \qmb{for some} k\in \frac{\bZ}{(q-1)\bZ}.\]
We consider the exact sequence coming from Lemma \ref{PicSeq}
	\begin{equation}\label{ConIPsi0} 1\to \Hom(I,K^\times)[p']\to \Con^I(\Psi_0)[p']\stackrel{\omega}{\to} \Con(\Psi_0)^I[p']. \end{equation}
Note that $\omega([\sL]) \in \PicCon^{\GL_2(\cO_F)}(\Omega)_{\tors}$ is killed by $q+1$ by Corollary \ref{KilledByq+1}. Therefore the image $\omega([\sL|_{I,\Psi_0}])$ of this class in $\Con(\Psi_0)^I$ is also killed by $q+1$.

 Since $w$ normalises $I$ and $\Psi_0$ is $\langle w, I\rangle$-stable, Corollary \ref{UnitsModD} and Proposition \ref{PicConDtorG} gives us an isomorphism of groups with $\langle w\rangle$-action \[\mu_{\Psi_0,q+1}\circ\theta_{q+1}\colon \Con(\Psi_0)^I[q+1] \stackrel{\cong}{\longrightarrow} M_0\left(h(\Psi_0),\frac{\bZ}{(q+1)\bZ}\right)^I.\] 
Next, $h(\Psi_0)$ has two $I$-orbits $\cO_1$ and $\cO_2$ of size $q$ by Lemma \ref{iomegapsi}(a). Hence $M_0\left(h(\Psi_0),\frac{\bZ}{(q+1)\bZ}\right)^I$ is generated by the image of $\Sigma_{\cO_1}-\Sigma_{\cO_2}$. Since $w$ swaps the two orbits, it acts on $M_0\left(h(\Psi_0),\frac{\bZ}{(q+1)\bZ}\right)^I$ by negation. Hence $w$ acts on $\Con(\Psi_0)^I[q+1]$ by inversion, so that
\[\omega\left([\sL|_{I,\Psi_0}] \cdot w[\sL|_{I,\Psi_0}]\right) = \omega\left([\sL|_{I,\Psi_0}]\right) \cdot w \omega\left([\sL|_{I,\Psi_0}]\right) = [\cO]\] is the trivial element of $\Con(\Psi_0)^I[q+1]$. 
The exact sequence (\ref{ConIPsi0}) above now implies that $[\sL|_{I,\Psi_0}]\cdot w [\sL|_{I,\Psi_0}]=[\cO_\chi]$ for some $\chi\in \Hom(I,\mu_{p'}(K))$. 
 
Finally, since $w^2\in Z(GL_2(F))$ acts trivially on $\Psi_0$ and $\Hom(I,\mu_{p'}(K))$, \[w \cdot ([\sL|_{I,\Psi_0}] \cdot w[\sL|_{I,\Psi_0}])=[\sL|_{I,\Psi_0}] \cdot w[\sL|_{I,\Psi_0}]. \] 
	Hence $[\cO_{w\chi}]=w \cdot [\cO_\chi] =[\cO_{\chi}]$, which implies that \[\chi\in \Hom(I,\mu_{p'}(K))^{\langle w\rangle}.\] Now Lemma \ref{Hommup'}(c) completes the proof. 
\end{proof}

\begin{cor}\label{fibreproduct} The following restriction map is an isomorphism of groups:
\[ \PicCon^{G^0}(\Omega)[p']\to \PicCon^{GL_2(\cO_F)}(\Omega)[p'].\]
\end{cor}
\begin{proof} Let $A:=GL_2(\cO_F)$ and $B={}^w A$. The commutative diagram \[ \xymatrix{\PicCon^{G^0}(\Omega)\ar[r] \ar[d] & \PicCon^{A}(\Omega) \ar[d]   \\
	\PicCon^{B}(\Omega) \ar[r] & \PicCon^I(\Omega)}\]
	 maps given by restriction is a pullback square by Lemma \ref{PicConGOmega}(b). Since taking $p'$-torsion preserves limits in the category of abelian groups it follows that \[ \xymatrix{\PicCon^{G^0}(\Omega)[p']\ar[r]^{p_1} \ar[d]_{p_2} & \PicCon^{A}(\Omega)[p'] \ar[d]^{q_2}   \\
	 	\PicCon^{B}(\Omega)[p'] \ar[r]_{q_1} & \PicCon^I(\Omega)[p']}\]
is also a pullback square and so the diagram \[ \xymatrix{\PicCon^{G^0}(\Omega)[p']\ar[r]^{p_1} \ar[d]_{p_2} & \PicCon^{A}(\Omega)[p'] \ar[d]^{q_2}   \\
	\PicCon^{B}(\Omega)[p'] \ar[r]_{q_1} & \im q_1+\im q_2}\] is a pullback square as well. Since pullbacks preserve isomorphisms, it suffices to see that in the last diagram, we have $\im q_2\subseteq \im q_1$ and that $q_1$ is injective. 
We consider the commutative diagram \[\xymatrix{ 1 \ar[r] &  \Hom(A,K^\times)[p'] \ar[r] \ar[d] & \PicCon^{A}(\Omega)[p'] \ar[d]_{q_2} \ar[r] & \PicCon(\Omega)^A[p'] \ar[d] \\
1 \ar[r] &  \Hom(I,K^\times)[p']  \ar[r] & \PicCon^I(\Omega)[p'] \ar[r] & \PicCon(\Omega)^I[p'] }	 \] whose rows are exact by Proposition \ref{PicSeq}. The left vertical map is injective by Lemma \ref{Hommup'}(a,b) and the right vertical map is an inclusion map. Now the injectivity of $q_2$ follows from the Snake Lemma. Therefore $q_1$ is also injective, because $q_2(w[\sL])=wq_1([\sL])$ for every $[\sL]\in \PicCon^{B}(\Omega)$.

Finally, by Proposition \ref{twistw}, $q_2([\sL])=q_1(w[\sL]^{-1}\otimes \cO_{\widehat{\det}^k})$ for some integer $k$. Hence the image of $q_2$ is contained in the image of $q_1$.
\end{proof}

\begin{rmk} One may wonder if it might be possible to strengthen the statement of Corollary \ref{fibreproduct} to give a similar description of \emph{all} torsion elements in $\PicCon^{G_0}(\Omega)$. However when $q=2$, the restriction map \[\PicCon^{G^0}(\Omega)_{\tors}\to \PicCon^{GL_2(\cO_F)}(\Omega)_{\tors}\] is not an isomorphism in general, because the homomorphism from the abelianization of $GL_2(\cO_F)$ to the abelianization of $G^0$ induced by inclusion has a kernel of order $2$ and so particular the restriction map \[\Hom(G^0,K^\times)_{\tors}\to \Hom(GL_2(\cO_F),K^\times)_{\tors}\] then may not be surjective.  \end{rmk}

Next we pass to the limit as $n \to \infty$ to deduce the consequences of Theorem \ref{phizisolocal} for the $p'$-torsion part of our main group of interest, namely $\PicCon^{G^0}(\Omega)$. First we recall the Sylow pro-$p$ subgroup $P_z$ of $SL_2(F)_z$ from Lemma \ref{stabz}.

\begin{lem} \label{phizisoglobal} Suppose that $K$ contains the quadratic unramified extension $L$ of $F$ and let $z \in \Omega_{F,0}(L)$. Then the homomorphism \[ \phi_z[p']\colon \PicCon^{G^0}(\Omega)[p'] \quad \longrightarrow \quad \Hom(G^0_z,K^\times)[p'] \] is an isomorphism. Moreover, every $p'$-torsion character $\chi :G^0_z \to K^\times$ kills $P_z$.
\end{lem}

\begin{proof} By Lemma \ref{PicConGOmega}(a) and Corollary \ref{fibreproduct}, restriction maps induce an isomorphism \[\PicCon^{G^0}(\Omega)[p'] \stackrel{\cong}{\longrightarrow} \invlim\Con^{GL_2(\cO_F)}(\Omega_n)[p'].\] Using Proposition \ref{defphiz}(d) together with Theorem \ref{phizisolocal}, we deduce that map $\phi_z[p']$ in the statement of the Lemma is an isomorphism. 

The last statement holds because $P_z$ is a (normal) pro-$p$ subgroup of $G^0_z$. 
\end{proof}
With the last result in hand, it is natural to wonder about the $p$-torsion part of $\PicCon^{G^0}(\Omega)$. The following description of this group does not require the full force the methods employed in the proof of Theorem \ref{phizisolocal}.

\begin{lem}\label{phizptorsion} Suppose that $K$ contains the quadratic unramified extension $L$ of $F$ and let $z \in \Omega_F(L)$. The homomorphism \[\phi_z[p^{\infty}]\colon \PicCon^{G^0}(\Omega)[p^\infty]\quad \longrightarrow \quad \Hom(G^0_z, K^\times)[p^\infty]\] is injective with image $\Hom(G^0_z/P_z,K^\times)[p^\infty]$. 
\end{lem}
\begin{proof} Since $K \supseteq L$ and $z \in \Omega_F(L)$ by assumption, we see that $z \in \Omega(K)$. Hence the map $\phi_z : \Pic^{G^0}(\Omega) \to \Hom(G^0_z, K^\times)$ exists by Proposition \ref{phiz}. 

Now consider the following triangle:
\[\xymatrix{   \Hom(G^0,K^\times)[p^\infty] \ar[rr]\ar[dr]_{\res} && \PicCon^{G^0}(\Omega)[p^\infty] \ar[dl]^{\phi_z[p^\infty]} \\ &\Hom(G^0_z, K^\times)[p^\infty]&
}\]
Here, the horizontal map sends the character $\chi$ to $[\cO_\chi]$, and the diagonal arrow $\res$ on the left is restriction of characters. The triangle is commutative by Lemma \ref{defphiz}(c), and the horizontal arrow is an isomorphism by
 Proposition \ref{PicSeq} and Corollary \ref{KilledByq+1}. Hence it suffices to show that $\res$ is injective, and that its image is $\Hom(G^0_z/P_z,K^\times)[p^\infty]$.
 
Note that $SL_2(F)$ is a perfect subgroup of $G^0$ and $K^\times$ is abelian. Hence $\chi|_{SL_2(F)} = 1$ for any character $\chi : G^0 \to K^\times$. In particular, $\res(\chi)$ vanishes on the subgroup $P_z$ of $SL_2(F)$. Now if $\chi: G^0 \to K^\times$ is a character such that $\res(\chi) = \chi|_{G^0_z} = 1$, then Corollary \ref{G0zSL2} immediately implies that $\chi = 1$. Therefore $\res$ is injective as required.
 \end{proof}
	 
\begin{prop} \label{imagephiztors} Suppose that $K$ contains the quadratic unramified extension $L$ of $F$. Then for all $g\in GL_2(F)$ and $z\in \Omega_F(L)$, there is a commutative diagram \[ \xymatrix{ \PicCon^{G^0}(\Omega)_{\tors} \ar[rr]^{\phi_z} \ar[d]_{g}&& \Hom(G^0_z/P_z, K^\times)_{\tors} \\ \PicCon^{G^0}(\Omega)_{\tors} \ar[rr]_{\phi_{g\cdot z}} && \Hom(G^0_{g\cdot z}/P_{g\cdot z},K^\times)_{\tors}  \ar[u]_{c_g^\ast}} \] whose arrows are all isomorphisms of abelian groups. 	
\end{prop}

\begin{proof} The diagram commutes by Proposition \ref{GactsPicConeq}(b), and its vertical arrows are isomorphisms with inverses $g^{-1}$ and $c_{g^{-1}}^\ast$ respectively. By Lemma \ref{phizisoglobal} and Lemma \ref{phizptorsion}, the top horizontal arrow is an isomorphism in the case when $z \in \Omega_{F,0}(L)$. But since $L$ is quadratic over $F$, $GL_2(F)$ acts transitively on $\Omega_F(L) = L \backslash F$, so we may choose $g \in GL_2(F)$ such that $g \cdot z \in \Omega_{F,0}(L)$ and then $\phi_{g \cdot z}$ is an isomorphism. The commutativity of the diagram now ensures that $\phi_z$ is always an isomorphism.
\end{proof} 

We can finally give our proof of Theorem A.
\begin{thm}\label{Main} Suppose that $K$ contains the quadratic unramified extension $L$ of $F$. Then there is an isomorphism of abelian groups \[ \PicCon^{G^0}(\Omega)_{\tors} \to \Hom(\cO_D^\times,K^\times)_{\tors}\] that descends to a natural bijection 
\[ \PicCon^{G^0}(\Omega)_{\tors}/G \to \Hom(\cO_D^\times,K^\times)_{\tors}/D^\times.\] 
\end{thm}
\begin{proof} Choose $z \in \Omega_F(L)$ as well as an $F$-algebra homomorphism $\iota : L \hookrightarrow D$. By Lemma \ref{stabz}(d,b,c), $j_z(P_z)$ is the Sylow pro-$p$ subgroup of $\ker N_{L/F} \cap \cO_L^\times$. In view of Definition \ref{DefP1L}, we see that $j_z(P_z) = P^1_L$. Now Proposition \ref{imagephiztors} together with Lemma \ref{stabz}(b) shows that \[j_z\circ \phi_z\colon \PicCon^{G^0}(\Omega)_{\tors} \to \Hom(\cO_L^\times/P^1_L, K^\times)_{\tors}\] is an isomorphism. We can now post-compose  $j_z\circ \phi_z$ with the inverse of the isomorphism $\overline{\varrho \circ \iota}^\ast$ from Corollary \ref{charactersofDandL}(a) to obtain the required isomorphism 
\[ (\overline{\varrho \circ \iota}^\ast)^{-1} \circ j_z \circ \phi_z :  \PicCon^{G^0}(\Omega)_{\tors} \quad\stackrel{\cong}{\longrightarrow} \quad \Hom(\cO_D^\times,K^\times)_{\tors}.\]
Although this does depend on the choice of $z \in \Omega(L)$ as well as the choice of the $F$-algebra embedding $\iota : L \hookrightarrow D$, using Proposition \ref{imagephiztors}, Remark \ref{jzcan} and Corollary \ref{charactersofDandL}(c) we see that it descends to a well-defined bijection
\[ \PicCon^{G^0}(\Omega)_{\tors}/G \quad\stackrel{\cong}{\longrightarrow} \quad \Hom(\cO_D^\times,K^\times)_{\tors}/D^\times\] 
which does not depend on any of the choices.
\end{proof}

\bibliographystyle{plain}
\bibliography{references}
\end{document}